\documentclass[11pt, twoside, reqno]{amsart}
\usepackage[a4paper,left=35mm,right=35mm,top=30mm,bottom=30mm,marginpar=25mm]{geometry} 

\usepackage[utf8]{inputenc}
\usepackage{lmodern}
\usepackage{booktabs}
\usepackage{hyperref} 
\usepackage{xcolor}    
\usepackage{esint}
\usepackage{accents}
\allowdisplaybreaks

\definecolor{darkcyan}{cmyk}{1, 0, 0, 0.6}
\hypersetup{
    colorlinks=true,
    linkcolor=darkcyan,  
    urlcolor=blue,    
    citecolor=darkcyan 
}

\usepackage{amsaddr}
\usepackage{amssymb}
    \usepackage{amsmath}
\usepackage{enumerate}
\usepackage{tikz}
\usepackage{comment}

\usepackage{mathrsfs}
\usepackage{stmaryrd}
\usepackage{enumitem}
\usepackage{dirtytalk}
\usepackage{tikz-cd} 
\newtheorem{thm}{Theorem}[section]
\newtheorem{prop}[thm]{Proposition}
\newtheorem{coro}[thm]{Corollary}
\newtheorem{lemma}[thm]{Lemma}
\theoremstyle{definition}

\newtheorem{defi}[thm]{Definition}
\theoremstyle{remark}
\newtheorem{remark}[thm]{Remark}
\numberwithin{equation}{section}

\newcommand*\dif{\mathop{}\!\mathrm{d}}

\newcommand{\ubar}[1]{\underaccent{\bar}{#1}}

\newcommand{\R}{\mathbb R}

\newcommand{\mc}[1]{\mathcal{#1}}

\newcommand{\mr}[1]{\mathrm{#1}}
\newcommand{\bs}[1]{\boldsymbol{#1}}
\newcommand{\br}[1]{\bs{\mathrm{#1}}}
\newcommand{\ms}[1]{\mathsf{#1}}
\newcommand{\msc}[1]{\mathscr{#1}}

\newcommand{\ul}[1]{\underline{#1}}
\newcommand{\wt}[1]{\widetilde{#1}}
\DeclareMathOperator*{\esssup}{ess\,sup}
\DeclareMathOperator*{\essinf}{ess\,inf}

\usepackage{orcidlink}

\title[Strong traces for $2\times 2$ systems]{Strong traces for solutions of $2\times 2$ nonlinear hyperbolic systems}
\author{Luca Talamini \orcidlink{0009-0008-5147-5184}}
\address{Mathematics Area, SISSA, Trieste}
\email{ltalamin@sissa.it  \textnormal{(Luca Talamini)}}
\begin{document}

\begin{abstract}
  In this paper we show that bounded entropy solutions of genuinely nonlinear $2\times 2$ systems of strictly hyperbolic conservation laws admit a strong trace on Lipschitz curves. The main argument relies on a new half-space Liouville-type theorem for isentropic solutions with constant normal traces, providing the first extension to general genuinely nonlinear systems of the strong trace properties available for scalar conservation laws.
\end{abstract}

\maketitle

\section{Introduction}
In this paper we study trace properties of $\mathbf L^\infty$ entropy solutions to $2\times 2$ systems of conservation laws in one space dimension
\begin{equation}\label{eq:systemi}
\partial_t\,\bs  u(t,x) + \partial_x \, f(\bs  u(t,x)) = 0  \qquad \text{in $\mathscr D^\prime_{t,x}$}.
\end{equation}
Here $\bs u: (0, +\infty) \times \mathbb R \to \mc U \subset  \mathbb R^2$ is a vector valued function, $\mc U\subset \mathbb R^2$ is a bounded open set where $\bs u = (u_1, u_2)$ takes values, and  $f : \mc U \mapsto \mathbb R^2$ is a  twice continuously differentiable flux. 

\vspace{0.2cm}

Developing the notion of strong traces in the setting of $\mathbf L^\infty$ solutions to $2 \times 2$ systems of conservation laws is important for several reasons. First, it allows for strong formulations of initial and initial-boundary value problems (see e.g. the discussion in \cite{Vas01}). Moreover,  it has been recently shown in \cite{CKV22} that the existence of sufficiently strong traces (indeed, even stronger than the ones considered in the present paper) implies uniqueness of $\mathbf L^\infty$ solutions when compared with small $BV$ solutions, where uniqueness is known (see \cite{BDL23} for the small $BV$ case, and \cite{ABM25, BMG25} for a class of small $\mathbf L^\infty$~data).
Note that the only available method to construct weak entropy solutions to genuinely nonlinear systems \eqref{eq:systemi} with $\mathbf L^\infty$ data is the compensated compactness approach \cite{DiP83a, Tar79, Ser00}. However, this procedure is inherently non-constructive, and as a consequence no regularity result
can be derived in the general case for $\mathbf{L}^\infty$~solutions obtained through this method.

\vspace{0.2cm}

Currently, available trace results concern only the scalar case. In particular, the seminal paper \cite{Vas01} established for the first time the existence of strong traces for scalar conservation laws under a non-degeneracy condition on the flux, while \cite{KV07, Pa07} extended this result to general fluxes. The main obstruction 
to establishing these results for  systems of conservation laws
is the intrinsically non-local structure of the kinetic formulations. This fact, already noted e.g. in \cite{LPT94b} in the context of the system of isentropic gas dynamics,  has so far prevented the development of a corresponding trace and regularity theory for systems. The only exception is the case of 
the Euler system with adiabatic coefficient $\gamma = 3$ (see \cite{Gol23}), where other regularizing effects have been identified \cite{AMT26}. \vspace{0.2cm}

In this paper, we prove that if the system is strictly hyperbolic and genuinely nonlinear, then $\mathbf L^\infty$ entropy solutions satisfy the same trace properties known in the scalar case, thus providing the first extension of \cite{Vas01} to systems.  

\vspace{0.2cm}
The main result is obtained through a blow-up argument within the framework of the kinetic formulation. The main difficulty lies in proving that, at almost every point of a Lipschitz curve, the corresponding blow-ups converge to a unique function that is constant on each side of a half-space. This is achieved by establishing a Liouville-type theorem for half-space isentropic solutions with constant normal kinetic trace, which, at the level of the kinetic formulation, plays the role of the normal traces introduced in \cite{CT05}.

\vspace{0.2cm}

In the following, we describe our main results in detail.

\subsection{General setting} Throughout the paper we assume that the system is strictly hyperbolic, that is, for every $\bs u \in \mc U$, the Jacobian matrix $Df(\bs u)$ has two distinct and real eigenvalues $\lambda_1(\bs u) < \lambda_2(\bs u)$, with corresponding eigenvectors $r_1(\bs u), r_2(\bs u)$. We will denote by $\ell_1(\bs u), \ell_2(\bs u)$ the corresponding left eigenvectors, normalized so that 
$$
\ell_i(\bs u) \cdot r_j(\bs u) = \delta_{i,j} \qquad \forall \; \bs u \in \mc U.
$$

We consider the corresponding \textit{Riemann invariants} change of coordinates $\phi = (\phi_1, \phi_2 ): \mc U \to \mathbb R^2$, that is defined, at least locally, by 
\begin{equation}\label{eq:riemdef}
    \nabla \phi_1(\bs u) = \ell_1(\bs u), \qquad \nabla \phi_2(\bs u) = \ell_2(\bs u) \qquad \forall \bs u \in \mc U
\end{equation}
and we let $$\mc W \doteq (\phi_1, \phi_2)(\mc U) \subset \mathbb R^2.$$ We assume throughout the paper that this change of variables is globally smooth and invertible on $\mc U$. Since we consider bounded solutions, here $\mc U$ is assumed to be a rectangle in Riemann invariants, i.e. there holds
\begin{equation}\label{eq:Udef}
\mc W = [\ubar w, \bar w] \times [\ubar z, \bar z], \qquad \mc U = (\phi_1, \phi_2)^{-1}(\mc W), \qquad (w,z) = (\phi_1(\bs u), \phi_2(\bs u)).
\end{equation}
We recall the definition of genuine nonlinear system.

\begin{defi}
The eigenvalue $\lambda_i$ is called \emph{genuine nonlinear} (\textbf{GNL}) if there is a positive constant $c > 0$ such that 
\begin{equation}\label{eq:gnldef}
    \nabla \lambda_i(\bs u) \cdot r_i(\bs u)  > c \qquad \text{for all $\bs u \in \mc U$ and for $i = 1,2$}.
\end{equation}
The system \eqref{eq:systemi} is \textbf{GNL} if both eigenvalues are \textbf{GNL}.
\end{defi}

A pair of Lipschitz functions $\eta, q: \mc U \subset \mathbb R^2 \to \mathbb R$ is called an \textit{entropy-entropy flux pair} for \eqref{eq:systemi} if
    \begin{equation}\label{eq:entropyeq}
        \nabla \eta(\bs u) \cdot \mr D f(\bs u) = \nabla q(\bs u) \qquad \text{for almost every $\bs u \in \mc U$}.
    \end{equation}

In the following $\Omega \subset (0,+\infty) \times \mathbb R$ is an open set.
\begin{defi}[Entropy solution]
    We say that $\bs u \in \mathbf L^\infty(\Omega, \; \mc U)$ is a \emph{weak entropy solution} of \eqref{eq:systemi} if it satisfies \eqref{eq:systemi} and for every convex entropy $\eta : \mc U \to \mathbb R$ there holds
    \begin{equation}\label{eq:entrdis}
        \mu_\eta := \partial_t \eta (\bs u) + \partial_x q(\bs u) \leq 0 \qquad \text{in $\mathscr D^\prime_{t,x}$}.
    \end{equation}
\end{defi}
The entropy $\eta : \mc U \to \mathbb R$ is said to be uniformly convex if there is a positive $C >0$ such that $\nabla^2 \eta \geq C >0$ in $\mc U$.
\begin{remark}
    It could be that the system does not admit any convex entropy. However, for most physical systems and for $2\times 2$ systems satisfying mild conditions on $f$, uniformly convex entropies always exist (see e.g. \cite[Chapter 12]{Daf16}).
\end{remark}

Before stating our main result, we first distinguish two concepts of strong traces on Lipschitz curves.
\begin{defi}[Pointwise traces on Lipschitz curves]\label{defi:strongt}
    Let $\bs u \in \mathbf L^\infty(\Omega, \; \mc U)$. We say that $\bs u$ admits \emph{strong traces on Lipschitz curves} if, for every $\gamma: [0,T] \to \mathbb R$ Lipschitz, there are two Borel functions $\bs u^-, \bs u^+ : [0,T] \to \mc U$, left and right pointwise traces of $\bs u$ on $\gamma$, such that for $\mathscr L^1$-a.e. $t$,
    \begin{equation}\label{eq:blowuptrace}
      \lim_{r \to 0^+} \frac{1}{r^2} \Bigg(\int_{B_r^-(t, \gamma(t))} |\bs u(s, y) - \bs u^-(t)|  \dif y \dif s +\int_{B_r^+(t, \gamma(t))} |\bs u(s, y) - \bs u^+(t)|  \dif y \dif s \Bigg) = 0
    \end{equation}
    where we set for a.e. $t >0$ 
    \begin{equation}\label{eq:halfball}
    B_r^{\pm}(t,\gamma(t)) := \Big\{ (s,y) \in B_r(t, \gamma(t))\; \big| \; y \gtrless \gamma(s)  \Big\}
        \end{equation}
    $B_r(t,\gamma(t))$ is a ball of radius $r$ centered at $(t,\gamma(t))$, and $\vec n(t)$ is the normal to the curve $\gamma$ at time $t$:
    $$
    \vec n(t) =\frac{(-\dot \gamma(t), 1)}{\sqrt{1+{\dot \gamma}^2}}
    $$
\end{defi}
A stronger notion of trace will also be used.

\begin{defi}[$\mathbf L^1$ traces on Lipschitz curves]\label{defi:strongL1t}
    Let $\bs u \in \mathbf L^\infty(\Omega, \; \mc U)$. We say that $\bs u$ admits \emph{$\mathbf L^1$ traces on Lipschitz curves} if, for every $\gamma: [0,T] \to \mathbb R$ Lipschitz, there are two Borel functions $\bs u^-, \bs u^+ : [0,T] \to \mc U$, left and right $\mathbf L^1$-traces of $\bs u$ on $\gamma$, such that  
    \begin{equation}
          \lim_{\delta \to 0^+} \esssup_{y \in (0, \delta)} \Bigg( \int_{0}^{T} \big| \bs u(t, \gamma(t) + y)- \bs u^{+}(t) \big| \dif t  + \int_{0}^{T} \big| \bs u(t, \gamma(t) - y)- \bs u^{-}(t) \big| \dif t \Bigg)  = 0
      \end{equation}
\end{defi}

The first main results of the paper asserts the existence of pointwise strong traces for solutions to general \textbf{GNL} systems. 

\begin{thm}\label{thm:main}
   Assume that the system \eqref{eq:systemi} is strictly hyperbolic and \textbf{GNL}, and that there exists at least one uniformly convex entropy. Let $\bs u \in \mathbf L^\infty(\Omega, \; \mc U)$ be an entropy solution of \eqref{eq:systemi}. Then $\bs u$ admits pointwise traces on Lipschitz curves in the sense of Definition \ref{defi:strongt}.
\end{thm}

Theorem \ref{thm:main} is the consequence of a more general result, valid for more general \emph{finite entropy solutions}.

\begin{defi}\label{defi:fes}
We say that a weak solution of \eqref{eq:systemi} $\bs u$ is a \emph{finite entropy solution} to \eqref{eq:systemi} if \begin{equation}\label{eq:feseta}
\mu_{\eta} =\eta(\bs u)_t + q(\bs u)_x \in \msc M(\Omega)
\end{equation}for every entropy-entropy flux pair of $C^2$ function $(\eta, q): \mc U \subset \mathbb R^2 \to \mathbb R$, where $\msc M(\Omega)$ is the space of Radon measures on $\Omega$.
\end{defi}

We also note in passing that the more general notion of finite entropy solution plays an important role in the theory of systems of conservation laws, particularly in connection with physically relevant limits. For instance, viscous approximations of the Euler system are known to converge, a priori, only to finite entropy solutions, which additionally dissipate the convex entropy associated with the energy. We refer to \cite{CP10} for further details on this topic.
Since no additional difficulties arise in proving the main results in this more general framework, it is natural to formulate them in this setting. \vspace{0.2cm}

\begin{thm}\label{thm:mainfes}
   Assume that the system \eqref{eq:systemi} is strictly hyperbolic and \textbf{GNL}. Let $\bs u \in \mathbf L^\infty(\Omega, \; \mc U)$ be a finite entropy solution of \eqref{eq:systemi}. Then $\bs u$ admits pointwise traces on Lipschitz curves in the sense of Definition \ref{defi:strongt}.
\end{thm}

We recall that entropy solutions are also finite entropy solutions under the assumption of the existence of a uniformly convex entropy, see e.g. \cite[Proposition 2.3]{Tal26}. Therefore one immediately sees how Theorem \ref{thm:main} is a consequence of Theorem \ref{thm:mainfes}. Notice also that in Theorem \ref{thm:mainfes} the existence of the uniformly convex entropy is not required, since this assumption is only used in Theorem \ref{thm:main} to deduce that entropy solutions are in particular also finite entropy solutions. 

\vspace{0.2cm}
The following theorem shows that the traces are attained in the stronger sense of Definition \ref{defi:strongL1t} if the system \eqref{eq:systemi} is \emph{uniformly strictly hyperbolic},
that is if the eigenvalues $\lambda_1, \lambda_2$ satisfy the following assumption:
\begin{equation}\label{eq:globhyp}
\sup_{\bs u \in \mc U}\; \lambda_1 (\bs u)  < \inf_{\bs u\in \mc U} \; \lambda_2(\bs u).    
\end{equation}

\begin{thm}\label{thm:main1}
      Assume that the system \eqref{eq:systemi} is uniformly strictly hyperbolic and \textbf{GNL}.
      Let $\bs u \in \mathbf L^\infty(\Omega, \; \mc U)$ be an  entropy solution of \eqref{eq:systemi} and assume that there exists a uniformly convex entropy. Then $\bs u$ admits $\mathbf L^1$ traces on Lipschitz curves in the sense of Definition \ref{defi:strongL1t}.
\end{thm}
Again the Theorem can be proved for more general finite entropy solutions.

\begin{thm}\label{thm:mainfes1}
   Assume that the system \eqref{eq:systemi} is uniformly strictly hyperbolic and \textbf{GNL}.
   Let $\bs u \in \mathbf L^\infty(\Omega, \; \mc U)$ be a finite entropy solution of \eqref{eq:systemi}. Then $\bs u$ admits $\mathbf L^1$ traces on Lipschitz curves in the sense of Definition \ref{defi:strongL1t}.
\end{thm}

\subsection{Structure of paper}
The paper is structured as follows. 

\vspace{0.5cm}
In Section \ref{sec:entropies} we introduce some preliminaries, in particular a description of the structure of entropies as presented in \cite{PT00}.

\vspace{0.2cm}
In Section \ref{sec:weaktrace} we discuss weak normal trace properties of finite entropy solutions. This topic is well studied, see e.g. \cite{CT05}, and for the sake of completeness here we reformulate and prove these results adapted to best fit our setting.

\vspace{0.2cm}
In Section \ref{sec:hamlag} we prove the existence of Lagrangian representations (i.e. characteristics) for the kinetic equations satisfied by an isentropic solution, i.e. a weak solution to \eqref{eq:systemi} with vanishing entropy production. In order to make the paper self contained and to easily derive additional useful properties, we derive from scratch the existence of the Lagrangian representation using a method based on level sets of Lipschitz Hamiltonians. A similar representation was used in \cite{AMT25}; however, in the present setting an additional difficulty arises due to the presence of the boundary.

\vspace{0.2cm}
Section \ref{sec:HSL} is the core of the paper, where we prove a half-space Liouville-type theorem for isentropic solutions in a half-space of the form $\{(t,x)\, | \;  x \gtrless \alpha\, t\}$, under the assumption of a constant weak normal trace on the boundary hyperplane, Theorem \ref{thm:Liuv1}. In particular, we show that such solutions must be constant, and that the constant is uniquely determined by the weak normal trace. This will yield a proof of Theorems \ref{thm:main}, \ref{thm:mainfes}. A Liouville-type theorem in $\mathbb{R}^2$ was obtained in \cite{AMT25}. In the present work, the situation is more involved: we must prove a corresponding result in a half-space and, moreover, show that the constant solution is uniquely determined by the kinetic normal trace, a key ingredient in the proof of the trace result, Theorem~\ref{thm:main}.

\vspace{0.2cm}
In Section \ref{sec:PtoL1} we show how Theorems \ref{thm:main1}, \ref{thm:mainfes1} follow from the results of Section \ref{sec:HSL}, in particular from Theorems \ref{thm:main}, \ref{thm:mainfes}. The argument relies on soft analytic tools and on a reduction of a certain Young measure to a Dirac mass. It is only at this point that the stronger assumption \eqref{eq:globhyp} comes into play.

\section{Preliminaries and entropies of \texorpdfstring{$2\times 2$}{2x2} systems} \label{sec:entropies}

We begin by recalling some structural results for the entropies of $2\times 2$ systems mainly from \cite{PT00}.
It is well known that systems of $2\times 2$ conservation laws admit infinite families of entropies. Aim of this section is to describe how all smooth entropies can be obtained as a superposition of \say{elementary} irregular entropies. These elementary entropies are of two types (associated respectively to the first and second Riemann invariants $w, z$) and solve the entropy equation 
\begin{equation}\label{e:wentropy}
\nabla q(\bs u) \, - \, \mr D f (\bs u) \, \nabla \eta(\bs u) = 0 \qquad \text{in $\msc D^\prime (\mc U)$}.   
\end{equation}
only distributionally, since they are discontinuous along lines $\{w = \xi\}$ and $\{z = \zeta\}$ respectively, where $\xi \in (\underline w, \overline w)$, $\zeta \in (\underline z, \overline z)$ are fixed. We recall that a standard computation shows (see e.g. \cite{Ser00}) that a smooth function $\eta : \mc U \to \mathbb R$ is an entropy of the system (i.e. there exists $q$ for which \eqref{e:wentropy} holds) if and only if it satisfies, in the Riemann invariant coordinates,
$$
\eta_{wz} = \frac{g_z}{g}\eta_w + \frac{h_w}{h} \eta_z \qquad \text{in $\mc  W$}
$$
where 
\begin{equation}\label{eq:ghdef}
\begin{aligned}
g(w, z) & = \exp\left[\int_{{\ubar z}}^{ z} -\frac{\lambda_{1z}(w, y)}{\lambda_1(w, y)-\lambda_2(w, y)}\dif y  \right] \\
h(w, z) & = \exp\left[\int_{{\ubar w}}^{w} \frac{\lambda_{2w}(y, z)}{\lambda_1(y, z)-\lambda_2(y, z)}\dif y  \right]
\end{aligned}
\end{equation}
where $\lambda_i$ are the eigenvalues of $Df$, as functions of the Riemann invariants $w,z$.
In the following we describe the construction of discontinuous entropies related to $w$, the case of discontinuous entropies related to $z$ is be entirely symmetric.
Let $g(w,z)$ be defined by

\begin{defi}
Let $\xi \in (\underline w, \overline w)$ be fixed. We denote by $\bs \Theta[\xi]: \mc U \to \mathbb R$ the entropy constructed as the unique solution to the Goursat-boundary value problem 
\begin{equation}\label{eq:goursat}
\begin{cases}
\bs \Theta[\xi]_{wz} = \frac{g_z}{g}\bs \Theta[\xi]_w + \frac{h_w}{h} \bs \Theta[\xi]_z, & \text{in} \; \mc W\\
\\
\bs \Theta[\xi](w, {\ubar z}) = 1, & \forall \; w \in [{\ubar w}, \bar w] \\
\\
\bs \Theta[\xi](\xi, z) = g(\xi, z) & \forall \; z \in [{\ubar z}, \bar z].
\end{cases}
\end{equation}
\end{defi}
The existence of a unique, smooth solution $\bs \Theta[\xi]$ to the above boundary value problem is standard; for a proof of this fact, see e.g. \cite[Section 9.3]{Ser00}. The solution is as smooth as the initial data and the coefficients, and moreover 
$$
w, z, \xi \mapsto \bs \Theta[\xi](w,z)
$$
is also smooth as a function of three variables $w, z, \xi$.

\vspace{0.5cm}
We label $\bs \Xi[\xi]$ the entropy flux of $\bs \Theta[\xi]$.
Up to an additive constant in the entropy flux we can assume that (see \cite{AMT25} or \cite{PT00})
\begin{equation}\label{eq:efluxxi}
\bs \Xi[\xi](\xi, z) = \lambda_1(\xi, z) \bs \Theta[\xi](\xi, z) \qquad \forall \; z \in [{\ubar z}, \bar z].
\end{equation}
We finally define the discontinuous entropies related to the hypograph/epigraph:
\begin{equation}\label{e:chir}
\bs \chi[\xi]  \doteq \bs \Theta[\xi]\cdot  \mathbf 1_{\{w \geq \xi\}} , \qquad \wt{\bs \chi}[\xi]  \doteq \bs \Theta[\xi]\cdot  \mathbf 1_{\{w \leq \xi\}} 
\end{equation}
and entropy fluxes
\begin{equation}\label{eq:fluxes}
\bs \psi[\xi] \doteq \bs \Xi[\xi]\cdot  \mathbf 1_{\{w \geq \xi\}} \qquad \wt{\bs \psi}[\xi] \doteq \bs \Xi[\xi]\cdot  \mathbf 1_{\{w \leq \xi\}}.
\end{equation}
The specific choice of the value of $\bs \Theta[\xi]$ at the line $\{ w= \xi\}$  in \eqref{eq:goursat} makes it possible to have the following:
\begin{thm}[\cite{PT00}]
    The functions $\bs \chi[\xi], \bs \psi[\xi]$ and $\wt{\bs \chi}[\xi], \wt{\bs \psi}[\xi]$ are entropy-entropy flux pairs, solving \eqref{e:wentropy} in the sense of distributions.
\end{thm}

A symmetric construction can be made for the entropies that can be cut along the second Riemann invariant; for these entropies, for $\zeta \in [{\ubar z}, \bar z]$, we let $\bs \upsilon[\zeta](w,z)$ be entropy corresponding to $\bs \chi[\xi](w,z)$, and $\bs \varphi[\zeta](w,z)$  the respective entropy flux, corresponding to $\bs \psi[\xi](w,z)$, and similarly for $\wt{\bs \upsilon}[\zeta]$, $\wt{\bs \varphi}[\zeta]$.

\begin{thm}[\cite{PT00}]
Let $\eta, q$ be a $C^k$ entropy-entropy flux pair expressed in Riemann coordinates with $\eta(\ubar w, \ubar z) = q(\ubar w, \ubar z) = 0$.  Then there holds
     \begin{equation}\label{eq:entsup}
    \begin{aligned}
\eta(w,z) & \doteq \int_{\ubar w}^{\bar w} \bs \chi[\xi] (w,z)\varrho_1(\xi) \dif \xi  +  \int_{\ubar z}^{\bar z} \bs \upsilon[\zeta](w,z) \varrho_2(\zeta) \dif \zeta \\
q(w,z)  & \doteq \int_{\ubar w}^{\bar w} \bs \psi[\xi] (w,z)\varrho_1(\xi) \dif \xi  +  \int_{\ubar z}^{\bar z} \bs \varphi[\zeta](w,z) \varrho_2(\zeta) \dif \zeta
    \end{aligned}
    \end{equation}
where
\begin{equation}\label{eq:formularho}
\varrho_1(\xi)  \doteq \eta_w(\xi, \ubar z), \qquad \varrho_2(\zeta) \doteq \eta_z(\ubar w, \zeta)  \in C^{k-1}.
\end{equation}
Conversely, if $\varrho_1, \varrho_2 \in C^{k-1}$ and $\bs \Theta[\cdot](\cdot) \in C^k$, then the functions $\eta, q$ defined through \eqref{eq:entsup} are $C^{k}$ entropy-entropy flux pairs of the system \eqref{eq:systemi}.
\end{thm}

\begin{remark}
   An analogous result also holds for Lipschitz entropy--entropy flux pairs $(\eta, q)$. Since the restriction of a Lipschitz function to a line $\{z = \bar z\}$ or $\{w = \bar w\}$ is differentiable almost everywhere along that line, the formulas \eqref{eq:formularho} are still well-defined in this setting. In this case, the representation formulas \eqref{eq:entsup} remain valid with $\varrho_1, \varrho_2 \in \mathbf{L}^\infty$.
   Conversely, if $\varrho_i \in \mathbf L^\infty$, then the functions $\eta, q$ defined through \eqref{eq:entsup} are Lipschitz entropy-entropy flux pairs.
\end{remark}

We conclude with a simple result, see e.g. \cite{AMT25}.

\begin{prop}\label{prop:localspeed}
    There exists positive $\bar {r}, c > 0$  such that, for every $\xi, w \in [\ubar w, \bar w]$  and $z \in [\ubar z, \bar z]$ such that $\xi \leq w \leq \xi + \bar r$, the following holds:
\begin{enumerate}
    \item Local strict positivity of the entropies: 
    $$
    \bs \chi[\xi](w, z) \geq c > 0
    $$
    \item If $\lambda_{1}$ is genuinely nonlinear, then we have the monotonicity of the kinetic speed:
$$
\frac{\dif}{\dif \xi} \bs \lambda_1[\xi](w, z) \geq c > 0 
$$
where
\begin{equation}\label{eq:lambdadef}
\bs \lambda_1[\xi](w, z) \doteq \frac{\bs \psi[\xi](w, z)}{\bs \chi[\xi](w, z)}.
\end{equation}
\end{enumerate}
A completely symmetric statement holds for the entropies $\wt{\bs \chi}$ and for  weak entropies related to the second Riemann invariant $z$, $\bs \upsilon[\zeta], \wt{\bs \upsilon}[\zeta]$.
\end{prop}

We remark that thanks to \eqref{eq:efluxxi}, (and the analogous for the epigraph entropies $\bs \upsilon, \bs \varphi$, there holds
\begin{equation}\label{eq:coincideonline}
    \bs \lambda_1[\xi](\xi, z) = \lambda_1(\xi, z), \qquad  \bs \lambda_2[\zeta](w, \zeta) = \lambda_2(w, \zeta) \qquad \forall \; \xi, z, w, \zeta.
\end{equation}
Moreover, by definition, there holds
\begin{equation}\label{eq:chiupsline}
\bs \chi[\xi](\xi+,z) = g(\xi, z), \qquad \bs \upsilon[\zeta](w, \zeta+) = h(w, \zeta)
\end{equation}
where $g, h$ are defined in \eqref{eq:ghdef}.
As mentioned in the introduction, the main difficulty in the case of systems is the non-locality of the kinetic speed $\boldsymbol{\lambda}_1[\xi](w,z)$. In contrast with the scalar case, this quantity depends on the full state $(w,z)$.

\begin{remark}
    Throughout the paper we assume that $\mc U$ is of the form \eqref{eq:Udef}, because it makes it easy to construct the discontinuous entropies, however as it will be clear, all the results of this paper remain true in more general domains as long as one can make this construction.
\end{remark}

\section{Kinetic normal traces}\label{sec:weaktrace}
The existence of a weak normal trace for divergence measure fields is a well studied topic, see e.g. \cite{CT05}, in the context of measure divergence vector fields. The clear connection with finite entropy solutions is that $X = (\eta(\bs u), q(\bs u))$ is a measure divergence field for every smooth entropy-entropy flux pair $\eta, q$. Here we state a version in our setting, formulated in terms of the entropies $\bs \chi[\xi], \wt {\bs \chi}[\xi]$ of the previous section, which we prove for completeness. The kinetic normal traces $\bs \Phi(\xi), \wt{\bs \Phi}(\xi)$ represent the boundary fluxes of the vector fields $X_\xi:= (\bs \chi[\xi], \bs \psi[\xi])$ and $\wt X_\xi:= (\wt{\bs \chi}[\xi], \wt{\bs \psi}[\xi])$. Notice however that these are not, in general measure divergence vector fields.  Clearly, a symmetric result will hold for the entropies $\bs \upsilon[\zeta], \wt{\bs \upsilon}[\zeta]$ related to the second Riemann invariant $z$.

\begin{prop}[Kinetic normal traces] \label{prop:weaktraces}
  Let $\bs u \in \mathbf L^\infty(\Omega, \mc U)$ be a finite entropy solution to \eqref{eq:systemi}. Consider any Lipschitz curve $\gamma: [0,+\infty) \to \mathbb R$, and define 
  $$
  \Omega^{\pm} := \Big\{(t,x) \in \Omega \; | \; x \lessgtr \gamma(t) \Big\}
  $$
  Then there exists measurable functions $\bs \Phi^{\pm} : \mathrm{Graph}\, (\gamma) \times (\ul w, \overline w) \to \mathbb R$, that we call left and right kinetic normal trace on $\gamma$, related to the hypograph, such that for every $\varphi \in C^1_c(\Omega)$ and $\varrho \in \mathbf L^\infty(\underline w, \overline w)$ there holds
   \begin{equation}\label{eq:kinetictrace}
   \begin{aligned}
       \iint_{\Omega^{\pm}} \varphi(t,x)   \dif \mu_{\eta_\varrho}& = 
  - \iint_{\Omega^{\pm}} \Bigg( \int_{\ubar w}^{\bar w} \varrho(\xi) (\bs \chi[\xi](\bs u), \bs \psi[\xi](\bs u)) \cdot \nabla \varphi\dif \xi \Bigg) \dif x \dif t   \\
    & + \int_{\mathrm{Graph}\, (\gamma)} \varphi(t,x)\Big(  \int_{\ubar w}^{\bar w}   \varrho(\xi)  \bs \Phi^{\pm}(t,x,\xi)  \dif \xi \Big) \dif \mathscr H^1(t,x)
     \end{aligned}
      \end{equation}
      where $\eta_\varrho$ is the entropy defined by 
      \begin{equation}\label{eq:etarhodef}
      \eta_{\rho}(\bs u) := \int_{\ubar w}^{\bar w} \varrho(\xi) \bs \chi[\xi](\bs u) \dif \xi, \qquad \forall \; \bs u \in \mc U
       \end{equation}
      with corresponding entropy flux 
      \begin{equation}\label{eq:qrhodef}
      q_{\rho}(\bs u) := \int_{\ubar w}^{\bar w} \varrho(\xi) \bs \psi[\xi](\bs u) \dif \xi, \qquad \forall \; \bs u \in \mc U
       \end{equation}
\end{prop}
Symmetric results hold for the entropies ${\bs \upsilon}$, and we denote by ${\bs \Psi}^{\pm}$ the corresponding kinetic normal traces.
\begin{proof}
    We prove the statement only in $\Omega^+$.    Let $\varphi \in C^1_c(\Omega)$ and consider a cutoff function $h^\varepsilon$ satisfying
    \begin{equation}\label{eq:cutoff1}
    h^\varepsilon(t,x) := \begin{cases}
        1 & \text{if $x > \gamma(t) + \varepsilon$}, \\
        0 & \text{if $x \leq \gamma(t)$},
    \end{cases}
        \end{equation}
    \begin{equation}\label{eq:cutoff2}
    h^\varepsilon \leq 1, \qquad |\nabla h^\varepsilon| \leq \frac{C}{\varepsilon} \quad \text{everywhere,} \quad \partial_x h^\varepsilon \geq 0.
     \end{equation}
     Consider the Lipschitz entropy $\eta_\varrho$ as in \eqref{eq:etarhodef}, with corresponding entropy flux 
     \begin{equation}\label{eq:Fetarhodef}
      q_{\rho}(\bs u) := \int_{\ul w}^{\overline w} \varrho(\xi) \bs \psi[\xi](\bs u) \dif \xi, \qquad \forall \; \bs u \in \mc U.
       \end{equation}
  Now we test $\mu_\eta$ in \eqref{eq:feseta} with $\psi^\varepsilon = h^\varepsilon \varphi$ and use the representation formulas \eqref{eq:etarhodef}, \eqref{eq:Fetarhodef} to obtain 
    \begin{equation}
\begin{aligned}
    \iint_\Omega \psi^\varepsilon \dif \mu_{\eta_\varrho}
    &= -\iint_\Omega
        \begin{pmatrix}
            \eta_\varrho(\bs u) \\[4pt]
            q_\varrho(\bs u)
        \end{pmatrix}
        \cdot \nabla \psi^\varepsilon \dif x \dif t
    \\[10pt]
    &= -\iint_{\Omega} h^\varepsilon
        \begin{pmatrix}
            \eta_\varrho(\bs u) \\[4pt]
            q_\varrho(\bs u)
        \end{pmatrix}
        \cdot \nabla \varphi \dif x \dif t
        - \iint_{\Omega} \varphi
        \begin{pmatrix}
            \eta_\varrho(\bs u) \\[4pt]
            q_\varrho(\bs u)
        \end{pmatrix}
        \cdot \nabla h^\varepsilon \dif x \dif t
    \\[10pt]
    &= -\iint_\Omega h^\varepsilon
        \left(
            \int \varrho(\xi)\,
            \bigl(\bs\chi[\xi](\bs u),\, \bs\psi[\xi](\bs u)\bigr)
            \cdot \nabla \varphi \dif \xi
        \right) \dif x \dif t
    \\[6pt]
    &\quad
        - \iint_{\Omega} \varphi
        \left(
            \int \varrho(\xi)\,
            \bigl(\bs\chi[\xi](\bs u),\, \bs\psi[\xi](\bs u)\bigr)
            \cdot \nabla h^\varepsilon \dif \xi
        \right) \dif x \dif t
\end{aligned}
\end{equation}
Passing to the limit for $\varepsilon \to 0^+$ we deduce 
\begin{equation}\label{eq:tracelimit}
\begin{aligned}
    \iint_{\Omega^+} \varphi \dif \mu_{\eta_\varrho}
    &+ \iint_{\Omega^+}
        \left(
            \int \varrho(\xi)\,
            \bigl(\bs\chi[\xi](\bs u),\, \bs\psi[\xi](\bs u)\bigr)
            \cdot \nabla \varphi \dif \xi
        \right) \dif x \dif t
    \\[10pt]
    &= \lim_{\varepsilon \to 0^+}
        \iint_{\Omega}
        \left(
            \int \varphi(t,x)\, \varrho(\xi)\,
            \bigl(\bs\chi[\xi](\bs u),\, \bs\psi[\xi](\bs u)\bigr)
            \cdot \nabla h^\varepsilon \dif \xi
        \right) \dif x \dif t
\end{aligned}
\end{equation}
where we used the fact that the limits in the left hand side exist to deduce that the limit in the right hand side exists as well.
Notice that 
\begin{equation}
    \alpha^\varepsilon(t, x, \xi) 
    :=  
    \bigl(\bs\chi[\xi](\bs u(t,x)),\, \bs\psi[\xi](\bs u(t,x))\bigr) 
    \cdot \nabla h^\varepsilon(t,x)
\end{equation}
is locally uniformly  bounded in $\mathbf L^1$ since
\begin{equation}
\begin{aligned}
    &\int_{t_1}^{t_2} \int_{\mathbb{R}} \int_{\underline{w}}^{\overline{w}}
        \lvert \alpha^\varepsilon(t,x,\xi) \rvert \dif \xi \dif x \dif t
    \\[10pt]
    & \leq \frac{C}{\varepsilon}
        \left( \sup_{\bs u} \sqrt{\bs\chi^2(\bs u) + \bs\psi^2(\bs u)} \right)
        (t_2 - t_1)(\overline{w} - \underline{w})\, \varepsilon
    \\[6pt]
    & = C
        \left( \sup_{\bs u} \sqrt{\bs\chi^2(\bs u) + \bs\psi^2(\bs u)} \right)
        (t_2 - t_1)(\overline{w} - \underline{w})
\end{aligned}
\end{equation}
which is a bound independent of $\varepsilon$. Therefore up to a subsequence $\varepsilon_j$ we have $\alpha^{\varepsilon_j} \to \alpha$ weakly as measures, where 
$$
\alpha \in \mathscr M\Big(\mathbb R^+ \times \mathbb R \times (\underline w, \overline w) \Big)
$$
is a measure supported on $\mathrm{Graph}\, (\gamma) \times (\underline w, \overline w)$. We observe that $\alpha$ is absolutely continuous with respect to the $\mathscr H^2$ measure on the surface $\mathrm{Graph}\, (\gamma)\times (\underline w, \overline w)$, i.e. we can write $$\alpha =  \bs \Phi^+ \cdot \mathscr H^2\llcorner \mathrm{Graph}\, (\gamma)\times (\underline w, \overline w)$$
for a measurable $\mathbf L^\infty$ function $$\bs \Phi^+ :\mathrm{Graph}\, (\gamma) \times (\underline w, \overline w) \to \mathbb R.$$
In fact, $\alpha^\varepsilon$ writes as $g(t,x,\xi) \cdot \mathscr L^1 \otimes \beta^\varepsilon$, where $\beta^\varepsilon$ converges in measure to $\beta = \mathscr H^1 \llcorner \mr{Graph} \, (\gamma)$, for some $\mathbf L^\infty$ function $g$.

This shows, using the sequence $\{\varepsilon_j\}$ to pass in the limit in \eqref{eq:tracelimit},
\begin{equation}
\begin{aligned}
    \iint_{\Omega^+} \varphi \dif \mu_\eta
    &+ \iint_{\Omega^+}
        \left(
            \int \varrho(\xi)\,
            \bigl(\bs\chi[\xi](\bs u),\, \bs\psi[\xi](\bs u)\bigr)
            \cdot \nabla\varphi \dif \xi
        \right) \dif x \dif t
    \\[10pt]
    &= \int_{\mathrm{Graph}(\gamma)} \varphi(t,x)
        \left(
            \int \varrho(\xi)\, \bs\Phi^{+}(t,x,\xi) \dif \xi
        \right)
        \dif \mathscr{H}^1(t,x)
\end{aligned}
\end{equation}
which is exactly \eqref{eq:kinetictrace} for the positive domain $\Omega^+$. The proof in $\Omega^-$ is entirely symmetric and therefore it is omitted.
\end{proof}

As it happens for divergence measure fields \cite{CT05}, the balance \eqref{eq:kinetictrace} can be used to show that the functions $h \mapsto \big(\bs \chi[\cdot](\bs u(\cdot, \gamma(\cdot)+h)), \bs \psi[\cdot](\bs u(\cdot, \gamma(\cdot)+ h)) \big) \in \mathbf L^\infty(\mathbb R^+_t \times (\underline w, \overline w))\cdot \vec n$ admit weak$^\star$-limits in $\mathbf L^\infty(\mathbb R^+_t \times (\underline w, \overline w))$.
\begin{prop}\label{prop:normallimit}
    In the setting of Proposition \ref{prop:weaktraces}, the following holds. There exists a full set $\mc F \subset \mathbb R$ such that 
    \begin{equation}\label{eq:normallimit}
    \begin{aligned}
         \big(\bs \chi[\cdot](\bs u(\cdot, \gamma(\cdot)+h)),&  \bs \psi[\cdot](\bs u(\cdot, \gamma(\cdot)+ h)) \big)\cdot \vec n \longrightarrow \bs \Phi^{\pm}(\cdot, \gamma(\cdot), \cdot) \\
         & \text{weakly$^\star$ in $\mathbf L^\infty$ as $h \to 0^{\pm}$, $h \in \mc F$}.
         \end{aligned}
    \end{equation}
\end{prop}
\begin{proof}
    Consider any $p \in  C^\infty_c(\underline w, \overline w)$, and $\varphi$ of the form $\varphi(t,x) := \psi(x-\gamma(t))g(t)$, where $\psi \in  C^\infty_c(\mathbb R)$ and $g \in C^\infty_c(\mathbb R^+)$. Then using Proposition \ref{prop:weaktraces} we deduce
    \begin{equation}\label{eq:transcurve}
        \begin{aligned}
            \iint_{\Omega^{\pm}}  & \Bigg(\int_{\underline w}^{\overline w} \varrho(\xi) \psi^\prime(x-\gamma(t)) g(t) (\bs \chi[\xi](\bs u), \bs \psi[\xi](\bs u)) \cdot (-\dot \gamma, 1) \Bigg) \dif \xi \dif x \dif t & \\
            & = \int_{\mathrm{Graph}(\gamma)} \varphi(t,x)
        \left(
            \int \varrho(\xi)\, \bs\Phi^{+}(t,x,\xi) \dif \xi
        \right)
        \dif \mathscr{H}^1(t,x) \\
        & - \iint_{\Omega^{\pm}} \Bigg( \int \varrho(\xi) g^\prime(t) \psi(x-\gamma(t)) \bs \chi[\xi](\bs u) \dif \xi  \Bigg) \dif x \dif t\\
        & - \iint_{\Omega^{\pm}} \varphi \dif \mu_\eta 
        \end{aligned}
    \end{equation}
    Next, for every $h > 0$, we consider 
    $$
    \psi^\varepsilon_h(y) := \begin{cases}
        0 & \text{if $y \geq h+\varepsilon$},\\
        \frac{1}{\varepsilon} (h+\varepsilon-x) & \text{if $y \in (h, h+\varepsilon)$}\\
         1 &\text{if $y \leq h $}.
    \end{cases}
    $$
    For a full measure set of $h \in \mathbb R$, using $\varphi^\varepsilon_h$ and passing to the limit as $\varepsilon \to 0$ in \eqref{eq:transcurve} yields
    \begin{equation}
    \begin{aligned}
         -\int \int_{\underline w}^{\overline w} & \varrho(\xi) g(t) (\bs \chi[\xi](\bs u(t, \gamma(t)+ h)), \bs \psi[\xi](\bs u(t, \gamma(t)+ h))) \cdot (-\dot \gamma, 1)  \dif \xi \dif t \\
        & = \int \int g(t) \varrho(\xi) \sqrt{1+\dot \gamma^2}\bs \Phi^+(t, \gamma(t), \xi) \dif t \dif \xi + o(1)_h  
    \end{aligned}
    \end{equation}
    where $o(1)_h$ is an error term that goes to zero as $h$ approaches zero. This proves the claim.
\end{proof}

The same holds for the entropies related to the epigraph, $\wt{\bs\chi}$. For later reference we state (without proving) the symmetric result.
\begin{prop}\label{prop:weaktracestilda}
    Let $\Omega \subset \mathbb{R}^2$ be an open set and let $\bs u \in \mathbf{L}^\infty(\Omega, \mc{U})$
    be a finite entropy solution to \eqref{eq:systemi}.
    Consider any Lipschitz curve $\gamma \colon [0,+\infty) \to \mathbb{R}$.

    Then there exist measurable functions
    $\wt{\bs\Phi}^{\pm} \colon \mathrm{Graph}(\gamma) \times (\underline{w}, \overline{w}) \to \mathbb{R}$
    such that for every $\varphi \in C^1_c(\Omega)$ and $\varrho \in \mathbf{L}^\infty(\underline{w}, \overline{w})$
    there holds
    \begin{equation}\label{eq:kinetictracetilda}
    \begin{aligned}
        \iint_{\Omega^{\pm}} \varphi(t,x) \dif \mu_{\wt\eta_\varrho}
        &= -\iint_{\Omega^{\pm}}
            \left(
                \int \varrho(\xi)\,
                \bigl(\wt{\bs\chi}[\xi](\bs u),\, \wt{\bs\psi}[\xi](\bs u)\bigr)
                \cdot \nabla\varphi \dif \xi
            \right) \dif x \dif t
        \\[10pt]
        &\quad
            + \int_{\mathrm{Graph}(\gamma)} \varphi(t,x)
            \left(
                \int \varrho(\xi)\, \wt{\bs\Phi}^{\pm}(t,x,\xi) \dif \xi
            \right)
            \dif \mathscr{H}^1(t,x)
    \end{aligned}
    \end{equation}
    where $\wt{\eta}_\varrho$ is the entropy defined by
    \begin{equation}\label{eq:etarhodeftilda}
        \wt{\eta}_\varrho(\bs u)
        := \int_{\underline{w}}^{\overline{w}} \varrho(\xi)\, \wt{\bs\chi}[\xi](\bs u) \dif \xi,
        \qquad \forall\; \bs u \in \mc{U}
    \end{equation}
    with entropy flux
    \begin{equation}\label{eq:qrhodeftilda}
        \wt{q}_\varrho(\bs u)
        := \int_{\underline{w}}^{\overline{w}} \varrho(\xi)\, \wt{\bs\psi}[\xi](\bs u) \dif \xi,
        \qquad \forall\; \bs u \in \mc{U}
    \end{equation}
\end{prop}

Symmetric results hold for the entropies $\wt{\bs \upsilon}$, and we denote by $\wt{\bs \Psi}^{\pm}$ the corresponding kinetic normal traces.

\vspace{0.2cm}
At this stage, we summarize the notation related to kinetic entropies and traces. Since several families of objects have been introduced, we collect the main quantities and their definitions in the following table for the reader’s convenience.
\begin{table}[h]
\centering
\renewcommand{\arraystretch}{2.0}
\setlength{\tabcolsep}{18pt}
\begin{tabular}{cccc}
\toprule
\textbf{Entropy} & \textbf{Entropy Flux} & \textbf{Support} & \textbf{Kinetic Normal Trace} \\
\midrule
$\bs\chi[\xi]$             & $\bs\psi[\xi]$             & $\{w \geq \xi\}$  & $\bs\Phi^{\pm}(\cdot, \cdot, \xi)$             \\
$\bs\upsilon[\zeta]$       & $\bs\varphi[\zeta]$        & $\{z \geq \zeta\}$ & $\bs\Psi^{\pm}(\cdot, \cdot, \zeta)$           \\
$\wt{\bs\chi}[\xi]$        & $\wt{\bs\psi}[\xi]$        & $\{w \leq \xi\}$  & $\wt{\bs\Phi}^{\pm}(\cdot, \cdot, \xi)$        \\
$\wt{\bs\upsilon}[\zeta]$  & $\wt{\bs\varphi}[\zeta]$  & $\{z \leq \zeta\}$ & $\wt{\bs\Psi}^{\pm}(\cdot, \cdot, \zeta)$      \\
\bottomrule
\end{tabular}
\vspace{8pt}
\caption{Summary of notation: kinetic entropies, fluxes, supports, and traces.}
\label{tab:notation}
\end{table}

\subsection{Blow-ups, constant weak traces}
We will show in Proposition \ref{prop:tangentu} that blow-ups along Lipschitz curves converge at $\mathscr{H}^1$-almost every point of the curve to highly rigid limiting objects, which we refer to as \textit{isentropic solutions with constant kinetic traces}.
We therefore introduce the following Definition. 
\begin{defi}\label{defi:ckt}
    Let $\vec{n} = (n_t, n_x) \in \mathbb{S}^1$ with $n_x \neq 0$, and let
    \begin{equation*}
        {\mr H}^\pm := \bigl\{(t,x) \in \mathbb R^2 \mid x \gtrless \gamma_b(t),\ t \in \mathbb{R} \bigr\},
        \qquad
        \gamma_b(t) := -\frac{n_t}{n_x} \cdot t. 
    \end{equation*}
    We say that $\bs u \in \mathbf{L}^\infty({\mr H}^\pm, \mc{U})$ is an \emph{isentropic solution
    in a half-space with constant kinetic traces} if there exist functions
    $\bs\Phi^\pm,\, \wt{\bs\Phi}^\pm : (\underline w, \overline w) \to \mathbb R$ and $\bs\Psi^\pm,\, \wt{\bs\Psi}^\pm: (\underline z, \overline z) \to \mathbb R$
    such that for every $\varphi \in C^1_c(\mathbb R^2)$ and $\varrho \in \mathbf{L}^\infty(\mathbb R)$:
    \begin{equation}\label{eq:normalxi}
    \begin{aligned}
        \iint_{{\mr H}^\pm}
        \left(
            \int \varrho(\xi)\,
            \bigl(\bs\chi[\xi](\bs u),\, \bs\psi[\xi](\bs u)\bigr)
            \cdot \nabla\varphi \dif\xi
        \right) \dif x \dif t
        \\[6pt]
        = \int_{\mathrm{Graph}(\gamma_b)}
        \int_{\underline{w}}^{\overline{w}}
        \varphi(t,x)\, \varrho(\xi)\, \bs\Phi^{\pm}(\xi)
        \dif\xi \dif\mathscr{H}^1(t,x)
    \end{aligned}
    \end{equation}
    \begin{equation}\label{eq:normalxitilde}
    \begin{aligned}
        \iint_{{\mr H}^\pm}
        \left(
            \int \varrho(\xi)\,
            \bigl(\wt{\bs\chi}[\xi](\bs u),\, \wt{\bs\psi}[\xi](\bs u)\bigr)
            \cdot \nabla\varphi \dif\xi
        \right) \dif x \dif t
        \\[6pt]
        = \int_{\mathrm{Graph}(\gamma_b)}
        \int_{\underline{w}}^{\overline{w}}
        \varphi(t,x)\, \varrho(\xi)\, \wt{\bs\Phi}^{\pm}(\xi)
        \dif\xi \dif\mathscr{H}^1(t,x)
    \end{aligned}
    \end{equation}
    \begin{equation}\label{eq:normalzeta}
    \begin{aligned}
        \iint_{{\mr H}^\pm}
        \left(
            \int \varrho(\zeta)\,
            \bigl(\bs\upsilon[\zeta](\bs u),\, \bs\varphi[\zeta](\bs u)\bigr)
            \cdot \nabla\varphi \dif\zeta
        \right) \dif x \dif t
        \\[6pt]
        = \int_{\mathrm{Graph}(\gamma_b)}
        \int_{\underline{z}}^{\overline{z}}
        \varphi(t,x)\, \varrho(\zeta)\, \bs\Psi^{\pm}(\zeta)
        \dif\zeta \dif\mathscr{H}^1(t,x)
    \end{aligned}
    \end{equation}
    \begin{equation}\label{eq:normalzetatilde}
    \begin{aligned}
        \iint_{{\mr H}^\pm}
        \left(
            \int \varrho(\zeta)\,
            \bigl(\wt{\bs\upsilon}[\zeta](\bs u),\, \wt{\bs\varphi}[\zeta](\bs u)\bigr)
            \cdot \nabla\varphi \dif\zeta
        \right) \dif x \dif t
        \\[6pt]
        = \int_{\mathrm{Graph}(\gamma_b)}
        \int_{\underline{z}}^{\overline{z}}
        \varphi(t,x)\, \varrho(\zeta)\, \wt{\bs\Psi}^{\pm}(\zeta)
        \dif\zeta \dif\mathscr{H}^1(t,x)
    \end{aligned}
    \end{equation}
\end{defi}

By Fubini's theorem, it is not difficult to see that it is equivalent
to have a condition pointwise in $\xi, \zeta$:

\begin{coro}\label{coro:isehalf}
    Let $\bs u$ be an isentropic solution in a half space with constant kinetic traces. Then for almost every $\xi$ and for all $\varphi \in C^1_c(\Omega)$ there holds
    \begin{equation}\label{eq:halfcont}
        \iint_{{\mr H}^\pm}
        \bigl(\bs\chi[\xi](\bs u),\, \bs\psi[\xi](\bs u)\bigr)
        \cdot \nabla\varphi \dif x \dif t
        = \int_{\mathrm{Graph}(\gamma_b)} \varphi(t,x)\,
        \bs\Phi^{\pm}(\xi) \dif\mathscr{H}^1(t,x)
    \end{equation}
    \begin{equation}\label{eq:halfconttilde}
        \iint_{{\mr H}^\pm}
        \bigl(\wt{\bs\chi}[\xi](\bs u),\, \wt{\bs\psi}[\xi](\bs u)\bigr)
        \cdot \nabla\varphi \dif x \dif t
        = \int_{\mathrm{Graph}(\gamma_b)} \varphi(t,x)\,
        \wt{\bs\Phi}^{\pm}(\xi) \dif\mathscr{H}^1(t,x)
    \end{equation}
    The same holds for the entropies related to the second Riemann invariant.
\end{coro}

\begin{proof}
    Applying Fubini's theorem to \eqref{eq:normalxi}, we deduce that
    for all $\varrho \in \mathbf L^\infty$ and $\varphi \in C^1_c$,
    \begin{equation}\label{eq:normalxi1}
    \begin{aligned}
        \int_{\underline{w}}^{\overline{w}} \varrho(\xi)
        \left(
            \iint_{{\mr H}^\pm}
            \bigl(\bs\chi[\xi](\bs u),\, \bs\psi[\xi](\bs u)\bigr)
            \cdot \nabla\varphi \dif x \dif t
        \right) \dif\xi
        \\[6pt]
        = \int_{\underline{w}}^{\overline{w}} \varrho(\xi)
        \left(
            \int_{\mathrm{Graph}(\gamma_b)}
            \varphi(t,x)\, \bs\Phi^{\pm}(\xi) \dif\mathscr{H}^1(t,x)
        \right) \dif\xi.
    \end{aligned}
    \end{equation}
    Therefore, for a fixed $\varphi \in C^1$, for almost every $\xi$ there holds \eqref{eq:halfcont}.
    Considering a countable dense set of test functions $\varphi \in C^1$
    in the uniform topology, \eqref{eq:halfcont} holds for this dense set
    up to a negligible set of $\xi$.
    Since both sides of \eqref{eq:halfcont} are continuous with respect
    to $C^1$ convergence of $\varphi$, the result extends to all $\varphi \in C^1_c$.
\end{proof}

Given a Lipschitz curve $\gamma:[0,T] \to \mathbb R$ and a finite entropy solution $\bs u$, we define for  $\bar t \in [0,T]$, $r  > 0$, the rescalings
\begin{equation}\label{eq:blowupt}
\bs u^{\bar t}_r(s, y) := \bs u\big(\bar t + sr, \gamma(\bar t) + y r\big)
\end{equation}
In the following we let 
\begin{equation}\label{eq:supentropy}
\bs \nu \doteq \bigvee_{\substack{\eta \in \mc E \\ |\eta|_{C^2} < 1}} |\mu_{\eta}| \; \in \; \msc M(\Omega).
\end{equation}
Here $\bigvee$ denotes the supremum in the sense of measures (see \cite[Definition 1.68]{AFP00}) and $\mc E$ is the set of smooth $C^2$ entropies $\eta : \mc U \to \mathbb R$ , while $\mu_\eta$ is the corresponding dissipation measure \eqref{eq:entrdis}.

\begin{prop}\label{prop:tangentu}
    Let $\bs u \in \mathbf L^\infty(\Omega, \mc U)$ be a finite entropy solution and let $\gamma:[0,T] \to \mathbb R$ be a Lipschitz curve. Then for almost every $\bar t > 0$ the following holds. Assume that for some subsequence $r_j \downarrow 0$ there holds
    $$
    \bs u^{\bar t}_{r_j} \longrightarrow \bs u^{\bar t}_\infty \qquad \text{in $\mathbf L^1_{loc}$ as $j \to +\infty$}
    $$
    for some limit function $\bs u_\infty^{\bar t}$, where $\bs u^{\bar t}_{r_j}$ are as in \eqref{eq:blowupt}. Then $\bs u^{\bar t}_\infty$  restricted to $\Omega_\infty^{\pm}$ is an isentropic solution with constant kinetic traces if in Definition \ref{defi:ckt} we pick
\begin{equation}\label{eq:omegainf}
{\mr H}^- \equiv \Omega^{-}_\infty := \Big\{(t,x) \in \mathbb R^2 \; | \; x < \gamma_\infty(t) \Big\}, \qquad {\mr H}^+ \equiv \Omega^{+}_\infty := \Big\{(t,x) \in \mathbb R^2 \; | \; x > \gamma_\infty(t) \Big\}
\end{equation}
where 
$$
\gamma_b \equiv \gamma_\infty(t) := \dot \gamma(\bar t) t, \quad \forall \; t \in \mathbb R;
$$
and the (constant) kinetic normal traces are given by the kinetic traces of Proposition \ref{prop:weaktraces} computed at the point $(\bar t, \gamma(\bar t))$:
$$
\bs \Phi^{\pm}(\xi) = \bs \Phi^\pm(\bar t, \gamma(\bar t), \xi), \qquad \wt{\bs \Phi}^{\pm}(\xi) = \wt{\bs \Phi}^\pm(\bar t, \gamma(\bar t), \xi)
$$
$$
\bs \Psi^{\pm}(\zeta) = \bs \Psi^\pm(\bar t, \gamma(\bar t), \zeta), \qquad \wt{\bs \Psi}^{\pm}(\zeta) = \wt{\bs \Psi}^\pm(\bar t, \gamma(\bar t), \zeta).
$$
\end{prop}

\begin{proof}
We prove that $\bs u_\infty^{\bar t}$ satisfies \eqref{eq:normalxi}, the proof of the other three conditions being exactly the same. We show that the result holds for the set of $\bar t \in \mathbb R^+$  satisfying the following three conditions:

\vspace{0.3cm}
\noindent\textbf{1.} $(\bar t, \gamma(\bar t))$ is a Lebesgue point of the map $(t,x) \mapsto \bs \Phi^{\pm}(t,x, \cdot)$ in the weak topology with respect to the measure 
    $$
    \mu_\gamma := \mathscr H^1 \llcorner  \mr{Graph}\, (\gamma)
    $$
    i.e. for all $p \in C^0_c(\mathbb R^2\times \mathbb R; \, \mathbb R)$:
\begin{equation}\label{eq:fluxlebp}
\begin{aligned}
\lim_{r \to 0^+}
\frac{1}{\mu_\gamma\!\left(B_r\!\left(\bar{t},\gamma(\bar{t})\right)\right)}
& \int_{B_r\!\left(\bar{t},\gamma(\bar{t})\right)}
\int
p\!\left(\frac{t-\bar{t}}{r},\;\frac{x-\gamma(\bar{t})}{r},\;\xi\right)
\\
&
\cdot
\Bigl(
    \boldsymbol{\Phi}^{\pm}(t,x,\xi)
    -
    \boldsymbol{\Phi}^{\pm}\!\left(\bar{t},\gamma(\bar{t}),\xi\right)
\Bigr)
\,\mathrm{d}\xi
\,\mathrm{d}\mu_\gamma(t,x)
\;=\; 0.
\end{aligned}
\end{equation}

\vspace{0.3cm}
\noindent\textbf{2.} $\bar t$ is a Lebesgue point of the tangent of $\gamma$, i.e.
    \begin{equation}\label{eq:tanglebp}
       \lim_{r \to 0^+} \frac{1}{r}\int_{\bar t-r}^{\bar t+r}|\dot \gamma(s) - \dot \gamma(\bar t)| \dif s = 0.
    \end{equation}

\vspace{0.3cm}
\noindent\textbf{3.}  $\bar t$ is such that the rescaled dissipation vanishes in half spaces:
    \begin{equation}\label{eq:isehalf}
        \lim_{r \to 0^+}\frac{1}{r}\bs \nu \big( B_r^{\pm}(\bar t, \gamma(\bar t)) \big) = 0
    \end{equation}
    where $B_r^{\pm}$ are as in \eqref{eq:halfball}.

By Proposition \ref{prop:lebp}, \ref{prop:besi} in the appendix and the Lebesgue differentiation theorem, this set of $\bar t$ is of full Lebesgue  measure.

\vspace{0.5cm}
    Let $\bar t$ be a time satisfying 1., 2. and 3. above, and let $\varphi \in C^1_c(\mathbb R^2) $ be fixed, and consider the rescalings 
    $$
    \varphi^r(t, x) := \varphi\big(r^{-1}(t-\bar t), r^{-1}(x-\gamma(\bar t))\big).
    $$
    Using $\varphi^r$ in \eqref{eq:kinetictrace} we deduce 
     \begin{equation}\label{eq:blowupping}
\begin{aligned}
    \iint_{\Omega^{\pm}} \varphi^r(t,x) \dif \mu_{\eta_\rho}
    &= -\iint_{\Omega^{\pm}} 
        \left(
            \int \varrho(\xi)
            \begin{pmatrix}
                \bs\chi[\xi](\bs u) \\[4pt]
                \bs\psi[\xi](\bs u)
            \end{pmatrix}
            \cdot \nabla_{t,x} \varphi^r(t,x) \dif \xi
        \right) \dif x \dif t 
    \\[10pt]
    &\quad
        + \int_{\mathrm{Graph}(\gamma)} \varphi^r(t,x)
        \left(
            \int \varrho(\xi)\, \bs\Phi^{+}(t,x,\xi) \dif \xi
        \right)
        \dif \mathscr{H}^1(t,x)
    \\[14pt]
    &= -r \iint_{\Omega^{\pm}_r}
        \left(
            \int \varrho(\xi)
            \begin{pmatrix}
                \bs\chi[\xi](\bs u^{\bar t}_{r_j}) \\[4pt]
                \bs\psi[\xi](\bs u^{\bar t}_{r_j})
            \end{pmatrix}
            \cdot \nabla_{s,y} \varphi \dif \xi
        \right) \dif y \dif s
    \\[10pt]
    &\quad
        + \int_{\mathrm{Graph}(\gamma)}
        \varphi\!\left(\frac{t - \bar{t}}{r},\, \frac{x - \gamma(\bar{t})}{r}\right)
        \left(
            \int \varrho(\xi)\, \bs\Phi^{\pm}(t,x,\xi) \dif \xi
        \right)
        \dif \mathscr{H}^1(t,x)
\end{aligned}
\end{equation}
where we used in the second line the change of variables 
$$
s = r^{-1}(t- \bar t), \qquad y = r^{-1}(x-\gamma(\bar t))
$$
and 
$$
\Omega_r^{\pm} := \Big\{(s, y) \; | \; (\bar t + r s, \gamma(\bar t) + r y) \in \Omega^\pm \Big\}.
$$
By \eqref{eq:isehalf} we deduce
\begin{equation}\label{eq:muetaphizero}
    \lim_{r \to 0^+} \left|\frac{1}{r} \iint_{\Omega^{\pm}}\varphi^r(t,x) \dif \mu_{\eta}  \right| = \mc O(1)  \frac{1}{r} \lim_{r \to 0^+} |\bs \nu|\big(B^{\pm}_{Kr}(\bar t, \gamma(\bar t))\big) = 0
\end{equation}
where $K> 0$ is big enough so that $\mathrm{supp} \, \varphi \subset B_K(0)$.
Notice that 
\begin{equation}\label{eq:blowest}
    \begin{aligned}
          \Bigg|\Bigg(\iint_{\Omega^{\pm}_r}-&\iint_{\Omega^{\pm}_{\infty}} \Bigg) \Big(\int \varrho(\xi)\begin{pmatrix}
    \bs \chi[\xi](\bs u^{\bar t}_{r}) \\[4pt]
    \bs \psi[\xi](\bs u^{\bar t}_{r})
\end{pmatrix}\cdot \nabla_{s,y} \varphi \dif \xi \Big) \dif y \dif s \Bigg| \\[10pt]
         & \leq \Big|(\Omega^{\pm}_r\Delta \Omega^\pm_\infty) \cap B_{K}(0)\Big| \Big(\sup_{\bs u} |(\bs \chi, \bs \psi|\Big) \|\varphi\|_{C^1}\|\varrho\|_{C^0}  = \mc O(1) \cdot r
    \end{aligned}
\end{equation}
where the last equality follows from the tangency property \eqref{eq:tanglebp}, which readily implies the bound on the area of the symmetric difference in a ball:
$$
\Big|(\Omega^{\pm}_r\Delta \Omega^\pm_\infty) \cap B_{K}(0)\Big| = \mc O(1) \cdot r.
$$
Finally we have that by \eqref{eq:fluxlebp}, setting 
$$
p(t, x, z) := \varphi(t,x) \varrho(\xi)
$$
we obtain, setting 
$$
C = \lim_{r \to 0^+} \frac{\mu_\gamma(B_{Kr}(\bar t, \gamma(\bar t))}{r}
$$
\begin{equation}\label{eq:blowest1}
\begin{aligned}
&\lim_{r \to 0^+} \frac{1}{r}
  \int_{\mathrm{Graph}(\gamma)}
  \varphi\!\left(\frac{t-\bar{t}}{r},\, \frac{x-\gamma(\bar{t})}{r}\right)
  \left(\int_{\underline w}^{\overline w} \varrho(\xi)\,\boldsymbol{\Phi}^{\pm}(t,x,\xi)\,\mathrm{d}\xi\right)
  \mathrm{d}\mathscr{H}^{1}(t,x) \\[6pt]
&\quad= \lim_{r \to 0^+} \frac{\mu_\gamma(B_{Kr})}{r}
  \cdot \frac{1}{\mu_\gamma(B_{Kr})}
  \int_{B_{Kr}}
  \varphi\!\left(\frac{t-\bar{t}}{r},\, \frac{x-\gamma(\bar{t})}{r}\right)
  \left(\int_{\underline w}^{\overline w}  \varrho(\xi)\,\boldsymbol{\Phi}^{\pm}(t,x,\xi)\,\mathrm{d}\xi\right)
  \mathrm{d}\mu_\gamma(t,x) \\[6pt]
&\quad= C \lim_{r \to 0^+}
  \frac{1}{\mu_\gamma(B_{Kr})}
  \int_{B_{kr}}
  \varphi\!\left(\frac{t-\bar{t}}{r},\, \frac{x-\gamma(\bar{t})}{r}\right)
  \left(\int_{\underline w}^{\overline w}  \varrho(\xi)\,\boldsymbol{\Phi}^{\pm}\bigl(\bar{t},\gamma(\bar{t}),\xi\bigr)\,\mathrm{d}\xi\right)
  \mathrm{d}\mu_\gamma(t,x) \\[6pt]
&\quad= \lim_{r \to 0^+}
  \left(\int \varphi(s,y)\,\mathrm{d}\!\left[\tfrac{1}{r}\,\alpha^{r,(\bar{t},\gamma(\bar{t}))}_\sharp\,\mu_\gamma\right](s,y)\right)
  \cdot
  \left(\int_{\underline w}^{\overline w}  \varrho(\xi)\,\boldsymbol{\Phi}^{+}\!\bigl((\bar{t},\gamma(\bar{t})),\xi\bigr)\,\mathrm{d}\xi\right) \\[6pt]
&\quad= \left(\int_{\mathrm{Graph}(\gamma_\infty)} \varphi(s,y)\,\mathrm{d}\mathscr{H}^{1}(s,y)\right)
  \cdot
  \left(\int_{\underline w}^{\overline w}  \varrho(\xi)\,\boldsymbol{\Phi}^{+}\!\bigl((\bar{t},\gamma(\bar{t})),\xi\bigr)\,\mathrm{d}\xi\right).
\end{aligned}
\end{equation}
where we have used \eqref{eq:fluxlebp} in the second equality, and the fact that by \eqref{eq:tanglebp}, there holds that 
$$
\frac{1}{r} \alpha^{r, (\bar t, \gamma(\bar t))}_\sharp \, \mathscr H^1 \llcorner \mr{Graph} \, \gamma \longrightarrow \mathscr H^1 \llcorner \mr{Graph} \, \gamma_{\infty} \quad \text{weakly in $\mathscr M_{t,x}$ as $r \to 0^+$}
$$
where  the pushforward is through the map
$$
\alpha^{r, (\bar t, \gamma(\bar t))}(t, x) := (r^{-1}(t- \bar t), r^{-1}(x-\gamma(\bar t)).
$$
Therefore we deduce from \eqref{eq:blowupping}, using \eqref{eq:muetaphizero}, \eqref{eq:blowest} and \eqref{eq:blowest1}:
\begin{equation}
\begin{aligned}
0 &= \lim_{ j \to +\infty} \frac{1}{r_j}
    \Bigg[
        -\iint_{\Omega^{\pm}}
        \left(
            \int \varrho(\xi)
            \begin{pmatrix} \boldsymbol{\chi}[\xi](\boldsymbol{u}) \\ \boldsymbol{\psi}[\xi](\boldsymbol{u}) \end{pmatrix}
            \cdot \nabla_{t,x}\varphi^{r_j} \,\mathrm{d}\xi
        \right)
        \mathrm{d}x\,\mathrm{d}t \\
    &\qquad\qquad
        +\int_{\mathrm{Graph}(\gamma)}
        \varphi^{r_j}(t,x)
        \left(
            \int \varrho(\xi)\,\boldsymbol{\Phi}^{\pm}(t,x,\xi)\,\mathrm{d}\xi
        \right)
        \mathrm{d}\mathscr{H}^{1}(t,x)
    \Bigg] \\[10pt]
    &= \lim_{j \to +\infty} 
        -\iint_{\Omega_{r_j}^{\pm}}
        \left(
            \int \varrho(\xi)
            \begin{pmatrix} \boldsymbol{\chi}[\xi](\boldsymbol{u}^{\bar t}_{r_j}) \\ \boldsymbol{\psi}[\xi](\boldsymbol{u}^{\bar t}_{r_j}) \end{pmatrix}
            \cdot \nabla_{s,y}\varphi \,\mathrm{d}\xi
        \right)
        \mathrm{d}y\,\mathrm{d}s \\
    &\qquad\qquad
        +\lim_{j \to +\infty} \frac{1}{{r_j}}\Bigg[\int_{\mathrm{Graph}(\gamma)}
        \varphi^{r_j}(t,x)
        \left(
            \int \varrho(\xi)\,\boldsymbol{\Phi}^{\pm}(t,x,\xi)\,\mathrm{d}\xi
        \right)
        \mathrm{d}\mathscr{H}^{1}(t,x)
    \Bigg] \\[10pt]
&= -\iint_{\Omega^{\pm}_\infty}
    \left(
        \int \varrho(\xi)
        \begin{pmatrix} \boldsymbol{\chi}[\xi](\boldsymbol{u}^{\bar t}_\infty) \\ \boldsymbol{\psi}[\xi](\boldsymbol{u}^{\bar t}_\infty) \end{pmatrix}
        \cdot \nabla_{s,y}\varphi \,\mathrm{d}\xi
    \right)
    \mathrm{d}y\,\mathrm{d}s \\
&\qquad
    +\int_{\mathrm{Graph}(\gamma_\infty)}
    \varphi(t,x)
    \left(
        \int \varrho(\xi)\,\boldsymbol{\Phi}^{\pm}\!\bigl((\bar{t},\gamma(\bar{t})),\xi\bigr)
        \,\mathrm{d}\xi
    \right)
    \mathrm{d}\mathscr{H}^{1}(t,x).
\end{aligned}
\end{equation}
as wanted. To obtain the last line we used that for almost every $\xi$, $\bs \chi[\xi](\bs u_{r_j}^{\bar t})$ converges in $\mathbf L^1_{loc}$ to $\boldsymbol{\chi}[\xi](\boldsymbol{u}^{\bar t}_\infty)$ if $\bs u_{r_j}^{\bar t}$ converges to $\boldsymbol{u}^{\bar t}_\infty$ in $\mathbf L^1_{loc}$.
\end{proof}

\section{Lagrangian tools for isentropic solutions with boundary}\label{sec:hamlag}
In the rest of the paper we will use some tools related to the Lagrangian representation, already introduced in setting of conservation laws in \cite{AMT25, BBM17, Mar19}, inspired by the superposition principle for the continuity equation (see e.g. \cite{AMB08}). Since we have no available regularity on entropy/finite entropy solutions $\bs u$ of \eqref{eq:systemi}, the Lagrangian representation is the main tool at our disposal, and we use it to perform computations at the kinetic level of the kinetic entropies $\bs \chi[\xi], \wt{\bs \chi}[\xi]$.

In this paper, we only need to consider isentropic solutions. Therefore, in order to make the presentation completely elementary and self contained, we first demonstrate how in this case  the Lagrangian representation can be easily constructed starting from the level sets of Lipschitz Hamiltonians, which also allow to easily deduce some additional properties which are needed for the following.

\vspace{0.5cm}
 We recall below the following easy Proposition (see \cite[Theorem 1.3]{AMT25}, for a more general version).

\begin{prop}\label{prop:kiniso}
A function $\bs u: \Omega \to \mc U$ is an isentropic solution to \eqref{eq:systemi} in an open set $\Omega \subset \mathbb R^2$, i.e. for every smooth entropy-entropy flux pair $(\eta, q)$ there holds
$$
\eta(\bs u)_t + q(\bs u)_x = \mu_\eta = 0 \qquad \text{in $\mathscr D^\prime_{t,x}$},
$$
if and only if for almost every $\xi \in (\underline w, \overline w)$ there holds (recalling the discontinuous entropies of Section \ref{sec:entropies})
     \begin{equation}\label{eq:kin22}
        \; \; \partial_t \bs \chi[\xi](\bs u(t,x)) + \partial_x \bs \psi[\xi](\bs u(t,x)) =0  \qquad \text{in $\msc D^\prime$}\big(\Omega\big)
    \end{equation}
     \begin{equation}\label{eq:kin221}
        \; \; \partial_t \wt{\bs \chi}[\xi](\bs u(t,x)) + \partial_x \wt{ \bs \psi}[\xi](\bs u(t,x)) =0  \qquad \text{in $\msc D^\prime$}\big(\Omega\big)
    \end{equation}
   and for almost every $\zeta \in (\underline z, \overline z)$ there holds
     \begin{equation}\label{eq:kin22z}
        \; \; \partial_t \bs \upsilon[\zeta](\bs u(t,x)) + \partial_x \bs \varphi[\zeta](\bs u(t,x)) =0  \qquad \text{in $\msc D^\prime$}\big(\Omega\big)
    \end{equation}
     \begin{equation}\label{eq:kin221z}
        \; \; \partial_t \wt{\bs  \upsilon}[\zeta](\bs u(t,x)) + \partial_x \wt{ \bs \varphi}[\zeta](\bs u(t,x)) =0  \qquad \text{in $\msc D^\prime$}\big(\Omega\big)
    \end{equation}
\end{prop}

By Proposition \eqref{prop:tangentu}, blow ups $\bs u_\infty$ are isentropic solutions in the open sets $\Omega^\pm$.

\begin{coro}\label{prop:kinH}
    Let $\bs u: \Omega \to \mc U$ be an isentropic solution to \eqref{eq:systemi} in an open set $\Omega \subset \mathbb R^2$. Then for almost every $\xi \in (\underline w, \overline w)$ there exists an Hamiltonian $\br H[\xi](t,x) \in \mathrm{Lip}(\Omega)$ such that 
    $$
    \nabla_{t,x} \, \br H[\xi](t,x) = \begin{pmatrix}
        - \bs \psi[\xi](\bs u(t,x)) \\[6pt]
        \bs \chi[\xi](\bs u(t,x))
    \end{pmatrix} \qquad \text{for a.e. $(t,x) \in \Omega$}.
    $$
Similarly for the epigraph entropies, for almost every $\xi \in (\underline w, \overline w)$ there exists an Hamiltonian $\wt{\br H}[\xi](t,x) \in \mathrm{Lip}(\Omega)$ such that 
    $$
    \nabla_{t,x} \, \wt{\br H}[\xi](t,x) = \begin{pmatrix}
        - \wt{\bs \psi}[\xi](\bs u(t,x)) \\[6pt]
        \wt{\bs \chi}[\xi](\bs u(t,x))
    \end{pmatrix} \qquad \text{for a.e. $(t,x) \in \Omega$}.
    $$
     The same holds for the entropies $\bs \upsilon[\zeta], \wt{\bs \upsilon}[\zeta]$ related to the second Riemann invariant, for some other Hamiltonians $\br G[\zeta], \wt{\br G}[\zeta]$.
\end{coro}

\subsection{Connection with the Lagrangian representation}
Let $\bar r$ be as in Proposition \ref{prop:localspeed}. For $\xi \in (\overline w - \bar  r, \overline w)$, consider the Hamiltonian $\br H[\xi]$, and consider the level sets
$$
\Gamma^\xi_\alpha := \Big\{(t, x) \in \Omega \; | \; \br H[\xi](t,x) = \alpha  \Big\}.
$$
Since, by $\xi \in (\overline w - \bar  r, \overline w)$ and by Proposition \ref{prop:localspeed}, there holds 
\begin{equation}\label{eq:xincr}
\partial_x \br H[\xi](t,x) = \bs \chi[\xi](\bs u(t,x)) > c \quad \text{for a.e. $(t,x)$}
\end{equation}
the level sets $\Gamma^\xi_\alpha$ are graphs of Lipschitz curves  $\gamma_\alpha^\xi$ parametrized by a closed subset of the timeline $\mathbb R$:
$$
\Gamma^{\xi}_\alpha  = \Big\{(t, \gamma_\alpha^\xi(t)) \; | \; t \in I \subset \mathbb R \Big\}.
$$
The same holds for $\wt {\br H}[\xi]$ for $\xi \in (\underline w, \underline w + \bar r)$. 
Define $\Gamma$ as the space of Lipschitz curves
$$
\Gamma = \Big\{(\gamma, I_\gamma)\; | \; I_\gamma \subset \mathbb R\; \text{is an closed set,} \quad \gamma: I_\gamma \to \mathbb R\;  \text{Lipschitz curve with $(t, \gamma(t)) \in \Omega$}  \Big\}
$$
endowed with the Hausdorff metric on graphs $\mr{dist}_{\mc H}$.
Consider the map $\mathcal R: \mathbb R \to \Gamma$ defined by
$$
\alpha \mapsto \mathcal R(\alpha) := \gamma^\xi_\alpha \in \Gamma
$$
sending each $\alpha \in \mathbb R$ to the curve corresponding to the level set $\Gamma_\alpha^\xi$.
Consider the finite, positive measure $\bs \omega_\xi \in \mathscr M(\Gamma)$ defined by 
\begin{equation}\label{eq:lagradef}
\bs \omega_\xi := \mathcal R_\sharp \mathscr L^1\llcorner A_\xi, \qquad A_\xi:= \mathrm{Im} \br H[\xi].
\end{equation}
We recall that here $\mathcal R_\sharp$ denotes the pushforward, and therefore $\bs \omega_\xi$ is defined by the property 
$$
\bs \omega_\xi(E) = \mathscr L^1\Big(\big\{\alpha \in A_\xi \; | \; \gamma_\alpha^\xi \in E \big\} \Big).
$$

\begin{defi}\label{def:lagradef}
   Let $\bs u$ be an isentropic solution to \eqref{eq:systemi} in an open set $\Omega \subset \mathbb R^2$, and let $\xi \in (\overline w - \bar r, \overline w)$, where $\overline w := \esssup_\Omega w$. We say that a measure $\bs \omega \in \msc M(\Gamma)$ is a Lagrangian representation for $\bs \chi[\xi](\bs u)$ if 
\begin{enumerate}
    \item For $\bs \omega$-almost every $\gamma$, $(t, \gamma(t))$ is a Lebesgue point of $\bs u$ for a.e. $t \in I_\gamma$, and  $\gamma$ is characteristic for $\bs \lambda_1[\xi](\bs u)$:
    \begin{equation}\label{eq:charomega}
    \dot \gamma(t) = \bs \lambda_1[\xi](\bs u(t,\gamma(t)) \qquad \text{for a.e. $t \in I_\gamma$}.
        \end{equation}
    \item Up to redefining $\bs \chi[\xi](\bs u)$  on a set of times of measure zero, we can recover it by superposition of the curves:
    \begin{equation}\label{eq:chisuper}
\bs \chi[\xi](\bs u(t, \cdot))\cdot \msc L^1 = (e_t)_\sharp \bs \omega\llcorner \Gamma_{t} \qquad \text{for all $t \in \mathbb R$}.
        \end{equation}
    where $e_t : \Gamma_t \to \mathbb R$ is the evaluation map $e_t(\gamma) = \gamma(t)$, defined on the set of curves $$\Gamma_t := \{\gamma \in \Gamma \; |\; t \in I_\gamma\}.$$
    \item For $\bs \omega$-almost every $\gamma$, there holds
    $$
    w(t,\gamma(t)) \geq \xi \quad \text{for a.e. $t \in I_\gamma$}.
    $$
\end{enumerate}
We call $\Gamma[\xi]$ the set of curves satisfying (1), (2), (3) where $\bs \omega$ is concentrated.
\end{defi}

We claim that $\bs \omega_\xi$ defined in \eqref{eq:lagradef} is a Lagrangian representation of $\bs \chi[\xi](\bs u)$.

\begin{prop}\label{prop:lagraprop}
   Let $\bs u$ be an isentropic solution to \eqref{eq:systemi} in an open set $\Omega \subset \mathbb R^2$, and let $\xi \in (\overline w - \bar r, \overline w)$, where $\overline w := \esssup_\Omega w$. Then the measure $\bs \omega_\xi$ defined in \eqref{eq:lagradef} is a Lagrangian representation of $\bs \chi[\xi](\bs u)$ in the sense of Definition \ref{def:lagradef}.
\end{prop}

\begin{proof}

\textbf{1.} We start by verifying (2), which amounts to notice that for every measurable set $E \subset \mathbb R$ and for almost every $t > 0$ there holds
    \begin{equation}
\begin{aligned}
(e_t)_\sharp\,\boldsymbol{\omega}_\xi \llcorner \Gamma_t(E)
&= \boldsymbol{\omega}_\xi\!\Big(\{\gamma \in \Gamma_t \mid \gamma(t) \in E\}\Big) \\[4pt]
&= \mathcal{R}_\sharp\,\mathscr{L}^1 \llcorner A_\xi
    \Big(\{\gamma \in \Gamma_t \mid \gamma(t) \in E\}\Big) \\[4pt]
&= \mathscr{L}^1 \llcorner A_\xi
    \big(\{\alpha \mid t \in I_{\gamma_\alpha^\xi},\; \gamma_\alpha^\xi(t) \in E\}\big) \\[4pt]
&= \mathscr{L}^1\!\left(\mathbf{H}[\xi](t, E)\right) \\[4pt]
&= \int_E \boldsymbol{\chi}[\xi](t, x)\,\mathrm{d}x.
\end{aligned}
\end{equation}
    where we denote in the penultimate line $\br H[\xi](t, E)$ the image of the set $E$ through the map $\br H[\xi](t, \cdot)$, and the last equality follows from the fact that $\bs \chi[\xi](t,x) = \partial_x \br H[\xi](t,  x)$.
   This proves (2).

\vspace{0.5cm}
\textbf{2.} We now show that for $\bs \omega_\xi$-almost every $\gamma$, $(t, \gamma(t))$ is a Lebesgue point of $\bs u$ for a.e. $t \in I_\gamma$.  Let $S \subset \mathbb R^2$ be any set of zero Lebesgue measure (in particular, it can be the set of non-Lebesgue points of $\bs u$ by the Lebesgue differentiation theorem). From \eqref{eq:chisuper}, which was proved in step \textbf{1.} above, we deduce that for almost every $t \in \mathbb R$ it holds $(e_t)_\sharp\bs \omega_\xi \llcorner \Gamma_t \ll \msc L^1$, therefore
    $$
    \msc L^1 \otimes(e_t)_\sharp\bs \omega_\xi \llcorner \Gamma_t  \ll \msc L^2.
    $$
    Tonelli's theorem gives
    \begin{equation}\label{eq:Szero}
    \begin{aligned}
        \int_{\Gamma} \msc L^1\Big(\big\{t \in I_\gamma \; | \; (t,\gamma(t)) \in S\big\} \Big) \dif \bs \omega_\xi(\gamma) & = \int_{\mathbb R} \bs \omega_\xi\Big( \big\{\gamma \in \Gamma_t \; | \; (t, \gamma(t)) \in S\big\} \Big) \dif t\\
        & = \left(\msc L^1 \otimes (e_t)_\sharp \bs \omega_\xi\llcorner \Gamma_t \right)(S) = 0
    \end{aligned}
        \end{equation}
and this proves the claim.

\vspace{0.5cm}
    \textbf{3.}
    We now show that $\bs \omega_\xi$ is concentrated on characteristic curves, that is, it satisfies the second part of point (1) of Definition \ref{def:lagradef}. Since $\br H[\xi]$ is Lipschitz, its set of differentiability points has full measure in $\Omega$. Therefore, by the same argument used in Step \textbf{2.}, we can show that $\bs \omega_\xi$-almost every $\gamma$ is such that $(t, \gamma(t))$ is a differentiability point of $\br H[\xi]$ for almost every $t \in I_\gamma$ (and Lebesgue point for $\bs u$). Moreover, $\bs \omega_\xi$ is concentrated on curves which are level sets of $\br H[\xi]$. Therefore, for $\bs \omega_\xi$-almost every $\gamma$, there holds for almost every $t \in I_\gamma$ that
    $$
    \br H[\xi](t+ \delta, \gamma(t+\delta)) - \br H[\xi](t, \gamma(t)) = \nabla \br H[\xi](t, \gamma(t)) \cdot (1, \dot \gamma(t))\delta  + o(\delta)
    $$
    where $o(\delta)$ is a quantity approaching zero faster then $\delta$.
    By dividing by $\delta$ and taking the limit for $\delta \to 0$, and using the definition of $\bs \lambda_1$ in \eqref{eq:lambdadef}, we deduce 
    $$
    \begin{aligned}
    0 & =  \nabla \br H[\xi](t, \gamma(t)) \cdot (1, \dot \gamma(t)) =   \begin{pmatrix}
        - \bs \psi[\xi](\bs u(t,\gamma(t))) \\[6pt]
        \bs \chi[\xi](\bs u(t,\gamma(t))
    \end{pmatrix} \cdot (1, \dot \gamma(t)) \\[10pt]
    & = \bs \chi[\xi](\bs u(t, \gamma(t)))  \begin{pmatrix}
        -\bs \lambda_1[\xi](\bs u(t,\gamma(t))) \\[6pt]
        1
    \end{pmatrix} \cdot (1, \dot \gamma(t)) 
      \end{aligned}
    $$
    
But this implies that for $\bs \omega_\xi$-almost every curve, and for almost every $t \in I_\gamma$, using that by the assumption $\xi \in (\overline w-\bar  r, \overline w)$ there holds $\bs \chi[\xi](\bs u(t,\gamma(t)) \geq c > 0$ for a.e. $t$, there holds $\dot \gamma(t) = \bs \lambda_1[\xi](\bs u(t, \gamma(t)))$.

\textbf{4.}
To prove (3), we first observe that for every $t \in \mathbb R$, 
$$
\begin{aligned}
\bs \omega_\xi & \left(\big\{ \gamma \in \Gamma_t \; | \; \xi > w(t, \gamma(t)\big)\big\} \right) \\[10pt]
 =\,&  (e_t)_{\sharp} \bs \omega_\xi\llcorner\Gamma_t \Big(\big\{x \;  | \; \xi > w(t,x) \big\} \Big) \\[10pt]
 = \, & \int_{\mathbb R} \bs \chi[\xi](\bs u(t,x)) \mathbf 1_{\{x \;  | \; \xi > w(t,x)\}}(t,x) \dif x = 0
\end{aligned}
$$
by definition of $\bs \chi[\xi](\bs u)$, and then we proceed as in \eqref{eq:Szero} using Tonelli's theorem to deduce
$$
\begin{aligned}
 & \int_{\Gamma} \msc L^1 \Big(\big\{t \in I_\gamma \; | \;\xi > w(t, \gamma(t)\big\} \Big) \dif \bs \omega_\xi(\gamma) \\
 =\, &  \int_{\mathbb R} \bs \omega_\xi\Big( \big\{\gamma \in \Gamma_t \; | \; \xi > w(t, \gamma(t)\big) \big\} \Big) \dif t= 0.
        \end{aligned}
$$

\end{proof}

\begin{remark}\label{rem:oderderd}
    Recall that $\bs \omega_\xi = \mathcal{R}_\sharp \mathscr L^1 \llcorner A_\xi$ is concentrated on curves which are level sets of the Lipschitz function $\br H[\xi](t,x)$. Therefore, by \eqref{eq:xincr}, we can additionally require that the set $\Gamma[\xi]$ of curves where $\bs \omega_\xi$ is concentrated satisfies the following property. Let $\gamma_1, \gamma_2 \in \Gamma[\xi]$, and let  $t \in I_{\gamma_1} \cap I_{\gamma_2}$ for some $t \in \mathbb R$ and $\gamma_1(t) < \gamma_2(t)$. Then there holds
    $$
\gamma_1(s) < \gamma_2(s) \quad \forall \; s \in I_{\gamma_1}\cap I_{\gamma_2}
    $$
\end{remark}

\begin{remark}\label{rem:levellagr}
For $\xi \notin (\overline w - \bar r, \overline w)$, the condition $\bs \chi[\xi](\bs u(t,x)) > 0$ may fail to hold for some, or even all, $(t,x)$. This does not, in principle, prevent us from obtaining a Lagrangian representation in this regime, still by considering the level sets of the Hamiltonian $\br H[\xi]$. However, one may then expect pathological (in our setting) behaviors to arise, such as curves that reverse their time orientation or the formation of cycles (see, e.g., \cite{BG22}).

\vspace{0.5cm}
Let $I_\gamma = \cup (a_i, b_i)$ be the union of the maximal connected components of $I_\gamma$. Since $\gamma$ is of the form $\gamma = \gamma_\alpha^\xi$ for some $\alpha$ (i.e. it is a level set of $\br H[\xi]$) it must hold that 
$$
(a_i, \gamma(a_i+)), \;  (b_i, \gamma(b_i-)) \in \partial \Omega
$$
i.e. it can end only at points where it touches to boundary of $\Omega$.
\end{remark}

The same can be done for the entropies $\wt{\bs \chi}[\xi]$, for $\xi \in (\underline w, \underline w + \bar r)$. The definition will be entirely symmetric. 
\begin{defi}\label{def:lagradefepi}
  Let $\bs u$ be an isentropic solution in an open set $\Omega \subset \mathbb R^2$ and let $\xi \in (\underline w, \underline w + \bar r)$, where $\underline w := \essinf_\Omega w$. We  say that $\wt{\bs \omega}$ is a Lagrangian representation for $\wt{\bs \chi}[\xi](\bs u)$ if
\begin{enumerate}
    \item For $\wt{\bs \omega}$-almost every $\gamma$, $(t, \gamma(t))$ is a Lebesgue point of $\bs u$ for a.e. $t \in I_\gamma$, and  $\gamma$ is characteristic for $\bs \lambda_1[\xi](\bs u)$:
    \begin{equation}\label{eq:charomega1}
    \dot \gamma(t) = \bs \lambda_1[\xi](\bs u(t,\gamma(t))) \qquad \text{for a.e. $t \in I_\gamma$}.
        \end{equation}
    \item Up to redefining $\wt{\bs \chi}[\xi](\bs u)$  on a set of times of measure zero, we can recover it by superposition of the curves:
    \begin{equation}\label{eq:chisuper1}
\wt{\bs \chi}[\xi](\bs u)(t, \cdot)\cdot \msc L^1 = (e_t)_\sharp \wt{\bs \omega} \llcorner \Gamma_t  \qquad \text{for all $t \in \mathbb R$}.
        \end{equation}
     \item For $\wt{\bs \omega}$-almost every $\gamma$, there holds
    $$
    w(t,\gamma(t)) \leq \xi \quad \text{for a.e. $t \in I_\gamma$}.
    $$
\end{enumerate}
We call $\wt \Gamma[\xi]$ the set of curves satisfying (1), (2) where $\wt{\bs \omega}$ is concentrated.
\end{defi}

The existence of a representation $\wt{\bs \omega}$ can clearly be obtained in the same way as $\bs \omega$.
In particular, again we consider the finite, positive measure $\wt{\bs \omega}_\xi \in \mathscr M(\Gamma)$ defined by 
\begin{equation}\label{eq:lagradef}
\bs \omega_\xi := \mathcal R_\sharp \mathscr L^1\llcorner \wt{A}_\xi, \qquad \wt A_\xi:= \mathrm{Im} \wt{\br H}[\xi].
\end{equation}
We recall the following standard fact. In the same way one verifies that $\wt{\bs \omega}$ is a Lagrangian representation of $\wt{\bs \chi}[\xi]$.

\begin{lemma}\label{lemma:fubinis}
    Let $\xi \in (\overline w-\bar r, \overline w)$. For every measurable $f \in \mathbf L^1(\Omega)$ and every measurable subset $A \subset \Omega$, there holds
\begin{equation}
    \int_A \bs \chi[\xi](\bs u(t,x)) f(t,x) \dif x \dif t  = \int_{\Gamma} \Bigg(\int_{\mathbb R} f(t, \gamma(t)) \mathbf 1_{\{t \in I_\gamma \, : \, \gamma(t) \in A_t\}}  \dif t \Bigg) \dif \bs \omega_\xi(\gamma).
\end{equation}
where $A_t = \{x \; | \; (t,x) \in A\}$ is the time section of $A$.
\end{lemma}
Clearly, the same result holds for $\wt{\bs \chi}[\xi]$ for $\xi \in (\underline w, \underline w + \bar r)$.
\begin{proof}
The definition of pushforward of point (2) of the Definition \ref{def:lagradef} is equivalent to 
$$
\int_E \bs \chi[\xi](\bs u(t,x)) \varphi(x) \dif x = \int_\Gamma \varphi(\gamma(t)) \dif \bs \omega_\xi(\gamma) \qquad \forall \; \varphi \in C^0, \quad \text{$E$ measurable}.
$$
    Therefore assume first that $f \in C^0$, then in this case the result is just an application of Fubini's theorem
    \begin{equation}
        \begin{aligned}
             \iint_A  \bs \chi[\xi](\bs u(t,x)) f(t,x) \dif x \dif t  & = \int_{\mathbb R} \int_{A_t} \bs \chi[\xi](\bs u(t,x)) f(t,x) \dif x \dif t \\[6pt]
             & = \int_{\mathbb R} \int_{\Gamma} \mathbf{1}_{\{\gamma(t) \in A_t\}} f(t,\gamma(t)) \dif  \bs \omega_\xi(\gamma) \dif t \\[6pt]
             &=\int_{\Gamma} \Bigg(\int_{\mathbb R} f(t, \gamma(t)) \mathbf 1_{\{\gamma(t) \in A_t\}}  \dif t \Bigg) \dif \bs \omega_\xi(\gamma).
        \end{aligned}
    \end{equation}
Next, for $f \in \mathbf L^1$, we consider any smooth approximation $f^\varepsilon$ of $f$ converging in $\mathbf L^1$ to $f$, and we notice that by using the smooth case we have, using  the dominated convergence theorem to bring the limits inside the integrals 
 \begin{equation}
        \begin{aligned}
        \iint_A  \bs \chi[\xi](\bs u(t,x)) f(t,x) \dif x \dif t  
           & =  \lim_{\varepsilon \to 0} \int_A  \bs \chi[\xi](\bs u(t,x)) f^\varepsilon(t,x) \dif x \dif t \\[6pt]
            & = \int_{\mathbb R} \int_{\Gamma}\mathbf 1_{\{\gamma(t) \in A_t\}}  \lim_{\varepsilon \to 0} f^\varepsilon(t,\gamma(t)) \dif  \bs \omega_\xi(\gamma) \dif t\\[6pt]
            &  = \int_{\mathbb R} \int_{\Gamma}\mathbf 1_{\{\gamma(t) \in A_t\}}f(t,\gamma(t)) \dif  \bs \omega_\xi(\gamma) \dif t
        \end{aligned}
    \end{equation}
    where we used that $f^\varepsilon(t, \cdot)$ converges pointwisely almost everywhere to $f(t, \cdot)$ for almost every $t$, and therefore since $(e_t)\sharp \bs \omega_\xi\ll \mathscr L^1$ we have that $f^\varepsilon(t,\gamma(t))$ converges to $f(t,\gamma(t))$ for $\bs \omega_\xi$-almost every $\gamma$ for almost every $t$.
\end{proof}

\vspace{0.3cm}
It is clear that the analogous representations can be obtained for the entropies $\bs \upsilon[\zeta], \wt{\bs \upsilon}[\zeta]$ related to the second Riemann invariant, and we call them $\bs \eta_\zeta$, $\wt{\bs \eta}_\zeta$.

\vspace{0.3cm}
The following standard Lemma is needed for later (see e.g \cite[Lemma 5]{Mar22}).
\begin{lemma}\label{lemma:traces}
Assume that $\gamma:(t_1,t_2)\subset \mathbb R\to \R$ is a Lipschitz curve and that for $\msc L^1$-a.e. $t \in (t_1,t_2)$ the point $(t, \gamma(t))$ is a Lebesgue point of $\bs u \in \mathbf L^\infty(\Omega; \,\mathbb R^2)$. Then
\begin{equation}\label{E_trace}
\lim_{\delta \to 0} \int_{t_1}^{t_2} \frac{1}{\delta}\int_{\gamma(t)-\delta}^{\gamma(t)+\delta}|\bs u(t,x)-\bs u(t,\gamma(t))|\dif x \dif t = 0.
\end{equation}
\end{lemma}

\section{Half space Liouville-type theorem}\label{sec:HSL}

In this section we prove the following half-space Liouville-type theorem, which is the main technical result of the paper.

\begin{thm}\label{thm:Liuv1}
    Let $\bs u\in \mathbf L^\infty(\mr{H}^{\pm}, \mc U)$ be an isentropic solution of \eqref{eq:systemi} in an half space $\mr{H}^{\pm}$, i.e. for some $\vec n = (n_t, n_x) \in \mathbf S^1$ with $n_x \neq 0$:
    \begin{equation}\label{eq:equa}
\mr{H}^{\pm}  := \Big\{(t,x) \, | \, x \gtrless \gamma_b(t), \quad t \in \mathbb R \Big\}, \qquad \gamma_b(t) := -\frac{n_t}{n_x} \cdot t
 \end{equation}
    Assume that $\bs u$ has constant  kinetic normal traces $\bs \Phi^\pm(\xi), \bs \Psi^\pm(\zeta)$, $\wt{\bs \Phi}^\pm(\xi), \wt{\bs \Psi}^\pm(\zeta)$ i.e. it satisfies \eqref{eq:normalxi}, \eqref{eq:normalxitilde}, \eqref{eq:normalzeta}, \eqref{eq:normalzetatilde}. Then $\bs u \equiv (w_c^\pm, z_c^\pm)$ is a constant function, with 
    $$
   w_c := \esssup \, \big\{ \xi \in \mathbb R\; | \;  \bs \Phi(\xi )  \neq 0\Big\} , \qquad  z_c := \esssup \, \big\{ \zeta \in \mathbb R\; | \;  \bs \Psi(\zeta )  \neq 0\Big\}.
    $$
\end{thm}
\begin{remark}\label{rem:essinf}
  With the symmetric proof, one can show also that 
 $$
       w_c^\pm  = \essinf \, \big\{ \xi \in \mathbb R\; | \;  \wt{\bs \Phi}^\pm(\xi )  \neq 0\Big\} , \qquad  z_c^\pm  = \essinf\, \big\{ \zeta \in \mathbb R\; | \;  \wt{\bs \Psi}^\pm(\zeta )  \neq 0\Big\}.
    $$
\end{remark}

We will prove in this section that $\phi_1(\bs u) \equiv w_c^\pm$, i.e. the first Riemann invariant is a constant function, in the regions $\mr{H}^{\pm}$ respectively. In order to prove the claim, we only need the kinetic traces corresponding to $w$, namely $\boldsymbol{\Phi}^\pm$ and $\widetilde{\boldsymbol{\Phi}}^\pm$. The identities $\phi_2(\boldsymbol{u}) \equiv z_c^\pm$ in $\mr{H}^\pm$ can be established in a completely analogous way. Therefore, in the following we state the arguments only for the first Riemann invariant, with the understanding that the symmetric result holds for the second one.

Also, it is clearly the same to prove the result only in $\mr{H}^{+}$: 
$$
\mr{H}^{+} := \Big\{(t,x) \, | \, x > \gamma_b(t), \quad t \in \mathbb R \Big\}, \qquad \gamma_b(t) := -\frac{n_t}{n_x} \cdot t
$$
where here $\vec n^\perp$ plays the role of $\dot \gamma_\infty/|\dot \gamma_\infty|$, therefore from now on we will restrict to this case.

Let $\bar r > 0$ be as in Proposition \ref{prop:localspeed}, so that in particular for all $(t,x) \in \mr{H}^{+}$ there holds
\begin{equation}\label{eq:chipos}
\bs \chi[\xi](\bs u(t,x))  > 0 \quad \text{for all $\xi \in (\overline w- \bar r, \overline w)$},
\end{equation}
\begin{equation}\label{eq:chitildepos}
\wt{\bs \chi}[\xi](\bs u(t,x))  > 0 \quad \text{for all $\xi \in (\underline w, \underline w + \bar r)$},
\end{equation}
where again 
$$
\overline w = \esssup_{{\mr H}^+} w, \qquad \underline w = \essinf_{{\mr H}^+} w.
$$

\vspace{0.5cm}
We provide a sketch of the proof of Theorem \ref{thm:Liuv1} and a summary of the results of this Section. 

\vspace{0.3cm}
\textbf{1.} Proposition \ref{prop:tracesleb} and Proposition \ref{prop:balance} are technical tools needed to compute balances in regions delimited by curves with good properties. These properties are  fulfilled  by almost every curve selected by the Lagrangian representations of Section \ref{sec:hamlag}.

Let
$$
w^\star := \esssup\Big\{ \xi \; | \; \bs \Phi(\xi) \neq 0\Big\} 
$$
be the essential supremum of the support of $\bs \Phi$.

\vspace{0.3cm}
\textbf{2.} In Proposition \ref{lemma:highcurves} we show the easy fact that curves belonging to $\Gamma[\xi]$, as defined in Definition \ref{def:lagradef} in the set $\Omega \equiv \mr{H}^+$, representing $\bs \chi[\xi]$, with $\xi > w^\star$, cannot reach the boundary curve $\gamma_b$. The intuition is clear in that for the range of $\xi > w^\star$ one has $\bs \Phi(\xi) = 0$, therefore there is no boundary flux in the continuity equation \eqref{eq:halfcont}, in particular curves are not allowed to flow through the boundary curve $\gamma_b$.

\vspace{0.3cm}
\textbf{3.} It is a bit more difficult to show that also curves $\wt\gamma \in \wt \Gamma[\xi]$ as defined in Definition \ref{def:lagradefepi} in the set $\Omega \equiv \mr{H}^+$, cannot reach the boundary if $\xi < w^\sharp := \min\{w^\star, \underline w + \bar r\}$, and is the content of Proposition \ref{prop:lowertracepre}. The intuition is slightly different here and goes as follows. If a curve $\wt \gamma \in \wt \Gamma[\xi]$ with $\xi < w^\sharp$ reaches the boundary, by (3) of Definition \ref{def:lagradefepi} (very informally) one has $w < \xi$ at that point of the boundary. By definition of $\bs \chi$, one has $\bs \chi[\hat \xi](\bs u) = 0$ if $\hat \xi > w = \phi_1(\bs u)$. Therefore if a curve in $\wt \Gamma[\xi]$ reaches the boundary, at that point we would have $\bs \chi[\hat \xi](\bs u) = 0$ for all $\hat \xi > w$, so in particular for all $\hat \xi > \xi$. But this, again very informally, would imply that $\bs \Phi(\hat \xi) = 0$ for all $\hat   \xi > \xi$ (i.e. $\bs \chi[\hat \xi]$ does not have flux at that point in the boundary, being zero at that point), and therefore  $w^\star  \leq \xi$, a contradiction.

\vspace{0.3cm}
\textbf{4.} In Corollary \ref{prop:lowertrace}, we show the quite standard fact that since curves in $\wt \Gamma[\xi]$ cannot reach the boundary for $\xi \leq w^\sharp$, the corresponding weak normal trace must be zero, i.e. we show that 
$$
\wt {\bs \Phi}(\xi) = 0 \quad \text{for all $\xi \leq w^\sharp$}.
$$

\vspace{0.3cm}
\textbf{5.} Finally, we prove Theorem \ref{thm:Liuv1}. By the above steps, we have that for $\xi > w^\star$, $\bs \chi[\xi]$ is not affected by the boundary, and for $\xi \leq w^\sharp$, $\wt {\bs \chi}[\xi]$ is also not affected by the boundary. Therefore we are in a position to prove a decay-dispersive type estimate using the genuine nonlinearity of the eigenvalue, in the region ${\mr H}^+$, which allows to conclude the proof of Theorem \ref{thm:Liuv1}.

\vspace{0.3cm}
\textbf{6.} The Section is concluded with the proof of Theorem \ref{thm:main}, which at this point is a standard consequence of point \textbf{5.}

\vspace{0.5cm}
We start by proving the following technical result. If $\bs\chi[\xi](\bs u), \bs \psi[\xi](\bs u)$ were Lipschitz functions of their arguments, the next Proposition would follow immediately from Lemma \ref{lemma:traces}, but they have a jump along the curve $\{\phi_1(\bs u) = \xi\}$. Therefore a small additional technical argument is needed, using the assumption \eqref{eq:gammanodeg}.
\begin{prop}\label{prop:tracesleb}
Let $\xi \in (\underline w, \overline w)$  be fixed, and let $\gamma : (t_1, t_2) \to \mathbb R$ be a Lipschitz curve such that $(t, \gamma(t))$ is a Lebesgue point of $\bs u$ for almost every $t$, and satisfying
\begin{equation}\label{eq:gammanodeg}
 w(t, \gamma(t)) \neq \xi \qquad \text{ for a.e. $t \in (t_1, t_2)$}
\end{equation}
Then there holds
\begin{equation}\label{eq:tracesleb}
\begin{aligned}
\lim_{\delta \to 0^+} \int_{I_\gamma}  \fint_{\gamma(t)-\delta}^{\gamma(t)+\delta} & \big( \bs \chi[\xi](\bs u(t, x)), \bs \psi[\xi](\bs u(t, x)) \big) \cdot (-\dot \gamma(t), 1) \dif x \dif t \\
& = \int_{t_1}^{t_2}    \big( \bs \chi[\xi](\bs u(t, \gamma(t))), \bs \psi[\xi](\bs u(t, \gamma(t))) \big) \cdot (-\dot \gamma(t), 1)  \dif x\dif t.
\end{aligned}
\end{equation}
and 
\begin{equation}\label{eq:tracesleb1}
\begin{aligned}
\lim_{\delta \to 0^+} \int_{I_\gamma}  \fint_{\gamma(t)-\delta}^{\gamma(t)+\delta} & \big( \wt{\bs \chi}[\xi](\bs u(t, x)), \wt{\bs \psi}[\xi](\bs u(t, x)) \big) \cdot (-\dot \gamma(t), 1)  \dif x\dif t \\
& = \int_{t_1}^{t_2}    \big( \wt{\bs \chi}[\xi](\bs u(t, \gamma(t))), \wt{\bs \psi}[\xi](\bs u(t, \gamma(t))) \big) \cdot (-\dot \gamma(t), 1) \dif x \dif t.
\end{aligned}
\end{equation}
where for a set $A$ we denote the mean operator on a set $A$ as $\fint_A := |A|^{-1} \int$ .
\end{prop}
\begin{proof}
We prove the statement for \eqref{eq:tracesleb}, since \eqref{eq:tracesleb1} is entirely symmetric.

    \vspace{0.5cm}
It is clearly enough to prove separately 
    \begin{equation}\label{eq:jumpleb}
    \lim_{\delta \to 0^+} \int_{t_1}^{t_2}  \fint_{\gamma(t)-\delta}^{\gamma(t)+\delta} |\dot \gamma(t)| \Big| \bs \chi[\xi](\bs u(t,x)) - \bs \chi[\xi](\bs u(t, \gamma(t)))\Big| \dif x \dif t = 0
    \end{equation}
 \begin{equation}\label{eq:jumpleb1}
    \lim_{\delta \to 0^+} \int_{t_1}^{t_2}  \fint_{\gamma(t)-\delta}^{\gamma(t)+\delta}\Big| \bs \psi[\xi](\bs u(t,x)) - \bs \psi[\xi](\bs u(t, \gamma(t)))\Big| \dif x \dif t = 0.
    \end{equation}
    We prove \eqref{eq:jumpleb}, since \eqref{eq:jumpleb1} can be proved in exactly the same way. We split the integral 
\begin{align}\label{eq:intsplit}
     \int_{I_\gamma}   & \fint_{\gamma(t)-\delta}^{\gamma(t)+\delta}  \Big| \bs \chi[\xi](\bs u(t,x)) - \bs \chi[\xi](\bs u(t, \gamma(t)))\Big| \dif x \dif t  \\[6pt]\notag
     & = \int_{\left\{\substack{t \in I_\gamma \, : \\ w(t, \gamma(t)) > \xi+\kappa^{-1}}\right\}}  \fint_{\gamma(t)-\delta}^{\gamma(t)+\delta}  \Big| \bs \chi[\xi](\bs u(t,x)) - \bs \chi[\xi](\bs u(t, \gamma(t)))\Big| \dif x \dif t\\[6pt]\notag
    & \qquad+ \int_{\left\{\substack{t \in I_\gamma \, : \\ w(t, \gamma(t)) < \xi-\kappa^{-1}}\right\}}  \fint_{\gamma(t)-\delta}^{\gamma(t)+\delta}  \Big| \bs \chi[\xi](\bs u(t,x)) - \bs \chi[\xi](\bs u(t, \gamma(t)))\Big| \dif x  \dif t\\[6pt]
    & \qquad +\int_{\left\{\substack{t \in I_\gamma \, : \\ \xi-\kappa^{-1} \leq  w(t, \gamma(t)) \leq \xi+\kappa^{-1}}\right\}}  \fint_{\gamma(t)-\delta}^{\gamma(t)+\delta}  \Big| \bs \chi[\xi](\bs u(t,x)) - \bs \chi[\xi](\bs u(t, \gamma(t)))\Big| \dif x  \dif t.
\end{align}
We start by estimating the first term in the right hand side by 
   \begin{align}\label{eq:int1}
&\int_{\left\{\substack{t \in I_\gamma \,:\\ w(t,\gamma(t)) > \xi+\kappa^{-1}}\right\}}
\fint_{\gamma(t)-\delta}^{\gamma(t)+\delta}
\Bigl|\boldsymbol{\chi}[\xi](\boldsymbol{u}(t,x))
- \boldsymbol{\chi}[\xi](\boldsymbol{u}(t,\gamma(t)))\Bigr|
\,\mathrm{d}x\,\mathrm{d}t \notag \\[6pt]
&\quad= \int_{\left\{\substack{t \in I_\gamma \,:\\ w(t,\gamma(t)) > \xi+\kappa^{-1}}\right\}}
\fint_{\gamma(t)-\delta}^{\gamma(t)+\delta}
\mathbf{1}_{\{w(t,x)>\xi\}}
\Bigl|\boldsymbol{\chi}[\xi](\boldsymbol{u}(t,x))
- \boldsymbol{\chi}[\xi](\boldsymbol{u}(t,\gamma(t)))\Bigr|
\,\mathrm{d}x\,\mathrm{d}t \notag \\
&\qquad+ \int_{\left\{\substack{t \in I_\gamma \,:\\ w(t,\gamma(t)) > \xi+\kappa^{-1}}\right\}}
\fint_{\gamma(t)-\delta}^{\gamma(t)+\delta}
\mathbf{1}_{\{w(t,x)<\xi\}}
\Bigl|\boldsymbol{\chi}[\xi](\boldsymbol{u}(t,x))
- \boldsymbol{\chi}[\xi](\boldsymbol{u}(t,\gamma(t)))\Bigr|
\,\mathrm{d}x\,\mathrm{d}t \notag \\[6pt]
&\quad\leq \int_{\left\{\substack{t \in I_\gamma \,:\\ w(t,\gamma(t)) > \xi+\kappa^{-1}}\right\}}
\fint_{\gamma(t)-\delta}^{\gamma(t)+\delta}
\mathrm{Lip}(\boldsymbol{\Theta}[\xi])
\,\bigl|\boldsymbol{u}(t,x) - \boldsymbol{u}(t,\gamma(t))\bigr|
\,\mathrm{d}x\,\mathrm{d}t \notag \\
&\qquad+ 2\sup|\boldsymbol{\chi}[\xi]|
\int_{\left\{\substack{t \in I_\gamma \,:\\ w(t,\gamma(t)) > \xi+\kappa^{-1}}\right\}}
\fint_{\gamma(t)-\delta}^{\gamma(t)+\delta}
\mathbf{1}_{\{w(t,x)<\xi\}}
\,\mathrm{d}x\,\mathrm{d}t \notag \\[6pt]
&\quad\leq \mathrm{Lip}(\boldsymbol{\Theta}[\xi])
\int_{\left\{\substack{t \in I_\gamma \,:\\ w(t,\gamma(t)) > \xi+\kappa^{-1}}\right\}}
\fint_{\gamma(t)-\delta}^{\gamma(t)+\delta}
\bigl|\boldsymbol{u}(t,x) - \boldsymbol{u}(t,\gamma(t))\bigr|
\,\mathrm{d}x\,\mathrm{d}t \notag \\
&\qquad+ 2\kappa^{-1}\sup|\boldsymbol{\chi}[\xi]|
\int_{\left\{\substack{t \in I_\gamma \,:\\ w(t,\gamma(t)) > \xi+\kappa^{-1}}\right\}}
\fint_{\gamma(t)-\delta}^{\gamma(t)+\delta}
\mathbf{1}_{\{w(t,x)<\xi\}}
\,\bigl|w(t,x)-w(t,\gamma(t))\bigr|
\,\mathrm{d}x\,\mathrm{d}t \notag \\[6pt]
&\quad\leq
\Bigl(
    \mathrm{Lip}(\boldsymbol{\Theta}[\xi])
    + 2\kappa^{-1}\sup|\boldsymbol{\chi}[\xi]|\,\mathrm{Lip}(\phi_1)
\Bigr)
\int_{I_\gamma}
\fint_{\gamma(t)-\delta}^{\gamma(t)+\delta}
\bigl|\boldsymbol{u}(t,x) - \boldsymbol{u}(t,\gamma(t))\bigr|
\,\mathrm{d}x\,\mathrm{d}t
\end{align}
For the second integral in \eqref{eq:intsplit} we estimate instead
    \begin{align}\label{eq:int2}
        & \int_{\left\{\substack{t \in I_\gamma \, : \\ w(t, \gamma(t)) < \xi-\kappa^{-1}}\right\}}   \fint_{\gamma(t)-\delta}^{\gamma(t)+\delta}  \Big| \bs \chi[\xi](\bs u(t,x)) - \bs \chi[\xi](\bs u(t, \gamma(t)))\Big| \dif x  \dif t \notag \\[6pt]
        & \quad = \int_{\left\{\substack{t \in I_\gamma \, : \\ w(t, \gamma(t)) < \xi-\kappa^{-1}}\right\}}   \fint_{\gamma(t)-\delta}^{\gamma(t)+\delta}  \mathbf 1_{\{ w(t,x) > \xi \}}(t,x) \Big| \bs \chi[\xi](\bs u(t,x)) - \bs \chi[\xi](\bs u(t, \gamma(t)))\Big| \dif x \dif t \notag \\[6pt]
        & \qquad+ \int_{\left\{\substack{t \in I_\gamma \, : \\ w(t, \gamma(t)) < \xi-\kappa^{-1}}\right\}}   \fint_{\gamma(t)-\delta}^{\gamma(t)+\delta}  \mathbf 1_{\{w(t,x) < \xi \}}(t,x) \Big| \bs \chi[\xi](\bs u(t,x)) - \bs \chi[\xi](\bs u(t, \gamma(t)))\Big| \dif x \dif t \notag \\[6pt]
        & \quad\leq 2\sup |\bs \chi[\xi] \int_{\left\{\substack{t \in I_\gamma \, : \\ w(t, \gamma(t)) < \xi-\kappa^{-1}}\right\}}   \fint_{\gamma(t)-\delta}^{\gamma(t)+\delta}  \mathbf 1_{\{ w(t,x) > \xi \}}(t,x)  \dif x \dif t \notag\\[6pt]
        & \quad \leq 2\sup |\bs \chi[\xi] \int_{\left\{\substack{t \in I_\gamma \, : \\ w(t, \gamma(t)) < \xi-\kappa^{-1}}\right\}}   \fint_{\gamma(t)-\delta}^{\gamma(t)+\delta}  \kappa^{-1} \, \mathbf 1_{\{ w(t,x) > \xi \}}(t,x) |w(t,x) - w(t, \gamma(t))|  \dif x \dif t  \notag \\[6pt]
        & \quad \leq 2 \kappa^{-1} \sup |\bs \chi[\xi]| \, \mr{Lip}(\bs \Theta[\xi](\cdot)) \, \mr{Lip} (\phi_1)  \int_{I_\gamma}   \fint_{\gamma(t)-\delta}^{\gamma(t)+\delta}  |\bs u(t,x) - \bs u(t, \gamma(t))|  \dif x \dif t 
    \end{align}
    Notice, for later, that the right hand sides of \eqref{eq:int1}, \eqref{eq:int2} approach zero as $\delta \to 0^+$ for every fixed $\kappa > 0$, due to Lemma \ref{lemma:traces}.
For the third integral in \eqref{eq:intsplit} we just estimate
\begin{equation}\label{eq:int3}
\begin{aligned}
&\int_{\left\{\substack{t \in I_\gamma \,:\\ \xi-\kappa^{-1} \leq w(t,\gamma(t)) \leq \xi+\kappa^{-1}}\right\}}
\fint_{\gamma(t)-\delta}^{\gamma(t)+\delta}
\Bigl|\boldsymbol{\chi}[\xi](\boldsymbol{u}(t,x))
- \boldsymbol{\chi}[\xi](\boldsymbol{u}(t,\gamma(t)))\Bigr|
\,\mathrm{d}x\,\mathrm{d}t \\[6pt]
&\qquad\leq
2\sup|\boldsymbol{\chi}[\xi]|\;
\mathscr{L}^1\!\Bigl(
    \bigl\{t \in I_\gamma \mid \xi-\kappa^{-1} \leq w(t,\gamma(t)) \leq \xi+\kappa^{-1}\bigr\}
\Bigr).
\end{aligned}
\end{equation}
By \eqref{eq:gammanodeg} therefore for every $\varepsilon > 0$ we can choose $\kappa >0$ big enough so that 
\begin{equation}\label{eq:int31}
     \int_{\left\{\substack{t \in I_\gamma \, : \\ \xi-\kappa^{-1} \leq w(t, \gamma(t)) \leq \xi+\kappa^{-1}}\right\}}   \fint_{\gamma(t)-\delta}^{\gamma(t)+\delta}  \Big| \bs \chi[\xi](\bs u(t,x)) - \bs \chi[\xi](\bs u(t, \gamma(t)))\Big| \dif x  \dif t \leq \varepsilon.
\end{equation}
Therefore combining \eqref{eq:int1}, \eqref{eq:int2}, \eqref{eq:int31}, by \eqref{eq:intsplit} we deduce, taking the limit as $\delta \to 0^+$,
\begin{equation}
    \begin{aligned}
        \lim_{\delta \to 0^+} \int_{I_\gamma}  \fint_{\gamma(t)-\delta}^{\gamma(t)+\delta} & \Big| \bs \chi[\xi](\bs u(t,x)) - \bs \chi[\xi](\bs u(t, \gamma(t)))\Big| \dif x \dif t  \leq 2 \sup |\bs \chi[\xi]| \varepsilon.
    \end{aligned}
\end{equation}
Since $\varepsilon > 0$ was arbitrary, the result is proved.
\end{proof}

As it can be expected, Proposition \ref{prop:tracesleb} is useful in that it can be used to compute explicitly the balance in some regions delimited by curves belonging to $\Gamma[\xi]$ or $\wt \Gamma[\xi]$ in terms of the Lebesgue values of $\bs u$ on the curves.

\begin{prop}\label{prop:balance}
   Let  $\gamma^\ell, \gamma^r$ be two curves such that for a.e. $t \in I_{\gamma^\ell}, I_{\gamma^r}$ $(t,\gamma^\ell(t))$ and $(t, \gamma^r(t))$, respectively, are Lebesgue points of $\bs u$, and that they both satisfy \eqref{eq:gammanodeg}.

    \vspace{0.3cm}
    \textbf{1.} Assume that for some $t^r > t^\ell \geq \bar t$ there holds $[\bar t, t^\ell) \subset I_{\gamma^\ell}$ and $[\bar t, t^r) \subset I_{\gamma^r}$ and that 
    \begin{equation}\label{eq:curv1}
    \gamma^\ell : [\bar t, t^\ell] \to \mathbb R, \quad \gamma^r : [\bar t, t^r] \to \mathbb R
    \end{equation}
such that 
\begin{equation}\label{eq:curv2}
\gamma_b(t) < \gamma^\ell(t) < \gamma^r(t) \quad \forall \; t \in (\bar t, t^\ell), \qquad \gamma^\ell(t^\ell) = \gamma_b(t^\ell)
    \end{equation}
$$
\gamma^r(t) > \gamma_b(t) \quad \forall \; t \in (\bar t, t^r), \qquad \gamma^r(t^r) = \gamma_b(t^r)
$$
where $\gamma_b$ is as in Definition \eqref{eq:equa}
Then there holds the balances
 \begin{equation}\label{eq:balanceprop}
          \begin{aligned}
              \int_{t^\ell}^{t^r} &    \bs \Phi^+(\xi) \sqrt{1+ \dot \gamma_b^2}\dif  t   - \int_{\gamma^\ell(\bar t)}^{\gamma^r(\bar t)} \bs \chi[\xi](\bs u(\bar t,x)) \dif x \\
              &  = \int_{\bar t}^{t^\ell}  (\bs \chi[\xi](\bs u(t, \gamma^\ell(t))), \bs \psi[\xi](\bs u(t, \gamma^\ell(t)))) \cdot (-\dot \gamma^\ell(t), 1)  \dif t\\
   & \qquad -\int_{\bar t}^{t^r}  (\bs \chi[\xi](\bs u(t, \gamma^r(t))), \bs \psi[\xi](\bs u(t, \gamma^r(t)))) \cdot (-\dot \gamma^r(t), 1)  \dif t 
          \end{aligned}
      \end{equation}
      and
       \begin{equation}\label{eq:balanceprop1}
          \begin{aligned}
              \int_{t^\ell}^{t^r} &    \wt{\bs \Phi}^+(\xi) \sqrt{1+ \dot \gamma_b^2}\dif  t   - \int_{\gamma^\ell(\bar t)}^{\gamma^r(\bar t)} \wt{\bs \chi}[\xi](\bs u(\bar t,x)) \dif x \\
              &  = \int_{\bar t}^{t^\ell}  (\wt{\bs \chi}[\xi](\bs u(t, \gamma^\ell(t))), \wt{\bs \psi}[\xi](\bs u(t, \gamma^\ell(t)))) \cdot (-\dot \gamma^\ell(t), 1)  \dif t\\
   & \qquad -\int_{\bar t}^{t^r}  (\wt{\bs \chi}[\xi](\bs u(t, \gamma^r(t))), \wt{\bs \psi}[\xi](\bs u(t, \gamma^r(t)))) \cdot (-\dot \gamma^r(t), 1)  \dif t 
          \end{aligned}
      \end{equation}

      \vspace{0.3cm}
    \textbf{2.} Assume that for some $t^f > \bar t$ there holds $[\bar t, t^f) \subset I_{\gamma^\ell} \cap I_{\gamma^r}$ and that $\gamma^\ell(t) < \gamma^r(t)$ for all $t \in [\bar t, t^f)$. Then there holds the balance
    \begin{equation}\label{eq:balancepropvar}
          \begin{aligned}
               \int_{\gamma^r( t^f)}^{\gamma^\ell(t^f)} & \bs \chi[\xi](\bs u(t^f,x)) \dif x - \int_{\gamma^\ell(\bar t)}^{\gamma^r(\bar t)} \bs \chi[\xi](\bs u(\bar t,x)) \dif x \\
              &  = \int_{\bar t}^{t^f}  (\bs \chi[\xi](\bs u(t, \gamma^\ell(t))), \bs \psi[\xi](\bs u(t, \gamma^\ell(t)))) \cdot (-\dot \gamma^\ell(t), 1)  \dif t\\
   & \qquad -\int_{\bar t}^{t^f}  (\bs \chi[\xi](\bs u(t, \gamma^r(t))), \bs \psi[\xi](\bs u(t, \gamma^r(t)))) \cdot (-\dot \gamma^r(t), 1)  \dif t 
          \end{aligned}
      \end{equation}
      and
       \begin{equation}\label{eq:balanceprop1var}
          \begin{aligned}
               \int_{\gamma^\ell( t^f)}^{\gamma^r(t^f)} & \wt{\bs \chi}[\xi](\bs u(t^f,x)) \dif x - \int_{\gamma^\ell(\bar t)}^{\gamma^r(\bar t)} \wt{\bs \chi}[\xi](\bs u(\bar t,x)) \dif x \\
              &  = \int_{\bar t}^{t^f}  (\wt{\bs \chi}[\xi](\bs u(t, \gamma^\ell(t))), \wt{\bs \psi}[\xi](\bs u(t, \gamma^\ell(t)))) \cdot (-\dot \gamma^\ell(t), 1)  \dif t\\
   & \qquad -\int_{\bar t}^{t^f}  (\wt{\bs \chi}[\xi](\bs u(t, \gamma^r(t))), \wt{\bs \psi}[\xi](\bs u(t, \gamma^r(t)))) \cdot (-\dot \gamma^r(t), 1)  \dif t 
          \end{aligned}
      \end{equation}

\end{prop}
\begin{remark}
    Notice in particular that, in point \textbf{1.}, by taking $t^\ell = \bar t$ one obtains the balance in the region between the curve $\gamma^r$ and the boundary curve $\gamma_b$.
\end{remark}
\begin{proof}
We prove \textbf{1.}, \eqref{eq:balanceprop}, the proof of \textbf{2.} being entirely similar. 
We start from the equality of Corollary \ref{coro:isehalf}, which yields for all Lipschitz $\varphi : \mathbb R^2 \to \mathbb R$:
  \begin{equation}\label{eq:normaltphi11}
   \begin{aligned}
   \iint_{{\mr H}^+}  (\bs \chi[\xi](\bs u), \bs \psi[\xi](\bs u)) \cdot \nabla \varphi \dif x \dif t  = \int_{\mathrm{Graph}\, (\gamma_b)} \varphi(t,x)   \bs \Phi(\xi) \dif \mathscr H^1(t,x)
   \end{aligned}
      \end{equation}
     Notice that by choosing a time cutoff 
      $$
      h^\varepsilon(t) : = \begin{cases}
          0 & \text{if $t \leq \bar t- \varepsilon$}\\
          \frac{1}{\varepsilon}(t-\bar t+\varepsilon) & \text{if $\bar t- \varepsilon \leq t \leq \bar t$}\\
          1 & \text{if $t > \bar t$}
      \end{cases}
      $$
      and using $\varphi h^\varepsilon$ as a test function in   \eqref{eq:normaltphi11}, we have that 
\begin{equation}\label{eq:normaltphi1}
\begin{aligned}
&\int_{\mathrm{Graph}(\gamma_b)}
\mathbf{1}_{t \geq \bar{t}}(t,x)\,\varphi(t,x)\,
\boldsymbol{\Phi}^+(\xi)
\,\mathrm{d}\mathscr{H}^1(t,x) \\[6pt]
&\lim_{\varepsilon \to 0^+}\int_{\mathrm{Graph}(\gamma_b)}
\mathbf{1}_{t \geq \bar{t}}(t,x)\,\varphi^\varepsilon(t,x)\,
\boldsymbol{\Phi}^+(\xi)
\,\mathrm{d}\mathscr{H}^1(t,x) \\[6pt]
&\quad= \lim_{\varepsilon \to 0^+}
\iint_{{\mr H}^+}
\bigl(\boldsymbol{\chi}[\xi](\boldsymbol{u}),\,\boldsymbol{\psi}[\xi](\boldsymbol{u})\bigr)
\cdot \nabla(h^\varepsilon \varphi)
\,\mathrm{d}x\,\mathrm{d}t \\[6pt]
&\quad=
\iint_{{\mr H}^+}
\mathbf{1}_{t \geq \bar{t}}(t)\,
\bigl(\boldsymbol{\chi}[\xi](\boldsymbol{u}),\,\boldsymbol{\psi}[\xi](\boldsymbol{u})\bigr)
\cdot \nabla\varphi
\,\mathrm{d}x\,\mathrm{d}t \\
&\qquad+
\lim_{\varepsilon \to 0^+} \frac{1}{\varepsilon}
\iint_{{\mr H}^+}
\mathbf{1}_{\bar{t}-\varepsilon \leq t \leq \bar{t}}\,
\boldsymbol{\chi}[\xi](\boldsymbol{u})\,\varphi
\,\mathrm{d}x\,\mathrm{d}t \\[6pt]
&\quad=
\iint_{{\mr H}^+}
\mathbf{1}_{t \geq \bar{t}}(t)\,
\bigl(\boldsymbol{\chi}[\xi](\boldsymbol{u}),\,\boldsymbol{\psi}[\xi](\boldsymbol{u})\bigr)
\cdot \nabla\varphi
\,\mathrm{d}x\,\mathrm{d}t \\
&\qquad+
\int_{\mathbb{R}}
\boldsymbol{\chi}[\xi](\boldsymbol{u}(\bar{t},x))\,\varphi(\bar{t},x)
\,\mathrm{d}x.
\end{aligned}
\end{equation}
where the last equality follows from the fact that $t \mapsto \bs \chi[\xi](\bs u(t, \cdot))$ is weakly$^\ast$ continuous from $\mathbb R^+$ to $\mathbf L^1_{loc}(\mathbb R)$, since it satisfies the equation in divergence form \eqref{eq:kin22}.

      We next choose $\varphi$ as an approximation of the characteristic function of the region 
      $$
      \begin{aligned}
      \Delta & := \Big\{(s,x) \in {\mr H}^+ \; | \; \gamma^\ell(s) \leq x \leq \gamma^r(s), \quad  s \in (t, t^\ell) \Big\} \\
      & \qquad\cup \Big\{(s, x) \in {\mr H}^+   \; | \; \gamma_b(s) \leq x \leq \gamma^r(s), s \in (t^\ell, t^r) \Big\}  
            \end{aligned}
      $$
      in the following way.
Let 
$$
\bar \gamma^\ell(t) := \begin{cases}
    \gamma^\ell(t) & \text{if $t \in (\bar t, t^\ell)$}\\
    \gamma_b(t) & \text{if $t \in (t^\ell, t^r)$}
\end{cases}
$$
be the left boundary of $\Delta$. We define for $t \in (\bar t, t^r)$:
$$
\varphi^\varepsilon(t,x) := \begin{cases}
    0 & \text{if $ x \leq \bar \gamma^\ell(t)$}\\
    \frac{1}{\varepsilon}(x-\bar \gamma^\ell(t)) & \text{if $ x \in (\bar \gamma^\ell(t), \bar \gamma^\ell(t) +\varepsilon)$} \\
    1 & \text{if $x \in (\bar \gamma^\ell(t), \gamma^r(t))$},\\
    1- \frac{1}{\varepsilon}(x-\gamma^r(t)+\varepsilon) & \text{if $x \in (\gamma^r(t)-\varepsilon, \gamma^r(t))$},\\
    0 & \text{if $x > \gamma^r(t)$}.
\end{cases}
$$
and we extended it to zero for $t \geq t^r$, while we extended it in an arbitrary Lipschitz way for $t \leq \bar t$ (say, compactly supported). 
Then we have, using \eqref{eq:tracesleb},
\begin{equation}\label{eq:normaltphi2}
\begin{aligned}
&\lim_{\varepsilon \to 0^+}
\iint_{{\mr H}^+}
\mathbf{1}_{t > \bar{t}}(t)\,
\bigl(\boldsymbol{\chi}[\xi](\boldsymbol{u}),\,\boldsymbol{\psi}[\xi](\boldsymbol{u})\bigr)
\cdot \nabla\varphi^\varepsilon
\,\mathrm{d}x\,\mathrm{d}t \\[6pt]
&\quad= \lim_{\varepsilon \to 0^+}
\Biggl(
    \frac{1}{\varepsilon}
    \int_{\bar{t}}^{t^\ell}\int_{\gamma^\ell(t)-\varepsilon}^{\gamma^\ell(t)}
    \bigl(\boldsymbol{\chi}[\xi](\boldsymbol{u}),\,\boldsymbol{\psi}[\xi](\boldsymbol{u})\bigr)
    \cdot \bigl(-\dot{\gamma}^\ell(t),\,1\bigr)
    \,\mathrm{d}x\,\mathrm{d}t \\
&\qquad\quad-
    \frac{1}{\varepsilon}
    \int_{\bar{t}}^{t^r}\int_{\gamma^r(t)}^{\gamma^r(t)+\varepsilon}
    \bigl(\boldsymbol{\chi}[\xi](\boldsymbol{u}),\,\boldsymbol{\psi}[\xi](\boldsymbol{u})\bigr)
    \cdot \bigl(-\dot{\gamma}^r(t),\,1\bigr)
    \,\mathrm{d}x\,\mathrm{d}t
\Biggr) \\[6pt]
&\quad=
\int_{\bar{t}}^{t^\ell}
\bigl(\boldsymbol{\chi}[\xi](\boldsymbol{u}(t,\gamma^\ell(t))),\,
      \boldsymbol{\psi}[\xi](\boldsymbol{u}(t,\gamma^\ell(t)))\bigr)
\cdot \bigl(-\dot{\gamma}^\ell(t),\,1\bigr)
\,\mathrm{d}t \\
&\qquad-
\int_{\bar{t}}^{t^r}
\bigl(\boldsymbol{\chi}[\xi](\boldsymbol{u}(t,\gamma^r(t))),\,
      \boldsymbol{\psi}[\xi](\boldsymbol{u}(t,\gamma^r(t)))\bigr)
\cdot \bigl(-\dot{\gamma}^r(t),\,1\bigr)
\,\mathrm{d}t.
\end{aligned}
\end{equation}
as $\varepsilon \to 0^+$, where we used Proposition \ref{prop:tracesleb} to deduce the last equality. Then we conclude using \eqref{eq:normaltphi1}, \eqref{eq:normaltphi2} that
\begin{equation}
\begin{aligned}
&\int_{t^\ell}^{t^r}
\boldsymbol{\Phi}^+(\xi)\,\sqrt{1+\dot{\gamma}_b^2}
\,\mathrm{d}t \\[6pt]
&\quad=
\int_{\mathrm{Graph}(\gamma_b)}
\mathbf{1}_{t \in  (t^\ell, t^r)}(t,x)\,\boldsymbol{\Phi}^+(\xi)
\,\mathrm{d}\mathscr{H}^1(t,x) \\[6pt]
&\quad=
\lim_{\varepsilon \to 0} \int_{\mathrm{Graph}(\gamma_b)}
\varphi^\varepsilon(t,x)\,\boldsymbol{\Phi}^+(\xi)
\,\mathrm{d}\mathscr{H}^1(t,x) \\[6pt]
&\quad= \lim_{\varepsilon \to 0^+}
\Biggl(
    \iint_{{\mr H}^+}
    \mathbf{1}_{t \geq \bar{t}}(t)\,
    \bigl(\boldsymbol{\chi}[\xi](\boldsymbol{u}),\,\boldsymbol{\psi}[\xi](\boldsymbol{u})\bigr)
    \cdot \nabla\varphi^\varepsilon
    \,\mathrm{d}x\,\mathrm{d}t
    +
    \int_{\mathbb{R}}
    \boldsymbol{\chi}[\xi](\boldsymbol{u}(\bar{t},x))\,\varphi^\varepsilon(\bar{t},x)
    \,\mathrm{d}x
\Biggr) \\[6pt]
&\quad=
\int_{\bar{t}}^{t^\ell}
\bigl(\boldsymbol{\chi}[\xi](\boldsymbol{u}(t,\gamma^\ell(t))),\,
      \boldsymbol{\psi}[\xi](\boldsymbol{u}(t,\gamma^\ell(t)))\bigr)
\cdot\bigl(-\dot{\gamma}^\ell(t),\,1\bigr)
\,\mathrm{d}t \\
&\qquad-
\int_{\bar{t}}^{t^r}
\bigl(\boldsymbol{\chi}[\xi](\boldsymbol{u}(t,\gamma^r(t))),\,
      \boldsymbol{\psi}[\xi](\boldsymbol{u}(t,\gamma^r(t)))\bigr)
\cdot\bigl(-\dot{\gamma}^r(t),\,1\bigr)
\,\mathrm{d}t \\
&\qquad+
\int_{\gamma^\ell(\bar t)}^{\gamma^r(\bar t)}
\boldsymbol{\chi}[\xi](\boldsymbol{u}(\bar{t},x))
\,\mathrm{d}x.
\end{aligned}
\end{equation}

\end{proof}

\begin{prop}\label{lemma:highcurves}
    Let $\xi > \max\{ w^\star, \overline w- \bar r\}$ and $\xi < \overline w$, satisfy 
    \begin{equation}\label{eq:xinondeg}
        \mathscr L^2\Big(\big\{ (t,x) \;  |\; w(t,x) = \xi \big\}\Big) = 0
    \end{equation}
    where we recall
    $$
w^\star := \esssup\Big\{ \xi \; | \; \bs \Phi^+(\xi) \neq 0\Big\}.
$$
    Then  $\bs \omega_\xi$-a.e. $\gamma \in \Gamma[\xi]$  satisfies $I_\gamma = \mathbb R$ and 
    $$
    \gamma(t) > \gamma_b(t) \qquad \forall \; t \in \mathbb R.
    $$
\end{prop}
\begin{proof}
    Notice that by Definition \ref{def:lagradef} for $\bs \omega_\xi$-a.e. $\gamma \in \Gamma[\xi]$, the following holds
    \begin{enumerate}
        \item $(t,\gamma(t))$ is a Lebesgue point of $\bs u$ for almost every $t$;
        \item $w(t,\gamma(t)) \neq \xi$ for almost every $t$. This second fact follows using the same argument of step \textbf{2.} of Proposition \ref{prop:lagraprop}, by choosing there $S = \big\{ (t,x) \;  |\; w(t,x) = \xi \big\}$.
    \end{enumerate}
    Consider the set of curves $\Gamma[\xi]$, of the Lagrangian representation \ref{def:lagradef} representing $\bs \chi[\xi]$ in the region $\mr{H}^+$. By contradiction, assume that there is a curve $\gamma^r \in \Gamma[\xi]$ satisfying (1), (2) such that for some $\bar t$  there holds $\gamma^r(\bar t) > \gamma_b(\bar t)$ and that there is a first $t^r >\bar t$ such that $\gamma^r(t^r) = \gamma_b(t^r)$ (if instead $t^r < \bar t$, the proof is entirely symmetric). Since for almost every $t \in I_{\gamma^r}$, $(t, \gamma^r(t))$ is a Lebesgue point of $\bs u$, and therefore of $w$, and $w(t, \gamma^r(t)) > \xi$ by Point (3) of Definition \ref{def:lagradef}, relying on point 1 of Proposition \ref{prop:localspeed}, there holds for arbitrarily small $\delta$ and possibly a slightly different $\bar t$ that 
\begin{equation}\label{eq:positivechidelta}
\int_{\gamma^r(\bar t)-\delta}^{\gamma^r(\bar t)} \bs \chi[\xi](\bs u(\bar t, x)) \dif x > 0.
\end{equation}
In particular, by (2) of Definition \ref{def:lagradef}, there exists another curve $\gamma^\ell \in \Gamma[\xi]$ with $\gamma^\ell(\bar t) < \gamma^r(\bar t)$. Moreover, since the curves are ordered (see Remark \ref{rem:oderderd}), being level sets of Lipschitz functions, $\gamma^\ell$ is forced to reach the boundary before $\gamma^r$ at some $t^\ell$:
$$
\gamma^\ell(t^\ell) = \gamma_b(t^\ell), \qquad t^\ell < t_r
$$
We can now use Proposition \ref{prop:balance} to deduce that 
      \begin{equation}\label{eq:balance1}
          \begin{aligned}
            \int_{t^\ell}^{t^r} &   \bs \Phi^+(\xi) \sqrt{1+ \dot \gamma_b^2}\dif  t   - \int_{\gamma^\ell(\bar t)}^{\gamma^r(\bar t)} \bs \chi[\xi](\bs u(\bar t,x)) \dif x \\
              &  = \int_{\bar t}^{t^\ell}  (\bs \chi[\xi](\bs u(t, \gamma^\ell(t))), \bs \psi[\xi](\bs u(t, \gamma^\ell(t)))) \cdot (-\dot \gamma^\ell(t), 1)  \dif t\\
   & \qquad -\int_{\bar t}^{t^r}  (\bs \chi[\xi](\bs u(t, \gamma^r(t))), \bs \psi[\xi](\bs u(t, \gamma^r(t)))) \cdot (-\dot \gamma^r(t), 1)  \dif t. 
          \end{aligned}
      \end{equation}
      Since 
$$
(\bs \chi[\xi](\bs u), \bs \psi[\xi](\bs u)) = \bs \chi[\xi](\bs u) ( 1, \bs \lambda_1[\xi](\bs u))
$$
is parallel to $( 1, \bs \lambda_1[\xi](\bs u))$, 
we deduce that the right hand side of \eqref{eq:balance1} is zero, because $\gamma^\ell, \gamma^r$ are characteristic curves satisfying (1) of Definition \ref{def:lagradef}:
$$
\dot \gamma^r(t) = \bs \lambda_1[\xi](\bs u(t, \gamma^r(t)), \qquad \dot \gamma^\ell(t) = \bs \lambda_1[\xi](\bs u(t, \gamma^\ell(t)) \quad \text{for a.e. $t$}.
$$
Therefore proceeding from \eqref{eq:balance1},  we deduce that 
$$
\begin{aligned}
0 =  \int_{t^\ell}^{t^r} &   \bs \Phi^+(\xi) \sqrt{1+ \dot \gamma_b^2}\dif  t   &  = \int_{\gamma^\ell(\bar t)}^{\gamma^r(\bar t)} \bs \chi[\xi](\bs u(\bar t, x)) \dif x  > 0
\end{aligned}
$$
where we recall that the first equality holds because $\xi$ is greater then the supremum of the support of $ \bs \Phi$, and the last inequality is due to \eqref{eq:positivechidelta}. This is a contradiction, therefore the claim is proved.
\end{proof}

\begin{prop}\label{prop:lowertracepre}
     Let $\xi < w^\sharp :=\min\{w^\star, \underline w + \bar r\}$, $\xi > \underline w$, satisfy \eqref{eq:xinondeg}. Then $\wt{\bs \omega}_\xi$-a.e. $\wt \gamma \in \wt \Gamma[\xi]$ satisfies $I_\gamma = \mathbb R$, and
    $$
    \wt{\gamma}(t) > \gamma_b(t) \qquad \forall \; t \in \mathbb R.
    $$
\end{prop}
\begin{proof}
  \vspace{0.5cm}
\textbf{1.} Let $\xi$ be  as in the statement. Again, as above, notice that for $\wt{\bs \omega}_\xi$-a.e. $\wt \gamma \in \Gamma[\xi]$, the following holds
    \begin{enumerate}
        \item $(t,\wt \gamma(t))$ is a Lebesgue point of $\bs u$ for almost every $t$;
        \item  $w(t,\wt \gamma(t)) < \xi$ for almost every $t$.
    \end{enumerate}
    Let $\wt \gamma^r \in \wt \Gamma[\xi]$ be a curve satisfying (1),(2). Assume that $\wt{\gamma}^r$ reaches the boundary $\gamma_b$ at some later time $t^r > \bar t$:
$$
\wt{\gamma}^r(t^r) = \gamma_b(t^r), \qquad \gamma^r(\bar t) > \gamma_b(\bar t).
$$
Notice that since the curves are ordered, all the curves with $\gamma(\bar t) < \wt \gamma^r(\bar t)$ belonging to $\wt \Gamma[\xi]$ must reach the boundary as well. Since we have 
$$
(e_{\bar t})_\sharp \wt {\bs \omega}_\xi = \wt{\bs \chi}[\xi](\bs u(\bar t, \cdot))
$$
surely there exists another curve (which we still call $\wt \gamma^r$) reaching the boundary at some other time (which we still call $t^r$) and for which $(\bar t,\wt \gamma^r(\bar t))$ is a Lebesgue point for the function $\bs u$, possibly for a slightly different $\bar t$.
 For fixed $\delta >0$ small, let then ${\wt \gamma}^\ell \in \wt \Gamma[\xi]$  with $\wt \gamma^r(\bar t)-\delta <  \wt \gamma^\ell(\bar t) < \wt \gamma^r(\bar t)$, and satisfying (1), (2) above reaching the boundary at some time $t^\ell < t^r$. Notice such a curve surely exists thanks to the fact that $(\bar t, \wt \gamma^r(\bar t))$ is a Lebesgue point of $\bs u$, and therefore of $w$. Consider the effective supremum of $w$ in $(\wt \gamma^\ell(\bar t), \wt \gamma^r(\bar t))$:
$$
\hat w_\delta := \esssup \Big\{ \xi \;  | \; \bs \chi[\xi](\bs u)_{\mid (\wt \gamma^\ell(\bar t), \wt \gamma^r(\bar t))} \neq 0 \Big\}.
$$

\vspace{0.5cm}
\textbf{2.} We claim that for every $\varepsilon > 0$, there exists $\delta$ small enough such that $\hat w_\delta \leq \xi +  \varepsilon$.

\vspace{0.5cm}
 
Assume that $\hat w_\delta > \xi$ (in the other case there is nothing to prove).
 Consider $\hat \xi \in (\hat w_\delta - 4\hat \varepsilon, \hat w_\delta)$, for 
 $$\hat \varepsilon = \min\{ (\hat w_\delta - \xi)/4, \, \bar r/4\}$$
and let $\hat \Gamma_{\hat \xi}$ be the curves of the Lagrangian representation $\hat {\bs \omega}_\xi$ of $\bs \chi[\hat  \xi]$ restricted to the region 
 $$
 \hat \Omega := \{ (t,x) \; | \; \wt \gamma^\ell(t) \leq x \leq \wt \gamma^r(t), \quad \bar t \leq t \leq t^\ell\} 
 $$
 according to Definition \ref{def:lagradef}. The existence of the representation is ensured by Proposition \ref{prop:lagraprop}, as soon as we can prove that
$$
\hat \xi \geq\esssup w_{\mid \, \hat \Omega} - \bar r.
$$
To prove it, notice that by the choice of $\hat \xi$, we have $\hat \xi  > \hat w-4\hat \varepsilon > \hat w-\bar r$, therefore it is sufficient to show that $\hat w \geq \sup w_{\mid \, \hat \Omega}$. 
If, by contradiction, $\hat w < \sup w_{\mid \hat \Omega}$, then there must be $\bar \xi > \hat w$ and a time $t > \bar t$ such that
\begin{equation}
\int_{\wt \gamma^\ell(t)}^{\wt \gamma^r(t)} \bs \chi[\bar \xi](\bs u(t, x)) \dif x > 0 
\end{equation}
However, by computing the balance in the region between the curves $\wt \gamma^\ell, \wt \gamma^r$, and between times $\bar t, t$, using point \textbf{2.} of Proposition \ref{prop:balance}, we get that 
\begin{equation}
\begin{aligned}
&\int_{\widetilde{\gamma}^\ell(t)}^{\widetilde{\gamma}^r(t)}
\boldsymbol{\chi}[\bar{\xi}](\boldsymbol{u}(t,x))
\,\mathrm{d}x \\[6pt]
&\quad=
\int_{\widetilde{\gamma}^\ell(\bar{t})}^{\widetilde{\gamma}^r(\bar{t})}
\boldsymbol{\chi}[\bar{\xi}](\boldsymbol{u}(\bar t,x))
\,\mathrm{d}x \\
&\qquad+
\int_{\bar{t}}^{t}
\bigl(\boldsymbol{\chi}[\bar{\xi}](\boldsymbol{u}(t,\widetilde{\gamma}^\ell(t))),\,
      \boldsymbol{\psi}[\bar{\xi}](\boldsymbol{u}(t,\widetilde{\gamma}^\ell(t)))\bigr)
\cdot\bigl(-\dot{\widetilde{\gamma}}^\ell(t),\,1\bigr)
\,\mathrm{d}t \\
&\qquad-
\int_{\bar{t}}^{t}
\bigl(\boldsymbol{\chi}[\bar{\xi}](\boldsymbol{u}(t,\widetilde{\gamma}^r(t))),\,
      \boldsymbol{\psi}[\bar{\xi}](\boldsymbol{u}(t,\widetilde{\gamma}^r(t)))\bigr)
\cdot\bigl(-\dot{\widetilde{\gamma}}^r(t),\,1\bigr)
\,\mathrm{d}t \\[6pt]
&\quad= 0.
\end{aligned}
\end{equation}
where the last equality follows from the facts that:
\begin{itemize}
    \item the first term in the right hand side is zero due to the fact that $\bar \xi > \hat w$
    \item the second and third terms in the right hand sides are zero (there is no flux through the boundaries) because we have, by point (2) above, being the curves $\wt \gamma^\ell, \wt \gamma^r \in \wt \Gamma_{\xi}$ and $\xi < \hat w_\delta < \bar \xi$,
$$
w(t, \wt \gamma^\ell(t)) < \xi < \bar \xi, \qquad w(t, \wt \gamma^r(t)) < \xi < \bar \xi \qquad \text{for a.e. $t \in (\bar t, t)$}
$$
    and therefore by definition $\bs \chi[\bar \xi](\bs u(t, \wt \gamma^\ell(t))) = \bs \psi[\bar \xi](\bs u(t, \wt \gamma^\ell(t)))= 0$, as well as $\bs \chi[\bar \xi](\bs u(t, \wt \gamma^r(t))) = \bs \psi[\bar \xi](\bs u(t, \wt \gamma^r(t)))= 0$.
\end{itemize}
This proves that $\bs \chi[\hat \xi]$ admits the Lagrangian representation $\bs \omega_{\hat \xi}$ in the domain $\hat \Omega$, for every $\hat \xi \in (\hat w-4\hat \varepsilon, \hat w)$.

\vspace{0.3cm}

We proceed with the proof of this step, and consider now $\hat \xi_1 \in (\hat w- \hat \varepsilon, \hat w)$ and  $\hat \xi_2 \in (\hat w - 3\hat{\varepsilon}, \hat w - 2 \hat \varepsilon)$ with the corresponding Lagrangian representations $\bs \omega_{\hat \xi_i}$ of $\bs \chi[\hat \xi_i]$ in the region $\hat \Omega$. Since the measure of degenerate $\xi$ is zero, we can require  that $\xi_1, \xi_2$ satisfy the non-degeneracy condition \eqref{eq:xinondeg}, so that $\bs \omega_{\hat \xi_i}$-almost every curve satisfies $w(t, \gamma(t)) > \hat \xi_i$ for almost every $t$. Consider then such a curve $\hat \gamma \in \hat \Gamma_{\hat \xi_1}$ and compute the balance for $\bs \chi[\hat \xi_2]$ in a region 
$$
\hat \Delta := \{(t, x) \; | \; \hat \gamma(t) \leq x \leq \wt \gamma^r(t), \qquad \bar t \leq t \leq \hat t\}.
$$
We deduce
\begin{equation}\label{eq:n1}
\begin{aligned}
      & \int_{\hat \gamma(\hat t)}^{\wt \gamma^r(\hat t)}  \bs \chi[\hat \xi_2](\bs u(\hat t,x)) \dif x    - \int_{\hat \gamma(\bar  t)}^{\wt \gamma^r(\bar t)} \bs \chi[\hat \xi_2](\bs u(\bar t,x)) \dif x \\[6pt]
      & \quad = -\int_{\bar t}^{\hat t} (\bs \chi[\hat \xi_2]( \bs u(t, \wt \gamma^r(t))), \bs \psi[\hat \xi_2]( \bs u(t, \wt \gamma^r(t)))) \cdot (-\dot {\wt \gamma}^r(t), 1) \dif t\\
      & \qquad  +\int_{\bar t}^{\hat t} (\bs \chi[\hat \xi_2]( \bs u(t, \hat \gamma(t))), \bs \psi[\hat \xi_2]( \bs u(t, \hat \gamma(t)))) \cdot (-\dot {\hat  \gamma}(t), 1) \dif t\\[6pt]
      & \quad \leq  - C (\hat \xi_1 -\hat \xi_2)(\hat t -\bar t) = -C \hat \varepsilon(\hat t- \bar t).
\end{aligned}
\end{equation}
where $C>0$ is a constant depending only on the system. Here in the last line we have used that the first term in the right hand side of the balance is zero (because one has $w(t, \wt \gamma^r(t)) < \xi < \hat w_\delta < \hat \xi_2$, for almost every $t$, the first inequality following from (2) above), while for the second one we have
$$
\begin{aligned}
     & \big(\bs \chi[\hat \xi_2]( \bs u(t, \hat \gamma(t))), \bs \psi[\hat \xi_2]( \bs u(t, \hat \gamma(t)))\big) \cdot \big(-\dot {\hat  \gamma}(t), 1\big)  \\[6pt]
     & \quad =  \bs \chi[\hat \xi_2]\;(\bs u(t, \hat \gamma(t)))\Big(1, \;\bs \lambda_1[\hat \xi_2](\bs u(t, \hat \gamma(t)))\Big) \cdot \Big(- \bs \lambda_1[\hat \xi_1](\bs u(t, \hat \gamma(t))), \; 1\Big)\\[6pt]
     & \quad = \bs \chi[\hat \xi_2](\bs u(t, \hat \gamma(t)))\Big( \bs \lambda_1[\hat \xi_2](\bs u(t, \hat \gamma(t))) -  \bs \lambda_1[\hat \xi_1](\bs u(t, \hat \gamma(t))) \Big)\\[6pt]
     & \quad < -C (\hat \xi_1 - \hat \xi_2)
\end{aligned}
$$
 where $C>0$ depends on Proposition \ref{prop:localspeed}, and therefore only on the system (here we recall that  $\hat \xi_2 \in (\esssup w_{\mid \hat \Omega}-\bar  r, \esssup w_{\mid \hat \Omega})$ and therefore there holds $\bs \chi[\xi](\bs u(t,\gamma(t)) \geq c > 0$ for a.e. $t$)
Moreover, we clearly have
\begin{equation}\label{eq:n2}
  \int_{\hat \gamma(\hat t)}^{\wt \gamma^r(\hat t)} \bs \chi[\hat \xi_2](\bs u(\hat t,x)) \dif x   - \int_{\hat \gamma(\bar t)}^{\wt \gamma^r(\bar t)} \bs \chi[\hat \xi_2](\bs u(\bar t,x))  \dif x \geq - (\sup \bs \chi) \delta.
\end{equation}
Combining \eqref{eq:n1}, \eqref{eq:n2} we deduce 
$$
\hat{\varepsilon} \leq \frac{\sup \bs \chi}{C} \frac{\delta }{\hat t - \bar t}
$$
 Since $(\hat t - \bar t)$ can be taken of the order of the distance to the boundary $d( (\bar t, \wt \gamma^r(\bar t)), \partial \Omega)$ (because the curves are Lipschitz), $\hat{\varepsilon} = |\hat w-\xi|/4$ can be made arbitrarily small for small $\delta$.

\vspace{0.5cm}
\textbf{3.} 
Next, by taking $\delta$ sufficiently small, by the previous step, we have that
\begin{equation}
\xi < \hat w_\delta < w^\sharp.
\end{equation}
Consider then any $\hat \xi > \hat w$. Computing the balance using Proposition \ref{prop:balance}, we deduce the balance equation for $\bs \chi[\hat \xi]$ in the region delimited by the curves $\wt \gamma^r, \wt \gamma^\ell$ introduced in Step \textbf{1.}:
\begin{equation}
    \begin{aligned}
     & \int_{\bar t}^{t^\ell}  (\bs \chi[\hat \xi](\bs u(t, \wt \gamma^\ell(t))), \bs \psi[\hat \xi](\bs u(t, \wt \gamma^\ell(t)))) \cdot (-\dot {\wt\gamma}^\ell(t), 1) \dif x \dif t \\
        &
    \qquad -\int_{\bar t}^{t^r}  (\bs \chi[\hat \xi](\bs u(t, \wt \gamma^r(t))), \bs \psi[\hat \xi](\bs u(t, \wt \gamma^r(t)))) \cdot (-\dot {\wt{\gamma}^r}(t), 1) \dif x \dif t \\[6pt]
    & \quad = \int_{\mathrm{Graph}\, \gamma_b} \bs \Phi(\hat \xi) \mathbf 1_{t^\ell < t < t^r} \dif \mathscr H^1(t,x) - \int_{\gamma^\ell(\bar t)}^{\wt \gamma^r(\bar t)} \bs \chi[\hat \xi ](\bs u(\bar t,x) )\dif x
    \end{aligned}
\end{equation}

 Since $\hat \xi > \hat w_\delta \geq w(t, \wt \gamma^\ell(t)), w(t, \wt \gamma^r(t))$, for a.e. $t$, the left hand side of the equality above is zero.  
 Moreover, the second term in the right hand side is zero again because $\hat \xi > \hat w$. Therefore we deduced that $\bs\Phi^+(\hat \xi) = 0$ for a.e. $\hat \xi > \hat w$, a contradiction because $$\hat w_\delta < w^\sharp \leq w^\star = \esssup \Big\{\xi \; | \; \bs \Phi^+(\xi) \neq 0 \Big\}.$$

\end{proof}

Define
$$\wt w_\star := \essinf \, \Big\{ \xi\; | \;  \wt{\bs\Phi}^+(\xi ) \neq 0\Big\}.$$  We recall that $\bs \Phi^+$, $\wt{\bs \Phi}^{+}$ are the normal traces respectively of the vector fields $X_\xi:= (\bs \chi[\xi], \bs \psi[\xi])$ and $\wt X_\xi:= (\wt{\bs \chi}[\xi], \wt{\bs \psi}[\xi])$, which satisfy 
$$
\esssup \Big\{\xi \; | \; X_\xi(t,x) \neq 0 \Big\} \leq w(t,x) \leq \essinf  \Big\{\xi \; | \; \wt X_\xi(t,x) \neq 0 \Big\} \quad \text{for a.e. $(t,x)$}.
$$
If, say, $w(t,x)$ was uniformly Lipschitz up to the boundary, we would have of course $\wt w_\star \geq w^\star$, due to the above equalities. In principle, for a general non-smooth solution, it could be that
$$
\wt w_\star < w^\star
$$
if the trace of $\bs u$ on $\gamma_b$ is not strong.
A useful consequence of the previous proposition is that at least one has $\wt w_\star > w^\sharp$, where $w^\sharp$ is as defined above, $w^\sharp =\min\{w^\star, \underline w + \bar r\}$, as it is shown in the next corollary.
\begin{coro}\label{prop:lowertrace}
    For every $\xi \in (\underline w, \overline w)$ with $\xi < w^\sharp$ satisfying \eqref{eq:xinondeg}, there holds $\wt {\bs \Phi}(\xi) = 0$. 
\end{coro}
\begin{proof}
Let $\underline w < \xi < w^\sharp$ satisfy the nondegeneracy condition \eqref{eq:xinondeg}. Then by Proposition \ref{prop:lowertracepre}  every $\wt \gamma \in \wt \Gamma[\xi]$ satisfies $I_\gamma = \mathbb R$, and
    $$
    \wt{\gamma}(t) > \gamma_b(t) \qquad \forall \; t \in \mathbb R.
    $$
We fix any $t_1 < t_2$.  We use Lemma \ref{lemma:fubinis}, choosing
$$
f(t,x) = (1, \bs \lambda_1[\xi](\bs u(t,x))) \cdot \vec n, \qquad A = \Big\{ (t,x) \; | \; \gamma_b(t) < x < \gamma_b(t)+\delta \Big\}
$$
and we compute using the Lagrangian representation $\wt {\bs \omega}_\xi$, and recalling the analogous of \eqref{eq:lambdadef} for the epigraph entropies,
\begin{equation}
    \begin{aligned}
       &  \int_{t_1}^{t_2} \frac{1}{\delta} \int_{\gamma_b(t)}^{\gamma_b(t)+\delta} (\wt{\bs \chi}[\xi](\bs u(t,x)), \wt{\bs \psi}[\xi](\bs u(t,x))) \cdot \vec n \dif x \dif t \\[6pt]
        & \quad = \frac{1}{\delta}  \int_{t_1}^{t_2}\int_{\gamma_b(t)}^{\gamma_b(t)+\delta} \wt{\bs \chi}[\xi](\bs u(t,x))(1, \bs \lambda_1[\xi](\bs u(t,x))) \cdot \vec n \dif x \dif t \\[6pt]
        & \quad= \frac{1}{\delta}\int_{\wt \Gamma[\xi]} \Bigg(\int_{t_1}^{t_2} \big(1, \,\dot {\wt \gamma}(t)\big) \cdot \vec n\; \mathbf 1_{\{\wt \gamma \; | \; \wt \gamma(t) \in (\gamma_b(t), \gamma_b(t)  + \delta)\}}({\wt \gamma})  \dif t\Bigg) \dif \wt{\bs \omega}_\xi(\wt \gamma)\\[6pt]
        & \quad= \frac{1}{\delta}\int_{\wt \Gamma[\xi]} \ms F_{\wt \gamma}\dif \wt{\bs \omega}_\xi(\wt \gamma)
    \end{aligned}
\end{equation}
where the penultimate equality holds because
$$
\bs  \lambda_1[\xi](\bs u(t, \wt \gamma(t)) = \dot {\wt \gamma}(t) \quad \text{for a.e. $t$}
$$
and we have set 
$$
\ms F_{\wt \gamma}:= \Bigg(\int_{t_1}^{t_2} \big(1, \,\dot {\wt \gamma}(t)\big) \cdot \vec n\; \mathbf 1_{\{\wt \gamma \; | \;\wt \gamma(t) \in (\gamma_b(t), \gamma_b(t)  + \delta)\}}({\wt \gamma})   \dif t\Bigg) 
$$
and $\vec n = \dot \gamma_b^\perp/|\dot \gamma_b|$. 
Next we partition the integral as follows:
\begin{equation}\label{eq:partint}
    \begin{aligned}
& \int_{\wt \Gamma[\xi]} \ms F_{\wt \gamma}\dif \wt{\bs \omega}_\xi(\wt \gamma) \\
& \quad = \int_{\Gamma^+} \ms F_{\wt \gamma}\dif \wt{\bs \omega}_\xi(\wt \gamma) + \int_{\Gamma^-} \ms F_{\wt \gamma}\dif \wt{\bs \omega}_\xi(\wt \gamma)
+ \int_{\Gamma^\star} \ms F_{\wt \gamma}\dif \wt{\bs \omega}_\xi(\wt \gamma)
 + \int_{\Gamma^m} \ms F_{\wt \gamma}\dif \wt{\bs \omega}_\xi(\wt \gamma)
    \end{aligned}
\end{equation}
where we set 
$$
I_{\wt \gamma}^{t_1, t_2}:= \Big\{t \in I_\gamma \cap(t_1,t_2) \;  |\; \wt \gamma(t) \in (\gamma_b(t) , \gamma_b(t) + \delta) \Big\}
$$
and 
$$
\Gamma^+ = \{{\wt \gamma} \in \wt \Gamma[\xi]\; | \; I_{\wt \gamma}^{t_1, t_2} \neq \emptyset, \quad \sup I_{\wt \gamma}^{t_1, t_2} = t_2, \quad \inf I_{\wt \gamma}^{t_1,t_2} \neq t_1\}
$$
$$
\Gamma^- = \{{\wt \gamma} \in \wt \Gamma[\xi]\; | \; I_{\wt \gamma}^{t_1, t_2} \neq \emptyset, \quad \sup I_{\wt \gamma}^{t_1, t_2} \neq t_2, \quad \inf I_{\wt \gamma}^{t_1,t_2} = t_1\}
$$
$$
\Gamma^\star = \{{\wt \gamma} \in \wt \Gamma[\xi] \; | \; I_{\wt \gamma}^{t_1, t_2} \neq \emptyset, \quad \sup I_{\wt \gamma}^{t_1, t_2} = t_2, \quad \inf I_{\wt \gamma}^{t_1,t_2} = t_1\}
$$
$$
\Gamma^m = \{{\wt \gamma} \in \wt \Gamma[\xi]\; | \; I_{\wt \gamma}^{t_1, t_2} \neq \emptyset, \quad \sup I_{\wt \gamma}^{t_1, t_2} \neq t_2, \quad \inf I_{\wt \gamma}^{t_1,t_2} \neq t_1\}
$$
Notice that we have that 
\begin{equation}\label{eq:gammam}
    \ms F_{\wt \gamma} = 0 \qquad \text{for all $\gamma \in \Gamma^m$}
\end{equation}
since these are the curves that enter and exit the region $(\gamma_b, \gamma_b+\delta)$. Next we show that the other curves cannot be too much: in fact we have that by (2) of Definition \ref{def:lagradef} there holds
$$
\wt{\bs \omega}_\xi(\Gamma^+) \leq \int_{\gamma_b(t_2)}^{\gamma_b(t_2) +\delta}  \wt{\bs \chi}[\xi](\bs u(t_2,x)) \dif x \, \leq\, C \delta.
$$
where $C := \sup \wt{\bs \chi}[\xi](\bs u)$
In the same way, 
$$
\wt{\bs \omega}_\xi(\Gamma^-), \wt{\bs \omega}_\xi(\Gamma^\star)\, \leq\, C \delta
$$
Moreover, we have for $\wt \gamma \in \Gamma^+$
$$
\ms F_{\wt \gamma} = \int_{I_{\wt \gamma}^{t_1,t_2}} (1, \dot{\wt \gamma}(t)) \cdot \vec n \dif t  = \wt \gamma(t_2)- \wt \gamma(\inf I_{\wt \gamma}^{t_1,t_2}) \, \leq\,\delta
$$
and entirely similar equalities hold for curves in $\Gamma^-, \Gamma^{\star}$
$$
\ms F_{\wt \gamma} \leq \delta \quad \text{for all $\wt \gamma \in \Gamma^-, \Gamma^{\star}$}.
$$
Therefore in total we have
\begin{equation}
    \int_{\wt \Gamma[\xi]} \ms F_{\wt \gamma}\dif \wt{\bs \omega}_\xi(\wt \gamma) \leq \wt{\bs \omega}_\xi(\Gamma^+ \cup \Gamma^- \cup \Gamma^{\star}) \cdot \delta \, \leq\, 3C \delta^2.
\end{equation}
We conclude that
\begin{equation}
\begin{aligned}
    \int_{t_1}^{t_2} \sqrt{1+\dot \gamma^2_b(t)}\wt{\bs \Phi}(\xi) \dif t & = \lim_{\delta \to 0^+} \int_{t_1}^{t_2} \frac{1}{\delta} \int_{\gamma_b(t)}^{\gamma_b(t)+\delta}(\wt{\bs \chi}[\xi](\bs u(t,x)), \wt{\bs \psi}[\xi](\bs u(t,x))) \cdot \vec n \dif x \dif t\\[6pt]
    & \leq  \lim_{\delta \to 0^+}  \frac{1}{\delta}3C \delta^2 = 0
    \end{aligned}
\end{equation}
   where the first equality is a direct consequence of \eqref{eq:normallimit}.
\end{proof}

We can now conclude the proof of the Liouville type theorem. 

\begin{proof}[Proof of Theorem \ref{thm:Liuv1}]
We show that the first Riemann invariant is constant and equal to $w^\star$.

\textbf{1.} We want to show  that $\esssup w \leq w^\star \leq \essinf w$, so that $w$ must be equal to the constant $w^\star$.
Assume first by contradiction that 
\begin{equation}\label{eq:supsopra}
    \overline w := \esssup_{{\mr H}^+} w > w^\star
\end{equation}
Let 
$$
\rho := \min\{ \bar r/2, (\overline w- w^\star)/2\} 
$$
and consider $\xi \in (\overline w - 2\rho, \overline w - \rho)$ satisfying the non degeneracy condition \eqref{eq:xinondeg}. By \eqref{eq:supsopra} and (2) of Definition \ref{def:lagradef}, the set $\Gamma_{\xi}$ is of positive measure for $\bs \omega_\xi$; moreover, as above, for $\bs \omega_\xi$ almost every $\gamma \in \Gamma_{\xi}$, the point $(t,\gamma(t))$ is a Lebesgue point of $\bs u$ satisfying $w(t,\gamma(t)) \neq \xi$ for almost every $t \in I_\gamma$. Consider any such curve $\gamma \in \Gamma_{\xi}$. By Proposition \ref{lemma:highcurves}, we have $I_\gamma = \mathbb R$ and $\gamma(t) > \gamma_b(t)$ for all times.

Consider also any $\hat \xi \in (\overline w -\rho, \overline w)$. Using Proposition \ref{prop:balance} on $\bs \chi[\hat \xi]$ in the domain
$$
\Delta(t_1, t_2) := \Big\{(t,x) \in {\mr H}^+ \; | \; \gamma_b(t) \leq x \leq \gamma(t), \quad t \in (t_1, t_2) \Big\}
$$
we deduce that
\begin{equation}\label{eq:t20}
\begin{aligned}
& \int_{\gamma_b(t_2)}^{\gamma(t_2)}  \bs \chi[\hat \xi](\bs u(t_2, x)) \dif x  - \int_{\gamma_b(t_1)}^{\gamma(t_1)} \bs \chi[\hat \xi](\bs u(t_1, x)) \dif x\\[6pt]
& \quad =- \int_{t_1}^{t_2} (\bs \chi[\hat \xi](\bs u(t, \gamma(t))), \bs \psi[\hat \xi](\bs u(t, \gamma(t)))) \cdot (-\dot \gamma(t), 1) \dif t \\
& \qquad + \int_{t_1}^{t_2}\bs  \Phi^+(\hat \xi) \sqrt{1+\dot \gamma_b^2}  \dif t 
\end{aligned}
\end{equation}
We further compute, recalling \eqref{prop:localspeed}, and relying on point (2) of Definition \ref{def:lagradef},
\begin{equation}
    \begin{aligned}
         & \int_{t_1}^{t_2}  (\bs \chi[\hat \xi](\bs u(t, \gamma(t)))), \bs \psi[\hat \xi](\bs u(t, \gamma(t))) \cdot (-\dot \gamma(t), 1) \dif t  \\[6pt]
         & \quad =  \int_{t_1}^{t_2} \bs \chi[\hat \xi](\bs u(t, \gamma(t)))(-\dot \gamma(t) + \bs \lambda_1[\hat \xi](\bs u(t, \gamma(t))) \dif t \\[6pt]
         & \quad= \int_{t_1}^{t_2} \bs \chi[\hat \xi](-\bs \lambda_1[\xi](\bs u(t, \gamma(t))) + \bs \lambda_1[\hat \xi](\bs u(t, \gamma(t)))) \dif t\\[6pt]
         & \quad \geq C (t_2-t_1)(\hat \xi - \xi)
    \end{aligned}
\end{equation}
where the last inequality follows by Proposition \ref{prop:localspeed}. Since the first term of \eqref{eq:t20} is positive, we conclude 
\begin{equation}
    \begin{aligned}
        -\int_{\gamma_b(t_1)}^{\gamma(t_1)} \bs \chi[\hat \xi](\bs u(t_1, x)) \dif x  & \leq -C (t_2- t_1)(\hat \xi - \xi) + \int_{t_1}^{t_2}\bs \Phi(\hat \xi)\sqrt{1+\dot \gamma_b^2}  \dif t
    \end{aligned}
\end{equation}
Since $t_2$ can be taken arbitrarily large and $\bs \Phi(\hat \xi) = 0$ by $\hat \xi > \xi > w^\star$, this is a contradiction, because the term in the left hand side is bounded and does not depend on $t_2$.

\textbf{2.} The proof of $\essinf w > w^\star$ is entirely symmetric, once we know that by Corollary \ref{prop:lowertrace} there holds 
$\wt {\bs\Phi}(\xi) = 0$ for all $\xi < w^\sharp$. In fact, it goes as follows. By contradiction, assume that
$$
\underline w := \essinf_{\Omega^+} w < w^\star
$$
and let
$$
\rho:= \min\{\bar r/2, (w^\star - \underline w )/2\}.
$$
and consider $\xi \in (\underline w , \underline w +\rho)$ satisfying \eqref{eq:xinondeg}. Again we consider a curve $\wt \gamma \in \wt \Gamma[\xi]$, satisfying the same properties as above, in particular it holds $I_{\wt \gamma} = \mathbb R$ by Proposition \ref{prop:lowertracepre}. Consider $\hat \xi \in (\underline w+ \rho, \underline w + 2\rho)$. We use Proposition \ref{prop:balance} for $\wt {\bs \chi}[\hat \xi]$ in the domain 
$$
\Delta(t_1, t_2) := \Big\{(t,x) \in \mr{H}^+ \; | \; \gamma_b(t) \leq x \leq \wt \gamma(t), \quad t \in (t_1, t_2) \Big\}
$$
and in the same way we deduce
\begin{equation}
    \begin{aligned}
        & \int_{\gamma_b(t_2)}^{\wt \gamma(t_2)} \wt{\bs \chi}[\hat \xi](\bs u(t_2,x)) \dif x -  \int_{\gamma_b(t_1)}^{\wt \gamma(t_1)} \wt{\bs \chi}[\hat \xi](\bs u(t_1,x)) \dif x \\
        & \quad \leq - C(t_2-t_1)(\hat \xi- \xi) + \int_{t_1}^{t_2}\wt{\bs \Phi}(\hat \xi)\sqrt{1+\dot \gamma_b^2} \dif t
    \end{aligned}
\end{equation}
By Corollary \ref{prop:lowertrace}, the last term in the right hand side is zero, and since $\wt {\bs \chi}[\hat \xi]\bs u((t_2, x)) \geq c > 0$ (by Proposition \ref{prop:localspeed} for $\wt \chi$, and by $\hat \xi \in (\underline w, \underline w+ \bar r)$) the first term in the left hand side is positive. Therefore we get that for all $t_2 > t_1$:
$$
-  \int_{\gamma_b(t_1)}^{\wt \gamma(t_1)} \wt{\bs \chi}[\hat \xi](\bs u(t,x)) \dif x \leq - C(t_2-t_1)(\hat \xi- \xi)
$$
which is a contradiction.
\end{proof}

\begin{proof}[Proof of Theorem \ref{thm:mainfes}]
    By Proposition \ref{prop:tangentu} and Theorem \ref{thm:Liuv1}, for almost every $\bar t$ the blow up limits $\bs u^{\bar t}_{\infty}$ are constant functions in half spaces
    $$
    \bs u^{\bar t}_{\infty} := \begin{cases}
        \bs u^-(\bar t) & \text{if $(t,x) \in \Omega_\infty^-$}\\
         \bs u^+(\bar t) & \text{if $(t,x) \in \Omega_\infty^+$}
    \end{cases}
    $$
 with $\Omega^\pm_\infty$ defined in \eqref{eq:omegainf}.
    In fact, they are uniquely determined by 
    \begin{equation}\label{eq:suptraces}
    \phi_1(\bs u^\pm(\bar t)) = \esssup \Big\{ \xi \; | \; \bs\Phi^\pm(\bar t, \gamma(\bar t), \xi ) \neq 0\Big\}
      \end{equation}
      \begin{equation}\label{eq:inftraces}
    \phi_2(\bs u^\pm(\bar t)) = \esssup \Big\{ \zeta \; | \; \bs\Psi^\pm(\bar t, \gamma(\bar t), \zeta ) \neq 0\Big\}
    \end{equation}
    where we recall that $\bs \Phi^{\pm}, \bs \Psi^\pm$ are the kinetic normal traces given by Proposition \ref{prop:weaktraces}. Notice that clearly $\bs u^{\pm}$ are measurable functions, e.g. being determined by \eqref{eq:suptraces}, \eqref{eq:inftraces}. By the compensated compactness (see e.g. \cite[Chapter 9]{Ser00}, \cite{Tza03}, or \cite{AMT25} for a summary of the results) the family $\{\bs u_r^{\bar t}\}_{r > 0}$ in \eqref{eq:blowupt} is strongly locally compact in $\mathbf L^1$.
Therefore, for almost every $\bar t$, the blow up limits in Proposition \ref{prop:tangentu} depend only on $\bar t$, and not on the particular subsequence.  By the above discussion, every subsequence converges to the same (constant) limit, and therefore we  obtain \eqref{eq:blowuptrace}: 
 \begin{equation}\label{eq:blowuptraceproof}
 \begin{aligned}
      \lim_{r \to 0^+}&  \frac{1}{r^2} \Bigg(\int_{B_r^-(\bar t, \gamma(\bar t))} |\bs u(t, x) - \bs u^-(t)|  \dif x \dif t +\int_{B_r^+(\bar t, \gamma(\bar t))} |\bs u(t, x) - \bs u^+(t)|  \dif x \dif t \Bigg) \\
      & \qquad =  \lim_{r \to 0^+} \int_{B_1} |\bs u_r^{\bar t}(s, y) - \bs u^{\bar t}_\infty|  \dif s \dif x  \\
      & +  2\|\bs u\|_{\mathbf L^\infty} \lim_{r \to 0^+} \frac{1}{r^2}\Big(\left|B_r^-(\bar t, \gamma(\bar t)) \Delta H^-_r \right| +\left|B_r^+(\bar t, \gamma(\bar t)) \Delta H^+_r \right|  \Big) = 0
       \end{aligned}
    \end{equation}
    where here 
    $$
    H^{\pm}_r := \Big\{(t,x) \in B_r(\bar t, \gamma(\bar t)) \; | \; x \lessgtr \gamma(\bar t) + (t-\bar t) \dot \gamma(\bar t) \Big\}
    $$
    and the last term in the right hand side goes to zero as $r \to 0$ due to the tangency property of $\gamma$ at the point $\bar t$.
\end{proof}

\section{From pointwise traces to \texorpdfstring{$\mathbf L^1$}{L¹} traces}\label{sec:PtoL1}
In this final section we discuss the equivalence with some other notions of trace, in particular we prove Theorems \ref{thm:main1} and \ref{thm:mainfes1}. As a first consequence of Theorem \ref{thm:main} we have the following, whose proof is quite standard once one knows Theorem \ref{thm:main}.
\begin{prop}\label{coro:main}
   Assume that the system \eqref{eq:systemi} is strictly hyperbolic and \textbf{GNL}. Let $\bs u \in \mathbf L^\infty((0, +\infty) \times \mathbb R)$ be a finite entropy solution of \eqref{eq:systemi}. Let $\gamma: (0,T) \to \mathbb R$ be a Lipschitz curve. Then  there holds
   \begin{equation}\label{eq:rtr}
       \lim_{r \to 0^+} \int_{0}^{T}\fint_{\gamma(t)}^{\gamma(t)+\delta} \big| \bs u(t, x) -\bs u^+(t)\big| \dif x \dif t  = 0
   \end{equation}
      \begin{equation}\label{eq:ltr}
       \lim_{r \to 0^+} \int_{0}^{T}\fint_{\gamma(t)-\delta}^{\gamma(t)} \big| \bs u(t, x) - \bs u^-(t)\big| \dif x \dif t  = 0
   \end{equation}
   where $\bs u^-, \bs u^+$ are the traces given by Theorem \ref{thm:main}.
\end{prop}
\begin{proof}
    We prove \eqref{eq:rtr}, the proof of \eqref{eq:ltr} being exactly the same. First, notice that by Theorem \ref{thm:mainfes} we have
    \begin{equation}\label{eq:zeropoint}
    \begin{aligned}
        & \lim_{r \to 0^+}   \int_{-1}^{1} \int_0^1 | \bs u^+(\bar t)  - \bs u(\bar t + r s, \gamma(\bar t + rs) + ry)|\dif y \dif s \\[6pt]
        & \quad=  \lim_{r \to 0^+}   \iint_{D_r} | \bs u^+(\bar t)  - \bs u(\bar t + r s, \gamma(\bar t) + rz)| \dif z \dif s  \\[6pt]
        & \quad \leq  \lim_{r \to 0^+}  \frac{1}{r^2} \iint_{B_{\kappa r}^+(\bar t, \gamma(\bar t))} | \bs u^+(\bar t)  - \bs u(t, x)| \dif y \dif s = 0 \quad \text{for a.e. $\bar t \in (t_1, t_2)$}
        \end{aligned}
    \end{equation}
   where we used \eqref{eq:blowuptrace} for the last equality.  Here in the second line we used the measure preserving change of variables 
    $$
     z = \frac{1}{r}\Big(\gamma(\bar t + rs) -\gamma(\bar t) + ry\Big)
    $$
    with $D_r$ being the corresponding domain of integration 
    $$
    D_r := \Big\{(\tau, z) \; | \; \tau = s, z = \frac{1}{r}\Big(\gamma(\bar t + rs) -\gamma(\bar t) + ry\Big) \quad \text{for some $(s, y)  \in (-1, 1)\times (0, 1)$}   \Big\}
    $$
    while the inequality follows from the fact that, since $\gamma$ is a Lipschitz curve, by choosing a constant $\kappa$ big enough, we can achieve 
    $$
    D_r \subset g(B^+_{\kappa r}) \qquad \text{for all $\bar t \in (t_1, t_2)$ and for all $r > 0$}
    $$
    where $g$ is the change of variables used in the last line of \eqref{eq:zeropoint}, mapping $(t,x)$ to 
    $$
    (s,z) = g(t,x) =  \Big(\frac{t-\bar t}{r}, \frac{x-\gamma(\bar t)}{r} \Big).
    $$
Let $\omega(\cdot)$ be the $\mathbf L^1$ modulus of continuity of $\bs u^+$ in $(0,T)$. Then we obtain, using Fubini, 
\begin{equation}\label{eq:longfubini}
    \begin{aligned}
        \int_0^{T} & \Bigg( \int_{-1}^{1} \int_0^1 | \bs u^+(\bar t)  - \bs u(\bar t + r s, \gamma(\bar t + rs) + ry)| \dif y \dif s \Bigg) \dif \bar t \\[6pt]
        &  = \int_{-1}^1 \Bigg(  \int_0^{T} \int_0^1 | \bs u^+(\bar t)  - \bs u(\bar t + r s, \gamma(\bar t + rs) + ry)|  \dif y \dif \bar t  \Bigg) \dif s \\[6pt]
        & \geq \int_{-1}^1 \Bigg(  \int_{t_1}^{t_2} \int_0^1 | \bs u^+(\bar t+ rs)  - \bs u(\bar t + r s, \gamma(\bar t + rs) + ry)|  \dif y \dif \bar t  \Bigg) \dif s -2\omega(r) \\[6pt]
        & = \int_{-1}^1 \Bigg(  \int_0^{T} \int_0^1 | \bs u^+(\bar t)  - \bs u(\bar t , \gamma(\bar t ) + ry)|  \dif y \dif \bar t  \Bigg) \dif s -2\omega(r) \\[6pt]
        & \geq \int_{-1}^1 \Bigg(  \int_0^{T} \int_0^1 | \bs u^+(\bar t)  - \bs u(\bar t , \gamma(\bar t ) + ry)|  \dif y \dif \bar t  \Bigg) \dif s -2\omega(r) - 4r\|\bs u\|_{\infty}\\[6pt]
        & = 2 \int_0^{T} \fint_0^r |\bs u^+(\bar t) - \bs u(\bar t, \gamma(\bar t) + z)| \dif z \dif \bar t -2\omega(r) - 4r\|\bs u\|_{\infty}.
    \end{aligned}
\end{equation}
Taking the limit in \eqref{eq:longfubini} we therefore obtain
\begin{equation}
    \begin{aligned}
        & \lim_{r \to 0^+} \int_0^{T}   \fint_0^r |\bs u^+(\bar t) - \bs u(\bar t, \gamma(\bar t) + z)| \dif z \dif \bar t \\[6pt]
        & \quad \leq \lim_{r \to 0^+}  \int_0^{T}  \Bigg( \int_{-1}^{1} \int_0^1 | \bs u^+(\bar t)  - \bs u(\bar t + r s, \gamma(\bar t + rs) + ry)| \dif y \dif s \Bigg) \dif \bar t \\[6pt]
        & \quad=   \int_0^{T}  \Bigg(\lim_{r \to 0^+}  \int_{-1}^{1} \int_0^1 | \bs u^+(\bar t)  - \bs u(\bar t + r s, \gamma(\bar t + rs) + ry)| \dif y \dif s \Bigg) \dif \bar t = 0
    \end{aligned}
\end{equation}
where the penultimate and last equalities follow respectively by   the dominated convergence theorem and \eqref{eq:zeropoint}, and this concludes the proof. 
    \end{proof}

We next prove Theorem \ref{thm:mainfes1}.

\begin{proof}[Proof of Theorem \ref{thm:mainfes1}]
Again, we can prove the theorem only for $\bs u^+$, since the proof for $\bs u^-$ is entirely symmetric.

\textbf{1.} In this first step, we prove the following preliminary formula for the weak kinetic normal trace $\bs \Phi^+$ of Proposition \ref{prop:weaktraces}: for almost every $t$  there holds
\begin{equation}\label{eq:yields}
    \bs \Phi^+(t, \xi) =  \Big( \bs \chi[\xi](\bs u^+(t)), \bs \psi[\xi](\bs  u^+(t)) \Big)\cdot \vec n(t)\qquad \text{for a.e. $\xi$}
\end{equation}
where we set for simplicity of notation $\bs \Phi^+(t, \xi) \equiv \bs \Phi^+(t, \gamma(t), \xi)$.
Consider any entropy-entropy flux pair of the form 
$$
\eta_\varrho(\bs u) := \int_{\underline w}^{\overline w} \varrho(\xi) \bs \chi[\xi](\bs u) \dif \xi, \quad q_\varrho(\bs u) := \int_{\underline w}^{\overline w} \varrho(\xi) \bs \psi[\xi](\bs u) \dif \xi.
$$
From \eqref{eq:rtr} we readily deduce the existence of a subsequence $\{\delta_n\}$ such that 
\begin{equation}\label{eq:deltanstrong}
\bs u(\cdot, \gamma(\cdot) + \delta_n) \longrightarrow \bs u^+(\cdot) \qquad \text{in $\mathbf L^1(0,T)$ as $n \to +\infty$.}
\end{equation}
Since by the weak normal trace properties of Section \ref{sec:weaktrace}, in particular by Proposition \ref{prop:normallimit}, we have that for almost every $\xi$ there holds 
$$
\bs \Phi^+(\cdot, \xi) = \text{weak}^\ast - \lim_{\delta \to 0^+} \Big( \bs \chi[\xi](\bs u(\cdot, \gamma(\cdot) + \delta), \bs \psi[\xi](\bs u(\cdot, \gamma(\cdot) + \delta)) \Big)\cdot \vec n(\cdot) \quad \text{in $\mathbf L^\infty(0,T)$}
$$
multiplying by $\varrho(\xi)$ and integrating in $\xi$ we obtain 
\begin{equation}
\begin{aligned}
    & \int \varrho(\xi) \bs \Phi^+(\cdot, \xi) \dif \xi  \\
    & \quad = \text{weak}^\ast - \lim_{\delta \to 0^+}  \int \varrho(\xi) \Big( \bs \chi[\xi](\bs u(\cdot, \gamma(\cdot) + \delta), \bs \psi[\xi](\bs u(\cdot, \gamma(\cdot) + \delta)) \Big)\cdot \vec n(\cdot) \dif \xi \\[6pt]
    & \quad = \text{weak}^\ast - \lim_{\delta \to 0^+} (\eta_\varrho(\bs u(\cdot, \gamma(\cdot) + \delta), q_{\varrho}(\bs u(\cdot, \gamma(\cdot) + \delta)) \cdot \vec n(\cdot)
    \end{aligned}
\end{equation}
Computing the limit along the sequence $\delta_n \to 0^+$ allows to pass the limit inside the nonlinear functions by \eqref{eq:deltanstrong}, and therefore we obtain 
\begin{equation}\label{eq:strongl}
\begin{aligned}
    & \int_{\underline w}^{\overline w} \varrho(\xi) \bs \Phi^+(\cdot, \xi) \dif \xi  \\
    & \quad =  \text{weak}^\ast - \lim_{n \to +\infty} (\eta_\varrho(\bs u(\cdot, \gamma(\cdot) + \delta_n)), q_{\varrho}(\bs u(\cdot, \gamma(\cdot) + \delta_n)) \cdot \vec n(\cdot)\\[6pt]
    & \quad = (\eta_\varrho(\bs u^+(\cdot)), q_{\varrho}(\bs u^+(\cdot)) \cdot \vec n(\cdot)\\
    & = \int_{\underline w}^{\overline w} \varrho(\xi) \Big(\bs \chi[\xi](\bs u^+(\cdot)), \bs \psi[\xi](\bs u^-(\cdot)) \Big) \cdot \vec n(\cdot) \dif \xi.
    \end{aligned}
\end{equation}
Since this holds for all $\varrho$, in turn this yields \eqref{eq:yields}.

\vspace{0.3cm}
\textbf{2.} Take now any other subsequence $\wt \delta_n$ for which 
$$
\bs u(\cdot, \gamma(\cdot) + \wt \delta_n) \longrightarrow \alpha_{(\cdot)} \in \mathscr P(\mc U) \qquad \text{in the sense of Young measures.}
$$
Computing the weak$^\ast$ limit above along this subsequence $\wt \delta_n \to 0^+$ we obtain in the same way
\begin{equation}\label{eq:weakl}
    \begin{aligned}
         \int \varrho(\xi) \bs \Phi^+(\cdot, \xi) \dif \xi & = \text{weak}^\ast - \lim_{n \to +\infty} (\eta_\varrho(\bs u(\cdot, \gamma(\cdot) + \wt\delta_n), q_{\varrho}(\bs u(\cdot, \gamma(\cdot) + \wt \delta_n)) \cdot \vec n(\cdot) \\
         & = \int_{\mc U}  \Big( \eta_{\varrho}(\bs q),  q_\varrho(\bs q) \Big)\cdot \vec n(\cdot) \dif \alpha_{(\cdot)}(\bs q).
    \end{aligned}
\end{equation}
Therefore from \eqref{eq:strongl}, \eqref{eq:weakl}, and the definition of $\eta_{\varrho}, q_{\varrho}$, we deduce that for every $\varrho \in C^0$ there holds for almost every $t$:
\begin{equation}
    \int_{\underline w}^{\overline w} \varrho(\xi) (\bs \chi[\xi](\bs u^+(t)), \bs \psi[\xi](\bs u^+(t)))\cdot \vec n(t) \dif \xi = \int_{\underline w}^{\overline w} \varrho(\xi) \int_{\mc U}  \Big( \bs \chi[\xi](\bs q), \bs \psi[\xi](\bs q) \Big)\cdot \vec n(t) \dif \alpha_{t}(\bs q).
\end{equation}
We then deduce that for a.e. $(t, \xi) \in (t_1, t_2) \times (\underline w, \overline w)$:
\begin{equation}\label{eq:deltarelh}
     \Big( \bs \chi[\xi](\bs u^+(t)), \bs \psi[\xi](\bs  u^+(t)) \Big)\cdot \vec n(t) = \int_{\mc U}  \Big( \bs \chi[\xi](\bs q), \bs \psi[\xi](\bs q) \Big)\cdot \vec n(t) \dif \alpha_t(\bs q). 
\end{equation}
Clearly, the same relation holds, with the same proof, for the epigraph entropies $\wt{\bs \chi}[\xi]$ for a.e. $(t, \xi) \in (t_1, t_2) \times (\underline w, \overline w)$
\begin{equation}\label{eq:deltarele}
     \Big( \wt{\bs \chi}[\xi](\bs u^+(t)), \wt{\bs \psi}[\xi](\bs  u^+(t)) \Big)\cdot \vec n(t) = \int_{\mc U}  \Big( \wt{\bs \chi}[\xi](\bs q), \wt{\bs \psi}[\xi](\bs q) \Big)\cdot \vec n(t) \dif \alpha_t(\bs q). 
\end{equation}
Moreover, the same holds for the entropies related to the second Riemann invariant, $\bs \upsilon[\zeta], \bs \varphi[\zeta]$.

\vspace{0.3cm}
\textbf{3.} Next,  by Step \textbf{1.}, for almost every $t \in (t_1, t_2)$, $\alpha_t$ is a probability measure satisfying the assumptions of Lemma \ref{lemma:alphaisdelta} (which is proved after the end of this proof) with $\bar {\bs u}  = \bs u^+(t)$. Therefore by Lemma \ref{lemma:alphaisdelta} we deduce 
$$
\alpha_t = \delta_{\bs u^+(t)} \qquad \text{for a.e. $t \in (t_1, t_2)$}.
$$
But this readily implies that there holds the strong convergence 
$$
\bs u(\cdot, \gamma(\cdot) + \wt \delta_n) \longrightarrow \bs u^+(\cdot) \qquad \text{strongly in $\mathbf L^1(t_1, t_2)$}.
$$
Since this holds for every subsequence $\wt \delta_n\to 0^+$, this concludes the proof of Theorem \ref{thm:mainfes1}.   
\end{proof}

We conclude by stating and proving the Lemma used in Step \textbf{2.} of the proof of Theorem \ref{thm:mainfes1}.

\begin{lemma}\label{lemma:alphaisdelta}
     Assume that the system \eqref{eq:systemi} is \textbf{GNL} and that it satisfies \eqref{eq:globhyp}. Let $\alpha \in \mathscr P(\mc U)$ be a probability measure, $\vec n \in \mathbb S^1$ of the form $\vec n = (\vec n_t, \vec n_x)$ with $\vec n_x \neq  0$,
    and $\bar {\bs u}\in \mc U$ a fixed point. Assume that for almost every
    $\xi, \zeta$ there holds that 
    \begin{equation}\label{eq:alphaisdeltawH}
         \Big( \bs \chi[\xi](\bar{\bs u}), \bs \psi[\xi](\bar{\bs u})\Big)\cdot \vec n = \int_{\mc U}  \Big( \bs \chi[\xi](\bs q), \bs \psi[\xi](\bs q) \Big)\cdot \vec n \dif \alpha(\bs q).
    \end{equation}
    \begin{equation}\label{eq:alphaisdeltawH1}
     \Big( \wt{\bs \chi}[\xi](\bar{\bs u}), \wt{\bs \psi}[\xi](\bar{\bs u}) \Big)\cdot \vec n = \int_{\mc U}  \Big( \wt{\bs \chi}[\xi](\bs q), \wt{\bs \psi}[\xi](\bs q) \Big)\cdot \vec n \dif \alpha(\bs q). 
\end{equation}
 \begin{equation}\label{eq:alphaisdeltazH}
         \Big( \bs \upsilon[\zeta](\bar{\bs u}), \bs \varphi[\zeta](\bar{\bs u})\Big)\cdot \vec n = \int_{\mc U}  \Big( \bs \upsilon[\zeta](\bs q), \bs \varphi[\zeta](\bs q) \Big)\cdot \vec n \dif \alpha(\bs q).
    \end{equation}
    \begin{equation}\label{eq:alphaisdeltazH1}
     \Big( \wt{\bs \upsilon}[\zeta](\bar{\bs u}), \wt{\bs \varphi}[\zeta](\bar{\bs u}) \Big)\cdot \vec n = \int_{\mc U}  \Big( \wt{\bs \upsilon}[\zeta](\bs q), \wt{\bs \varphi}[\zeta](\bs q) \Big)\cdot \vec n \dif \alpha(\bs q). 
\end{equation}
    Then $\alpha = \delta_{\bar u}$.
\end{lemma}

\begin{remark}
    For our purposes we only need the version of this Lemma where $\vec n = (\vec n_t, \vec n_x)$ with $\vec n_x \neq 0$, since the version with $\vec n_x =0$ corresponds to traces on an hyperplane of the form $\{t = \bar t\}$. However the version with $\vec n_x = 0$ is actually way easier and can be found for example in \cite{Tal26}, where the Lemma is proved in this particular case without the assumption \ref{eq:globhyp}.
\end{remark}

\begin{proof}[Proof of Lemma \ref{lemma:alphaisdelta}]
\textbf{1.} Let $\bar {\bs u} = (\bar w, \bar z) \in \mc U$ be the Riemann coordinates of the point $\bar {\bs u}$. Notice that, since by definition of $\bs \chi, \bs \psi$, one has that 
    $$
    \bs \chi[\xi](\bar w, \bar z) = 0, \quad \bs \psi[\xi](\bar w,\bar z) = 0 \qquad \text{for all}\; \xi > \bar w
    $$
and 
$$
\wt{\bs \chi}[\xi](\bar w, \bar z) = 0, \quad \wt{\bs \psi}[\xi](\bar w,\bar z) = 0 \qquad \text{for all}\; \xi < \bar w
$$
equations \eqref{eq:alphaisdeltawH}, \eqref{eq:alphaisdeltawH1} actually imply that 
 \begin{equation}\label{eq:alphadw}
         \int_{\mc U}  \Big( \bs \chi[\xi](\bs q), \bs \psi[\xi](\bs q) \Big)\cdot \vec n \dif \alpha(\bs q) = 0 \qquad \text{for all}\; \xi > \bar w
    \end{equation}
    \begin{equation}\label{eq:alphadw1}
      \int_{\mc U}  \Big( \wt{\bs \chi}[\xi](\bs q), \wt{\bs \psi}[\xi](\bs q) \Big)\cdot \vec n \dif \alpha(\bs q) = 0 \qquad \text{for all}\; \xi < \bar w .
\end{equation}
In an entirely similar way, one deduces that 
\begin{equation}\label{eq:alphadz}
         \int_{\mc U}  \Big( \bs \upsilon[\zeta](\bs q), \bs \varphi[\zeta](\bs q) \Big)\cdot \vec n \dif \alpha(\bs q) = 0 \qquad \text{for all}\; \zeta > \bar z
    \end{equation}
    \begin{equation}\label{eq:alphadz1}
      \int_{\mc U}  \Big( \wt{\bs \upsilon}[\zeta](\bs q), \wt{\bs \varphi}[\zeta](\bs q) \Big)\cdot \vec n \dif \alpha(\bs q) = 0 \qquad \text{for all}\; \zeta < \bar z .
\end{equation}

\vspace{0.5cm}
\textbf{2.} We notice that, due to assumption \eqref{eq:globhyp}, only one of the quantities 
$$
(1, \lambda_1) \cdot \vec n, \qquad (1, \lambda_2) \cdot \vec n
$$
can change sign in $\mc U$: this is clear because, if both change sign, then there are $\bs q_1, \bs q_2 \in \mc U$ for which
$$
0 = (1, \lambda_1(\bs q_1)) \cdot \vec n = (1, \lambda_2(\bs q_2)) \cdot \vec n
$$
implying $\lambda_1(\bs q_1)) = \lambda_2(\bs q_2)$, which is excluded by \eqref{eq:globhyp}.
Without loss of generality, we assume, up to a symmetry, that $(1, \lambda_1)\cdot \vec n$ can change sign, and therefore $(1, \lambda_2)\cdot \vec n$ has constant positive sign:
\begin{equation}\label{eq:lambda2pos}
(1, \lambda_2)\cdot \vec n > C > 0 \qquad \text{in $\mc U$}.
\end{equation}

Consider now 
$$
z^u := \sup \big\{ z \; | \; \exists \; w \; \text{such that} \; (w, z) \in \mr{supp} \, \alpha \big\}
$$
$$
z^\ell := \inf \big\{ z \; | \; \exists \; w \; \text{such that} \; (w, z) \in \mr{supp} \, \alpha \big\}
$$
We claim that $z^u = z^\ell = \bar z$, where we recall $\bar{\bs u} = (\bar w, \bar z)$. We start by recalling that, by \eqref{eq:coincideonline}, 
$$
\bs \varphi[\zeta](w, \zeta) = \lambda_2(w, \zeta) \bs \upsilon[\zeta](w, \zeta), \qquad \text{for all $w, \zeta$}  
$$
and therefore by smoothness of $\bs \varphi[\zeta](w,z), \bs \upsilon[\zeta](w,z)$ in the open region $\{(w, z, \zeta) \; | \; \zeta < z\}$ we deduce that for all $\zeta < z$ there holds:
\begin{equation}
    \Big( \bs \upsilon[\zeta](w, z), \bs \varphi[\zeta](w, z)\Big) \cdot \vec n = \bs \upsilon[\zeta](w, \zeta) (1, \lambda_2(w, \zeta)) \cdot \vec n + \mc O(1) \cdot |\zeta - z|
\end{equation}
We also recall that by \eqref{eq:chiupsline}
  $$\bs v[\zeta](w, \zeta)  = h(w, \zeta) > c > 0 \qquad \text{for all $w, \zeta$}$$ for some positive constant $c$, where $h$ is defined in Section \ref{sec:entropies}, therefore using also \eqref{eq:lambda2pos} we deduce that there exists a $\bar \delta > 0$ sufficiently small such that
\begin{equation}
     \Big( \bs \upsilon[\zeta](w, z), \bs \varphi[\zeta](w, z)\Big) \cdot \vec n > \wt c \qquad \text{for all $z-\bar \delta <   \zeta < z$}
\end{equation}
for some other positive constant $\wt c$.
Therefore, if by contradiction we have $z^u > \bar z$, we deduce from \eqref{eq:alphadz}, by taking 
$$
\zeta = z^u - \min\Big\{ \bar \delta, |z^u-\bar z|/2 \Big\}
$$
that 
$$
       0 =   \int_{\mc U}  \Big( \bs \upsilon[\zeta](\bs q), \bs \varphi[\zeta](\bs q) \Big)\cdot \vec n \dif \alpha(\bs q) \geq \wt c \int_{\mc U} \mathbf 1_{\{ z \in (\zeta, z^u]\}}(w, z) \dif \alpha(w, z)
$$
which is a contradiction because by definition of $z^u$ we must have instead 
$$
\int_{\mc U} \mathbf 1_{\{ z \in (\zeta, z^u]\}}(w, z) \dif \alpha(w, z) > 0.
$$
This shows $z^u = \bar z$; an entirely symmetric argument using \eqref{eq:alphadz1}  proves that $z^\ell = \bar z$ as well.

\textbf{3.} At this point, by Step \textbf{2.} we know that 
$$
\mr{supp} \, \alpha \subset \big\{(w, z) \; | \; z = \bar z\big\}
$$
Since the system is genuinely nonlinear, there exists at most one point $w_0$ such that 
$$
(1, \lambda_1(w_0, \bar z)) \cdot \vec n = 0
$$
because 
$$
\frac{\dif}{\dif w} (1, \lambda_1(w_0, \bar z)) \cdot \vec n = \vec n_x \cdot \frac{\dif}{\dif w}  \lambda_1(w_0, \bar z) > 0.
$$
Assume first that $w_0 < \bar w$. In the same way as in the previous step, we have that 
there exists a $\bar \delta > 0$ sufficiently small such that
\begin{equation}
     \Big( \bs \chi[\xi](w, z), \bs \psi[\xi](w, z)\Big) \cdot \vec n > \wt c \qquad \text{for all $w-\bar \delta <   \xi < w$}.
\end{equation}
Let 
$$
w^u := \sup \big\{ w \; | \; (w, \bar z) \in \mr{supp} \, \alpha \big\}.
$$
We first show that $w^u \leq \bar w$. By contradiction, if $w^u > \bar w$, we obtain as above by taking 
$$
\xi = w^u  - \min\Big\{\bar \delta, |w^u - \bar w|/2 \Big\}
$$
and using \eqref{eq:alphadw}, that 
$$
 0 =  \int_{\mc U}  \Big( \bs \chi[\xi](\bs q), \bs \psi[\xi](\bs q) \Big)\cdot \vec n \dif \alpha(\bs q) \geq \wt c \int_{\mc U} \mathbf 1_{\{w \in (\xi, w^u]\}}(w, z) \dif \alpha(w, z)
$$
which is a contradiction with the definition of $w^u$. Therefore we obtained $w^u \leq \bar w$. Now it is sufficient to use that $\alpha$ is a probability measure to conclude. In fact, consider now $\xi < \bar w$ such that \eqref{eq:alphaisdeltawH} holds, and we deduce
$$
         \Big( \bs \chi[\xi](\bar{\bs u}), \bs \psi[\xi](\bar{\bs u})\Big)\cdot \vec n = \int_{\mc U}  \Big( \bs \chi[\xi](\bs q), \bs \psi[\xi](\bs q) \Big)\cdot \vec n \dif \alpha(\bs q)
$$
By taking the limit as $\xi \to \bar w^-$, since the support of $\alpha$ is contained in the set $\{(w,z) \; | \; z = \bar z, \quad w \leq w^u\}$,  we obtain
$$
\begin{aligned}
 \Big( \bs \chi[\bar w](\bar{\bs u}), \bs \psi[\bar w](\bar{\bs u})\Big)\cdot \vec n & = \lim_{\xi \to \bar w^-} \int_{\mc U}  \Big( \bs \chi[\xi](\bs q), \bs \psi[\xi](\bs q) \Big)\cdot \vec n \dif \alpha(\bs q)\\
 & = \alpha(\{(\bar w, \bar z)\})\Big( \bs \chi[\bar w](\bar{\bs u}), \bs \psi[\bar w](\bar{\bs u})\Big) \cdot \vec n
\end{aligned}
$$
Since we assumed $w_0 < \bar w$, there holds 
$$
\Big( \bs \chi[\bar w](\bar{\bs u}), \bs \psi[\bar w](\bar{\bs u})\Big) > 0
$$
and therefore $\alpha(\{(\bar w, \bar z)\}) = 1$. Since $\alpha$ is a probability measure, this implies that $\alpha$ is concentrated in the point $\bar u = (\bar w, \bar z)$. 

The other case $w_0 > \bar w$ is proved symmetrically, using \eqref{eq:alphaisdeltawH}, and this concludes the proof of the Lemma.
\end{proof}

\appendix

\section{}

\begin{prop}\label{prop:lebp}
    Let $\Omega \subset \mathbb R^d$ and $\mu \in \mathscr{M}(\Omega)$  be a positive locally finite Borel measure, and let $\bs \Phi: \Omega \to \mathbf L^\infty(\mathbb R)$, $y \mapsto \bs \Phi(y, \cdot)$ be a weakly measurable map. Then for $\mu$-almost every $x$ there holds for all $p \in C^0_c(\mathbb R^d\times \mathbb R; \, \mathbb R)$:
    \begin{equation}
        \begin{aligned}
            \lim_{r \to 0^+} \frac{1}{\mu(B_r(x))}  \Bigg\{ & \int_{B_r(x)} p\left(\frac{y-x}{r}, \xi\right) \Big(\bs \Phi(y, \xi) - \bs \Phi(x, \xi)\Big) \dif \xi \dif \mu(y)\Bigg\} = 0.
        \end{aligned}
    \end{equation}
\end{prop}
\begin{proof}Let $\mathscr G \subset C^0_c(\mathbb R)$ be a dense and countable subset in the uniform topology. 
    Since finite sums $\sum_i \varphi_i(y) \varrho_i(\xi)$ with $\varrho_i \in \mathscr G$ are dense in $C^0_c(\mathbb R^d \times \mathbb R; \, \mathbb R)$ in the uniform topology, it is clearly sufficient to prove the result for functions $p \in C^0_c$ that factorize as $p(y, \xi) = \varphi(y) \varrho(\xi)$, with $\varrho \in \mathscr G$. In order to do this, for every $\varrho \in \mathscr G$, define a function 
$$
F_\varrho(y) := \int \varrho(\xi) \bs \Phi(y, \xi) \dif \xi
$$
and let 
$$
G_{\varrho} := \Big\{x \in \Omega \; | \; \text{$x$ is a $\mu$-Lebesgue point for $F_{\varrho}$} \Big\}, \qquad G := \bigcap_{\varrho \in \mathscr G} G_{\varrho}
$$
and notice that there holds $\mu(\Omega \setminus G) = 0$. Then there holds for every $\varphi \in C^0_c(\mathbb R^d)$ and $x \in G$ that
$$
\begin{aligned}
 \lim_{r \to 0^+} & \frac{1}{\mu(B_r(x))} \Bigg|\int_{B_r(x)} \varphi\left(\frac{y-x}{r} \right) (F_\varrho (y) - F_\varrho(x))\dif \mu(y)\Bigg|  \\[6pt]
 & \leq \|\varphi\|_{C^0}\lim_{r \to 0^+}  \frac{1}{\mu(B_r(x))} \int_{B_r(x)}|F_\varrho (y) - F_\varrho(x)|\dif \mu(y) = 0
\end{aligned}
$$
which is what we wanted to prove.

\end{proof}

\section{}

The following is a standard geometric measure theory result valid in much more generality, but for simplicity we state it and prove it only our setting.
\begin{prop}\label{prop:besi}
    Let $\gamma: \mathbb R^+ \to \mathbb R$ be a Lipschitz curve and $\bs \nu \in \mathscr M(\mathbb R^+\times \mathbb R)$ be a positive locally finite Borel measure. Then for almost every $\bar t > 0$ there holds
    \begin{equation}
        \lim_{r \to 0} \,  \frac{1}{r}\bs \nu\big(B_r^{\pm}(\bar t, \gamma(\bar t))\big) = 0
    \end{equation}
    where 
    $$
    B_r^{+}(\bar t, \gamma(\bar t)) := \{(t,x) \in B_r(t, \gamma(\bar t)) \; | \; x >\gamma(t)\}
    $$
$$
 B_r^{-}(\bar t, \gamma(\bar t)) := \{(t,x) \in B_r(t, \gamma(\bar t)) \; | \; x < \gamma(t)\} 
$$
\end{prop}
\begin{proof}
Let 
$$\mu := \mathscr H^1 \llcorner  \mr{Graph} \, \gamma, \qquad \wt{\bs \nu}:= \bs \nu -\bs \nu \llcorner \mr{Graph} \, \gamma
    $$

    By the Besicovitch differentiation theorem (see for example \cite[Theorem 2.22]{AFP00}) we have
    $$
    \wt{\bs \nu} = a \mu + \wt{\bs \nu}^s
    $$
    where $a$ is defined $\mu$-a.e. by 
    $$
    a(t,x) := \lim_{r \to 0^+} \frac{\wt{\bs \nu}(B_r((t,x)))}{\mu(B_r((t,x)))} \qquad \text{for $(t,x) \in \mr{supp}\, \mu$}
    $$
    and $\wt{\bs \nu}^s$ is a measure concentrated on a $\mu$-negligible set. Since $\wt{\bs \nu}(\mr{Graph}\, (\gamma))=0$,  we must have $a(t,x) = 0$ for $\mu$-a.e. $(t, x)$. Moreover for some constant $C > 0$ depending only on the Lipschitz constant of $\gamma$ there holds
    $$
    \mu(B_r(\bar t, \gamma(\bar t))) \leq C r
    $$
   so that we deduce that for almost every $\bar t$ 
$$
\limsup_{r \to 0^+}  \frac{\wt{\bs \nu}\big(B_r(\bar  t, \gamma(\bar t) ) \big)}{r} \leq C^{-1} \limsup_{r \to 0^+} \;   \frac{\wt{\bs \nu}\big(B_r(\bar  t, \gamma(\bar t) ) \big)}{\mu(B_r(\bar t, \gamma(\bar t))} = 0
$$
    which is what we wanted to prove, because clearly for all $r > 0$:
$$
\bs \nu (B^{\pm}_r(\bar t, \gamma(\bar t))) \leq \wt{\bs \nu}(B_r(\bar t, \gamma(\bar t))).
$$
\end{proof}

\vspace{1cm}

\noindent {\bf Acknowledgements.} The author is partially supported by the GNAMPA--INDAM Project 2026, \say{Problemi di controllo e regolarità per evoluzione di insiemi}. He would also like to thank Prof.~Fabio Ancona for his careful reading of a preliminary version of the manuscript.

\vspace{0.5cm}
\noindent {\bf Declaration of interests.}  
The author reports there are no competing interests to declare.

\vspace{0.5cm}
\noindent {\bf Data availability statement.}  
Data sharing is not applicable to this article as no datasets were generated or analyzed during the current study.

\end{document}